\documentclass[11pt,reqno]{amsart}
\usepackage{amssymb,mathrsfs,color,mathtools,empheq, verbatim, epstopdf}
\usepackage{pinlabel}
\mathtoolsset{showonlyrefs}
\usepackage{hyperref} 

\usepackage{enumerate, bbm}
\usepackage{tensor}

\usepackage{graphicx}
\usepackage{tensor}
\usepackage{slashed}
\usepackage{cite}
\usepackage[dvipsnames]{xcolor}

\usepackage[shortlabels]{enumitem}

\usepackage{geometry}

\numberwithin{equation}{section}

\newtheorem{mainthm}{Theorem}
\newtheorem{thm}{Theorem}[section]
\newtheorem{cor}[thm]{Corollary}
\newtheorem{lem}[thm]{Lemma}

\theoremstyle{definition} 
\newtheorem{rem}[thm]{Remark}
\newtheorem{defn}[thm]{Definition}

\theoremstyle{remark}

\def\bR {\mathbb{R}}
\def\bS {\mathbb{S}}

\def\frakh{\mathfrak{h}}

\newcommand{\bs}[1]{\boldsymbol{#1}}

\newcommand{\dist}{\operatorname{dist}}

\definecolor{deepgreen}{cmyk}{1,0,1,0.5}

\newcommand{\E}{\mathcal{E}}

\newcommand{\N}{\mathbb{N}}
\newcommand{\R}{\mathbb{R}}
\newcommand{\Sp}{\mathbb{S}}

\newcommand{\al}{\alpha}
\newcommand{\be}{\beta}

\newcommand{\de}{\delta}

\newcommand{\om}{\omega}
\newcommand{\lam}{\lambda}
\newcommand{\te}{\theta}

\newcommand{\De}{\Delta}
\newcommand{\Om}{\Omega}

\newcommand{\p}{\partial}
\newcommand{\na}{\nabla}

\newcommand{\loc}{\operatorname{loc}}

\makeatletter

\newcommand{\Rmnum}[1]{\expandafter\@slowromancap\romannumeral #1@}
\makeatother

\newcommand{\ti}{\widetilde}

\newcommand{\abs}[1]{\left\lvert{#1}\right\rvert}

\newcommand{\EQ}[1]{\begin{equation}\begin{split} #1 \end{split}\end{equation}}

\newcommand{\Del}[1]{}

\newcommand{\mand}{{\ \ \text{and} \ \  }}

\newcommand{\mas}{{\ \ \text{as} \ \ }}

\definecolor{green}{rgb}{0.3,0.8,0} % Redefines the color green.
\newcommand{\ud}{\mathrm{d}}

\newcommand{\eps}{\epsilon}

\newcommand{\bfd}{{\bf d}}

\newcommand{\bbS}{\mathbb S}

\newcommand{\calB}{\mathcal B}
\newcommand{\calC}{\mathcal C}
\newcommand{\calD}{\mathcal D}
\newcommand{\calE}{\mathcal E}

\newcommand{\calI}{\mathcal I}
\newcommand{\calJ}{\mathcal J}
\newcommand{\calK}{\mathcal K}
\newcommand{\calL}{\mathcal L}
\newcommand{\calM}{\mathcal M}
\newcommand{\calN}{\mathcal N}

\newcommand{\calP}{\mathcal P}
\newcommand{\calQ}{\mathcal Q}
\newcommand{\calR}{\mathcal R}

\newcommand{\calT}{\mathcal T}

\newcommand{\calV}{\mathcal V}

\newcommand{\les}{\lesssim}

\def\loc{\mathrm{loc}}

\begin{document}

\title[Bubbling for the harmonic map heat flow]{Continuous in time bubble decomposition\\ for the harmonic map heat flow}
\author{Jacek Jendrej}
\author{Andrew Lawrie}
\author{Wilhelm Schlag}

\begin{abstract}
We consider the harmonic map heat flow for maps $\bR^{2} \to \bS^2$.  It is known that solutions to the initial value problem exhibit bubbling along a well-chosen sequence of times. %  -- the solution decouples into a superposition of harmonic maps and the weak limit.  
We prove that every sequence of times admits a subsequence along which bubbling occurs. 
% regardless of whether or not it is a Palais-Smale sequence of the energy functional. 
%along every sequence of times, after passing to a subsequence. 
%every sequence of times tending to the maximal time of existence admits a subsequence where the solution admits such a bubble decomposition. 
This is deduced as a corollary of our main theorem, which shows that the solution approaches the family of multi-bubble configurations in continuous time. 
\end{abstract}

%\subjclass[2010]{35L71 (primary), 35B40, 37K40}

\thanks{J. Jendrej is supported by  ANR-18-CE40-0028 project ESSED.  A. Lawrie is supported by NSF grant DMS-1954455 and the Solomon Buchsbaum Research Fund. W. Schlag is supported by NSF grant DMS-1902691.
}

\maketitle

\tableofcontents

\section{Introduction}

\subsection{Setting of the problem}

Consider the harmonic map heat flow (HMHF) for maps $u: \R^2 \to \Sp^2 \subset \R^3$, that is, the gradient flow of the Dirichlet energy 
\EQ{ \label{eq:energy} 
E(u) := \frac{1}{2} \int_{\R^2} \abs{\na u(x)}^2 \, \ud x, 
}
for the $L^2$ inner product. 
%Viewing $\Sp^2 \subset \R^3$ as isometrically embedded, 
The initial value problem for the HMHF is given by
\EQ{ \label{eq:hmhf} 
\p_t u &= \De u + u \abs{\na u}^2  \\
u(0, x) &= u_0(x).
}

We consider initial data in the  energy class,  
\EQ{
\E:= H^1( \R^2; \Sp^2) := \{ u_0 \in \dot H^1(\R^2; \R^3) \mid  |u_0(x)|^2 =1  \, \, \textrm{for almost every} \, x \in \R^2\}.
}
%are measurable functions $u_0: \R^2 \to \R^3$ such that  $|u_0(x)|^2 = 1$ for a.e. $x \in \R^2$, with finite energy, i.e., the class $ \E:=H^1(\R^2; \Sp^2)$. 
%and tending to a fixed point on the sphere at $\infty$. 
%We denote this energy space by $\E:= H^1(\R^2; \Sp^2)$.
%, i.e., 
%\EQ{
%\calE:=  \{ u_0 \in C^\infty(\R^2;  \Sp^2) \mid  E( u_0)< \infty, \,  % | u_0(x)|^2 = 1, \,  \forall x \in \R^2 \mand 
%\exists u_\infty \in \Sp^2 \, \,  \textrm{such that} \, \, \lim_{x \to \infty} |u_0(x) - u_\infty|   = 0 \}.
%}
%such that there exists $u_\infty \in \Sp^2 \subset \R^3$ so that  $\lim_{x \to \infty} |u_0(x) - u_\infty|   = 0 $. 
The HMHF was proved to be well-posed in $\E$ by Struwe~\cite{Struwe85}, and we can associate to each initial data $u_0 \in \E$  a maximal time of existence $T_+= T_+(u_0) \in (0, \infty]$, and unique solution $u(t) \in \E$, which is regular for $t \in (0, T_+)$. The maximal time $T_+$ is characterized as the first time at which energy concentrates at a point in space; see Lemma~\ref{lem:lwp}. 
Of fundamental importance is the energy identity, 
\EQ{ \label{eq:en-id-i} 
E(u(t_2)) +  \int_{t_1}^{t_2} \| \calT(u(t)) \|_{L^2(\R^2)}^2\, \ud t = E(u(t_1)), 
}
which holds for any $0 \le t_1 < t_2 < T_+$ (see~\cite[Lemma 3.4]{Struwe85}),  and where  $\calT(u):= \De u + u \abs{\na u}^2 $, which is called the \emph{tension} of $u$.   

The HMHF for maps between Riemannian manifolds was introduced by Eells and Sampson~\cite{ES}.  Though we do not do this here, when studying the HMHF for maps $\R^2  \to \Sp^2$, it is natural to further restrict the class of initial data by intersecting the space $\calE$ with the set of continuous maps $u_ 0$ that tend to a fixed vector on $\Sp^2$ at $\infty$, i.e., such that there exists $u_\infty \in \Sp^2$ so that $ \lim_{x \to \infty} |u_0(x) - u_\infty|   = 0$. By assigning to the point at $\infty$ the vector $u_\infty$,  $u_0$ induces a continuous map $\ti u_0: \Sp^2 \to \Sp^2$ and we can define the topological degree of $u_0$ to be the degree of $\ti u_0$. One can show that this condition is preserved by the flow, that is the solution $u(t, x)$ satisfies $ \lim_{x \to \infty} |u(t, x) - u_\infty|   = 0$ for all $0\le  t< T_+$.  Under this restriction,  the solution $u(t, x)$ gives a continuous deformation of the initial data $u_0(x)$ within its homotopy class, which was one of the motivations mentioned in~\cite{ES}. 

 Harmonic maps $\omega:\R^2\to\bbS^2 \subset \R^3$ have vanishing tension and give stationary solutions to~\eqref{eq:hmhf}. They are formal critical points of the energy~\eqref{eq:energy} 
%    \[
%E(\om) :=   \int_{\R^2} |\nabla \om(x)|^2\, \ud x  = \sum_{j=1}^2   \int_{\R^2} |\abl_j \om(x)|^2\, \ud x. 
%  \]
and satisfy the PDE, 
 \EQ{\label{eq:EL}
 \Delta \om + \om |\nabla \om|^2 =0. 
}
It is a well-known general property of harmonic maps in two dimensions that they are conformal (up to change of orientation) and minimize the energy in their homotopy class;~\cite{EW, Eel-Lem, Lem}. The energy of a harmonic map $\om$ is given by $E(\om) = 4 \pi | \deg(\om)|$. Weak solutions, that is $\om \in \E$ for which \eqref{eq:EL} holds in the weak sense, are smooth by a result of H\'elein~\cite{Hel}; see Theorem~\ref{thm:hmap} in Section~\ref{sec:hmap}. 

%Weak solutions of~\eqref{eq:EL} are smooth by a result of H\'elein~\cite{Hel}. This class is defined as 
%$\om \in \E$ for which \eqref{eq:EL} holds in the weak sense; see Theorem~\ref{thm:hmap} in Section~\ref{sec:hmap}. 
%%$\om\in H^1(\R^2;\bbS^2)$ with $|\om|=1$ a.e.\ and for which~\eqref{eq:EL} holds in the weak sense; see Theorem~\ref{thm:hmap} in Section~\ref{sec:hmap}. 

\subsection{Statement of the results}
The goal of this paper is to give asymptotic descriptions of solutions $u(t)$ to~\eqref{eq:hmhf} with initial data $u_0 \in \E$. Our first main result is that every sequence of times tending to the maximal time admits a subsequence $t_n \to T_+$ along which  $u(t_n)$ admits a decomposition into a finite superposition of rescaled and translated harmonic maps. 

We use the notation $D(y, \rho) \subset \R^2$ to denote the open disc of radius $\rho>0$  centered at the point $y \in \R^2$.

%  \dots  series of works by Struwe~\cite{Struwe85},  Qing~\cite{Qing}, Ding-Tian~\cite{DT}, Wang~\cite{Wang}, Qing-Tian~\cite{QT}, and Topping~\cite{Topping-winding}.  
% \Red{maybe say how sequence of times is chosen, vanishing tension, etc.}. 
 
% Our first main result is that solutions to~\eqref{eq:hmhf} admit bubble decompositions in the sense of~\cite{Qing} along \emph{any} sequence of times $t_n \to T_+$, up to passing to a subsequence. 

\begin{mainthm}[Bubble decomposition along any time sequence] \label{thm:main1} 
Let $u(t)$ be the unique solution to~\eqref{eq:hmhf} associated to initial data $u_0 \in \E$.  Let $T_+ = T_+(u_0) \in (0, \infty]$ denote the maximal time of existence.

\emph{(Finite time blow-up)} Suppose $T_+ < \infty$. There exist a finite energy map $u^*: \R^2 \to \Sp^2$, an integer $L \ge 1$, and points $\{x^\ell\}_{\ell =1}^L \subset \R^2$ with the following properties. 

Let $t_n \to T_+$ be any time sequence. After passing to a subsequence, which we still denote by $t_n$, we can associate to each $\ell \in \{1, \dots, L\}$ an integer $M^{(\ell)}$, sequences $a_{ j, n}^{(\ell)} \in \R^2$   and $\lam_{ j, n}^{(\ell)} \in (0, \infty)$ for each $j \in \{1, \dots, M^{(\ell)}\}$,  with $a_{ j, n}^{(\ell)} \to x^{\ell}$,  $\frac{\lam_{ j, n}^{(\ell)}}{\sqrt{T_+- t_n}}   \to 0$ as $n \to \infty$,   and nontrivial harmonic maps $\om_{1}^{(\ell)},  \dots, \om_{M^{(\ell)}}^{(\ell)}$ so that 
\EQ{ \label{eq:ao-param} 
\lim_{n \to \infty} \bigg( \frac{\lam_{ j, n}^{(\ell)}}{\lam_{ k, n}^{(\ell)}} + \frac{\lam_{k, n}^{(\ell)}}{\lam_{ j, n}^{(\ell)}} + \frac{ | a_{j, n}^{(\ell)} - a_{k, n}^{(\ell)}|}{ \lam_{ j, n}^{(\ell)}}\bigg)  = \infty \quad \textrm{for all}\,\,   j \neq k , 
}
and 
\EQ{
\lim_{n \to \infty} E \bigg( u(t_n) - u^* - \sum_{\ell = 1}^{L}   \sum_{j =1}^{M^{(\ell)}} \Big( \om_{j}^{(\ell)} \Big( \frac{ \cdot - a_{ j, n}^{(\ell)}}{\lam_{ j, n}^{(\ell)}} \Big) - \om_{j}^{(\ell)}(\infty) \Big) \bigg) = 0, 
}
where $\om_{j}^{(\ell)}(\infty) := \lim_{\abs{x} \to \infty} \om_{j}^{(\ell)} (x) \in \Sp^2$.  %and $R_0>0$ such that for all $0< R< R_0$, 
%\EQ{
%E( u(t)
%}

Moreover, there exists a sequence $r_n \to \infty$ with the following property. Fix any $\ell \in \{1, \dots, L\}$. For each $j \in \{1, \dots, M^{(\ell)}\}$, there exists $0 \le K_j^{(\ell)}< M^{(\ell)}$ many discs $D(x_{j, k, n}, \mu_{j,k, n}) \subset D(a_{j, n}^{(\ell)}, r_n \lambda_{j, n}^{(\ell)})$ such that for each $k \in \{1, \dots, K_j^{(\ell)}\}$, 
\EQ{
\lim_{n \to \infty} \Big( \frac{\mu_{j,k, n}}{\lambda_{j, n}^{(\ell)}}  + \frac{\mu_{j, k, n}}{ \dist(x_{j, k, n}, \partial D( a_{j, n}^{(\ell)}, r_n \lam_{j, n}^{(\ell)}))}\Big) = 0, 
}
 and so that 
\EQ{\label{eq:sup-blowup} 
\lim_{n \to \infty} \Big\| u(t_n) - \om_{j}^{(\ell)} \Big( \frac{ \cdot - a_{ j, n}^{(\ell)}}{\lam_{ j, n}^{(\ell)}} \Big) \Big\|_{L^\infty( D^*_{j, n})}  = 0, 
} 
where $D^*_{j, n} = D(a_{j, n}^{(\ell)}, r_n \lambda_{j, n}^{(\ell)}) \setminus \bigcup_{k =1}^{K_j^{(\ell)}} D(x_{j, k, n}, \mu_{j,k, n})$. 

Finally, there exist  constants $\om_\infty^{(1)}, \dots, \omega_{\infty}^{(L)} \in \Sp^2$ and sequences $\xi_n, \nu_n \to 0$ so that for each $\ell \in \{1, \dots, L\}$, 
\EQ{\label{eq:ss-neck} 
\lim_{n \to \infty}\Big(  \| u( t_n) - \om^{(\ell)}_{\infty}  \|_{L^{\infty}( D( x^{\ell},  \nu_n) \setminus D( x^{\ell},  \xi_n))} + \frac{\xi_n}{\sqrt{T_+-t_n}} + \frac{\sqrt{T_+-t_n}}{\nu_n} \Big) = 0 .
}

%Moreover, there exists constants $\om^{(1)}_\infty, \dots, \om^{(L)}_{\infty} \in \Sp^2$ and a sequence $r_n \to \infty$ such that for each $\ell \in \{1, \dots, L\}$,  and each $j \in \{1, \dots, M^{(\ell)}\}$, 
%\EQ{
%\lim_{n \to \infty} \Big\| u(t_n) -\omega^{(\ell)}_{\infty} - \sum_{j =1}^{M^{(\ell)}} \Big( \om_{j}^{(\ell)} \Big( \frac{ \cdot - a_{ j, n}^{(\ell)}}{\lam_{ j, n}^{(\ell)}} \Big) - \om_{j}^{(\ell)}(\infty) \Big) \Big\|_{L^\infty( D( a_{j, n}^{(\ell)}, r_n \lam_{j, n}^{(\ell)}))}  = 0.
%} 

\emph{(Global solution)} Suppose $T_+= \infty$. Let $t_n \to \infty$ be any time sequence. After passing to a subsequence, which we still denote by $t_n$, we can find an integer $M \ge 0$, sequences $a_{j, n} \in \R^2$ and $\lam_{ j, n} \in (0, \infty)$  for each $j \in \{1, \dots, M\}$, with $ \lim_{n \to \infty} \frac{|a_{j, n}| + \lam_{ j, n}}{\sqrt{t_n}} = 0$,  and nontrivial harmonic maps $\om_1, \dots, \om_{M}$, so that 
\EQ{
\lim_{n \to \infty} \bigg( \frac{\lam_{j, n}}{\lam_{k, n}} + \frac{\lam_{k, n}}{\lam_{j, n}} + \frac{ | a_{ j, n} - a_{k, n}|}{ \lam_{j, n}} \bigg) = \infty \quad \textrm{for all} \, \,  j \neq k, 
}
and
\EQ{
\lim_{n \to \infty} E \bigg( u(t_n) - \om_\infty - \sum_{j =1}^{M} \Big(\om_{j} \Big( \frac{ \cdot - a_{ j, n}}{\lam_{ j, n}} \Big) - \om_{j}(\infty)\Big) \bigg) = 0, 
}
%and, 
%\EQ{ 
%\lim_{n \to \infty}  \| u(t_n) - \om_{\infty} \|_{L^\infty( D(0, \sqrt{t_n}) \setminus D(0, 2^{-1}\sqrt{t_n})} = 0, 
%}
%where $\om_\infty \in \Sp^2$ and 
where $\om_{j}(\infty) := \lim_{\abs{x} \to \infty} \om_{j} (x) \in \Sp^2$. 

Moreover, there exists a sequence $r_n \to \infty$ with the following property.  For each $j \in \{1, \dots, M\}$, there exists $0 \le K_j< M$ many discs $D(x_{j, k, n}, \mu_{j,k, n}) \subset D(a_{j, n}, r_n \lambda_{j, n})$ such that for each $k \in \{1, \dots, K_j\}$, 
\EQ{
\lim_{n \to \infty} \Big( \frac{\mu_{j,k, n}}{\lambda_{j, n}}  + \frac{\mu_{j, k, n}}{ \dist(x_{j, k, n}, \partial D( a_{j, n}, r_n\lam_{j, n}))}\Big) = 0, 
}
 and so that 
\EQ{ \label{eq:sup-global} 
\lim_{n \to \infty} \Big\| u(t_n) - \om_{j}^{(\ell)} \Big( \frac{ \cdot - a_{ j, n}^{(\ell)}}{\lam_{ j, n}^{(\ell)}} \Big) \Big\|_{L^\infty( D^*_{j, n})}  = 0,
} 
where $D^*_{j, n} = D(a_{j, n}, r_n \lambda_{j, n}) \setminus \bigcup_{k =1}^{K_j} D(x_{j, k, n}, \mu_{j,k, n})$. 

Finally, there exists a constant $\om_\infty \in \Sp^2$ and sequences $\xi_n, \nu_n \in (0, \infty)$ so that  
\EQ{\label{eq:ss-neck2} 
\lim_{n \to \infty}\Big(  \| u( t_n) - \om_{\infty}  \|_{L^{\infty}( D( x^{\ell},  \nu_n) \setminus D( x^{\ell},  \xi_n))} + \frac{\xi_n}{\sqrt{t_n}} + \frac{\sqrt{t_n}}{\nu_n} \Big) = 0 .
}

%Moreover, there exists a sequence $r_n \to \infty$ such that for  each $j \in \{1, \dots, M\}$, 
%\EQ{
%\lim_{n \to \infty} \Big\| u(t_n) -\omega_{\infty} -  \sum_{j =1}^{M} \Big(\om_{j} \Big( \frac{ \cdot - a_{ j, n}}{\lam_{ j, n}} \Big) - \om_{j}(\infty)\Big)\Big\|_{L^\infty( D( a_{j, n}, r_n \lam_{j, n}))}  = 0. 
%} 

\end{mainthm}  

\begin{rem}  \label{rem:seq} 
An influential series of works by Struwe~\cite{Struwe85}, Qing~\cite{Qing},  Ding-Tian~\cite{DT}, Wang~\cite{Wang}, Qing-Tian~\cite{QT}, Lin-Wang~\cite{Lin-Wang98}, and Topping~\cite{Topping-winding} showed that solutions $u(t)$ to~\eqref{eq:hmhf} admit a bubble decomposition in the sense of Theorem~\ref{thm:main1} along a well-chosen sequence of times $t_n \to T_+$; see also the book by Lin-Wang~\cite{Lin-Wang}. In these works, the bubbling time sequence $t_n \to T_+$ and corresponding sequence of maps $u(t_n)$ become a Palais-Smale sequence after rescaling. Indeed, in the case when $T_+ = \infty$ it follows from~\eqref{eq:en-id-i} that $\int_0^\infty  \| \calT(u(t)) \|_{L^2}^2\, \ud t < \infty$, so there exists a sequence $t_n \to \infty$ so that $\lim_{n \to \infty} \sqrt{t_n} \| \calT(u(t_n) \|_{L^2}  = 0$. By similar logic, in the case of finite time blow-up ($T_+ < \infty$) there is a sequence $t_n \to T_+$ so that $\lim_{n \to \infty} \sqrt{T_+ - t_n} \| \calT(u(t_n)) \|_{L^2} = 0$. In other words, after the rescaling $u_n(x) := u(t_n, \sqrt{t_n}x)$ (or $u_n(x) = u(t_n, \sqrt{T_+-t_n} x)$),  the $u_n$ are  Palais-Smale sequences for the energy functional because $\sup_nE(u_n) < \infty$ and $DE(u_n) = -\calT(u_n) \to 0$ in $L^2$. Elliptic bubbling analysis (see e.g,.~\cite{Struwe85, BrezisCoron, Qing}) is then used to extract bubbles up to the scale $\sqrt{t_n}$ (or $\sqrt{T_+- t_n}$ in the case $T_+ <\infty$). Theorem~\ref{thm:main1} is distinct from this classical literature in that we show bubbling occurs along every time sequence (after passing to a suitable subsequence), without the aid of a Palais-Smale sequence in the sense described above. On the other hand, the works~\cite{QT, Lin-Wang} show $L^\infty$ convergence including in the neck regions between the bubbles whereas here we control only the energy in the neck regions -- we do not address the question of $L^\infty$ convergence on the neck regions. 
%$L^\ifnty$ on in the that there are no necks between the bubbles in the $L^\infty$
% \Red{mention how the classical work gives the decomposition in $L^\infty$ up to the parabolic scale $\sqrt{T_+-t_n}$ in the blow-up case and $\sqrt{t_n}$ in the global case, but that we don't get that here...we only get $L^\infty$ at the scales of the bubbles... }
 \end{rem} 

%\begin{rem} 
%\Red{Maybe emphasize somewhere either here or above that the tension need not vanish along the sequence $t_n$, even after passing to the subsequence. }
%\end{rem} 
% \begin{rem} 
%Asymptotic decompositions of solutions to~\eqref{eq:hmhf}  were proved along a  \emph{well-chosen sequence of times} $t_n \to T_+$,  in a series of works by Struwe~\cite{Struwe85},  Qing~\cite{Qing}, Ding-Tian~\cite{DT}, Wang~\cite{Wang}, Qing-Tian~\cite{QT}, and Topping~\cite{Topping-winding}. \Red{Elaborate} 
%\end{rem} 

\begin{rem} 
One can also study the two-dimensional HMHF for more general domains and targets, that is for maps $u: \calM \to \calN$, where $\calM$ is a $2$-dimensional closed, orientable Riemannian manifold (or $\R^2$) and $\calN$ is a closed $n$-dimensional sub-manifold of $\R^N$ for some $N$, as in this case the bubbling theory of~\cite{Struwe85, Qing, DT, QT, Lin-Wang98} is understood. But we do not pursue this here. %The reason is that we use crucially the quantization energy 
%One can also take as the domain any closed, orientable  Riemannian surface   $\calM$.  
Moreover, the choice of $\R^2$ as the domain is for convenience as we could have instead considered maps $u: \Sp^2 \to \Sp^2$. 

%and $\Sp^2$ for the target is for convenience, as we could have instead considered finite energy maps $u_0: \calM \to \calN$, where $\calM, \calN$ are closed, orientable, simply connected Riemannian surfaces and we believe an analogous result holds via similar arguments. 

% which was one of the intentions of Eells and Sampson when they introduced the HMHF in~\cite{ES}. 

%\Red{mention how what is used is energy quantization and smoothness of harmonic maps, and thus results should extend in a straightforward way to the heat flow for maps $u: \calM \to \calN$ where $\calM, \calN$ are closed, orientable two-dimensional Riemannian manifolds. }

\end{rem}

We deduce Theorem~\ref{thm:main1} as a consequence of a more refined result, where we show that every smooth solution $ u(t)$ converges, continuously in time, to the family of multi-bubble configurations, locally about any point in space. To state this result we first define a notion of \emph{scale} and~\emph{center} of a non-trivial harmonic map.  

 \label{sec:scalecenter}

%\colorbox{BurntOrange}{Changes start}

\begin{defn}[Scale of a harmonic map] 
\label{def:scale}
To each non-constant harmonic map $\om: \R^2  \to \Sp^2 \subset \R^3$ and each $\gamma_0 \in (0, 2\pi)$  we associate a scale $\lam( \om; \gamma_0)$ defined by 
\EQ{
\lam( \om; \gamma_0):= \inf\{ \lam \in (0, \infty) \mid \textrm{there exists}\, \, a \in \R^2\, \, \textrm{such that} \, \, E( \om; D(a, \lam)) \ge E( \om) - \gamma_0\}.
}
\end{defn}

%Since $E(\omega;D(0,R))\to E(\omega)$ as $R\to\infty$, it follows that this notion is well-defined. 
%If $\lam( \om; \gamma_0)=0$, then there exist $a_n\in\R^2$ so that 
%\EQ{\label{eq:scale contra}
%E( \om; D(a_n, 1/n)) \ge E( \om) - \gamma_0 \qquad \forall\; n\ge1 
%}
%If $n\ne m$, the $D(a_n, 1/n)\cap D(a_m, 1/m)\ne\emptyset$. Indeed, otherwise
%\[
%E(\omega)\ge E( \om; D(a_n, 1/n))+E( \om; D(a_m, 1/m)) \ge 2E( \om) - 2\gamma_0
%\]
%whence $E(\omega)\le 2\gamma_0< 4\pi$ which contradicts that $\omega$ is not constant. Therefore, $\{a_n\}_{n=1}^\infty$ is a Cauchy sequence in $\R^2$, and $a_n\to a_\infty$. Passing to the limit in~\eqref{eq:scale contra} gives a contradiction. 

\begin{defn}[Center of a harmonic map] 
\label{def:center} 
Given the scale of a harmonic map $\om$ as above, we  define the associated center of $\omega$ by fixing a choice of $a = a(\om; \gamma_0) \in \R^2$ so that 
\EQ{\label{eq:cent def}
E( \om; D( a(\om; \gamma_0), \lam(\om; \gamma_0))) \ge E( \om) -  \gamma_0. 
}
\end{defn} 

We prove in Lemma~\ref{lem:scale} that these notions are well-defined and transform naturally under the rescaling and translation of a harmonic map. Indeed, the scale $\lam( \om; \gamma_0)$ is a uniquely defined, strictly positive number. Regarding a choice of center, equality occurs in~\eqref{eq:cent def}. However,  $a(\om; \gamma_0)$ is defined only up to a distance of $2\lam( \om; \gamma_0)$.  
%any nontrivial harmonic map $\omega$, there exists a unique positive scale $\lam( \om; \gamma_0)$ with and center $a(\om, \gamma_0)$ that is defined only up to a distance of $2\lam( \om; \gamma_0)$. 
 % but $a(\om; \gamma_0)$ is defined only up to a distance of $2\lam( \om; \gamma_0)$.  
Given a harmonic map $\om(x)$, translating by $b \in \R^2$ and  rescaling by $\mu \in (0, \infty)$ we obtain $ \om_{b, \mu}(x) := \om\big( \frac{x - b}{\mu})$.  Then $\lam(\om_{b, \mu}) = \lam(\omega)\mu$ and $|a(\om_{b, \mu})- a(\omega)\mu -b| \le 2\lam(\om)\mu $. %see Lemma~\ref{lem:scale}. 

%To see that this center is well-defined, take $\lambda_n\to \lam(\om; \gamma_0)$ and $a_n\in\R^2$ such that 
%\[
%E( \om; D( a_n, \lam_n)) \ge E( \om) -  \gamma_0
%\]
%As before, we conclude that no two disks $\{D( a_n, \lam_n)\}_{n=1}^\infty$ can be disjoint. Thus,  $a_n\in\R^2$ lie in a compact set and we may assume that $a_n\to a_\infty$ as $n\to\infty$, which is the desired center. 
%We note that $\lam( \om; \gamma_0)$ is uniquely defined, but $a(\om; \gamma_0)$ is defined only up to a distance of $2\lam( \om; \gamma_0)$.  
%
%Given a harmonic map $\om(x)$, translating by $b \in \R^2$ and  rescaling by $\mu \in (0, \infty)$ we obtain $ \om_{b, \mu}(x) := \om\big( \frac{x - b}{\mu})$.  Then $\lam(\om_{b, \mu}) = \lam(\omega)\mu$ and $|a(\om_{b, \mu})- a(\omega)\mu -b| \le 2\lam(\om)\mu $; see Lemma~\ref{lem:scale}. 

%\colorbox{BurntOrange}{Changes end}

\begin{defn}[Multi-bubble configuration] 
%Let $\gamma_0>0, M \in \{0, 1, 2, \dots\}$ and let $\calN_{\Sp^2, \gamma_0}$ denote the closed tubular neighborhood of $\Sp^2 \in \R^3$ of width $\gamma_0$.  
%We define an approximately sphere-valued multi-bubble configuration to be a superposition of harmonic maps 
Let $M \in \{0,1, 2, \dots\}$. We define an $M$-bubble configuration to be a superposition 
\EQ{
\calQ( \om,  \om_1, \dots, \om_M; x) = \om+  \sum_{j=1}^M ( \om_j(x)  - \om_j(\infty)), 
}
where $\om \in \Sp^2$ is a constant, and each $\om_j: \R^2 \to \Sp^2$ is a smooth non-constant harmonic map, and $\om_j(\infty):= \lim_{\abs{x} \to \infty} \om_j(x)$. We include constant maps as $M=0$. 
%\Red{[The constants $\calQ_{\infty}$ and $\om_{j}(\infty)$ do not play a role in the analysis, since they are ignored by the energy, which only involves $\na \calQ$, however, we include them in the definition of a multi-bubble for consistency with the literature (see e.g.,~\cite{Qing}) and because they do become relevant if one studies the dynamics of solutions to~\eqref{eq:hmhf} in different function spaces, such as $L^\infty$. -- possibly remove this]} %\Red{[Question: should we include the constants $\calQ_{\infty}$ and $\om_j(\infty)$ even though the energy norm ignores them? One reason for doing so is this way the decompositions appear as in Qing]}
%\Red{[if later we van obtain $L^\infty$ convergence we can discuss ``almost $\Sp^2$-valued'' confirgurations]} 
%and such that $\calQ( \om_1, \dots, \om_M; x) \in \calN_{\Sp^2, \gamma_0}$ for all $x \in \R^2$. 
\end{defn} 

We will occasionally use boldface notation $\bs \om: = (\om,  \om_1, \dots, \om_M)$, for finite sequences of harmonic maps with $\om \in \Sp^2$ a constant harmonic map and $\om_{1}, \dots, \om_{M}$ non-constant, and we reserve the arrow notation for vectors (finite sequences) in other contexts.  With this notation we will often express multi-bubbles as $ \calQ( \bs \omega):=\calQ( \om,  \om_1, \dots, \om_M)$. We reserve the character $\frakh$ to denote an infinite sequence of non-constant harmonic maps, i.e., $\frakh:= \{ \om_{n}\}_{n =1}^\infty$, where each $\om_n$ is a harmonic map. 

\begin{defn}[Localized distance to a multi-bubble configuration]  \label{def:d}
 Let $\xi, \rho, \nu \in (0, \infty)$, with $\xi \le \rho\le \nu$,  $y \in \R^2$, $u: D( y, \nu) \to \Sp^2$, and $\gamma_0  \in (0, 2\pi)$ as in Definition~\ref{def:scale}. Let  $M \in \{0, 1, 2, \dots\}$, $\om\in \Sp^2$ a constant, and let $\om_1, \dots, \om_M$ be non-constant harmonic maps with centers $a(\om_j)\in D( y, \xi)$ for each $j \in \{1, \dots, M\}$ and scales $\lambda(\om_{j}) \in (0, \infty)$.  Let $\calQ( \bs \omega)$  be the associated multi-bubble configuration.  Let $ \vec \nu = (\nu, \nu_1, \dots, \nu_M) \in (0, \infty)^{M+1}$ be such that $D(a(\om_j), \nu_j) \subset D(y, \xi)$ for each $j \in \{1, \dots, M\}$. Let $\vec \xi = (\xi,  \xi_1, \dots,\xi_M) \in (0, \infty)^{M+1}$ be such that $\xi_j < \lam(\om_j)$ for each $j \in \{1, \dots, M\}$.  Denote by $\calI_j:= \{ k \neq j \mid  D(a(\om_k), \xi_j) \subset D(a(\om_j), \nu_j)\}$,   and let 
 \EQ{
 D_{j}^*:=D( a(\om_j), \nu_j) \setminus \bigcup_{k\in \calI_j }D(a(\om_k), \xi_j).
 }%\colorbox{BurntOrange}{Do we want $\omega_0$ here?}
Define, 
 \EQ{
  \bfd_{\gamma_0}( u, \calQ( \bs  \omega); D(y, \rho); \vec \nu, \vec \xi) &:=  E\big( u - \calQ(  \bs \omega);  D(y, \rho)\big) + \sum_j  \| u -  \omega_j \|_{L^\infty( D_j^*)}  \\
  & \quad + \| u - \om \|_{L^{\infty}(D( y, \nu) \setminus D(y, \xi))} + E(u; D( y, \nu) \setminus D(y, \xi))+ \frac{\xi}{\rho} + \frac{\rho}{\nu}    \\  
& \quad  + \sum_{j \neq k} \bigg( \frac{\lam(\om_j)}{\lam(\om_k)} +  \frac{\lam(\om_k)}{\lam(\om_j)} + \frac{| a(\om_j) - a(\om_k)|}{\lam(\om_j)} \bigg)^{-1}  \\
&  \quad + \sum_j \Big(\frac{\lam(\om_j)}{ \dist( a(\om_j), \p D(y, \xi))} + \frac{\lambda(\om_j)}{\nu_j} + \frac{\xi_j}{\lam(\om_j)}\Big)  \\
& \quad + \sum_j \sum_{k \in\calI_j }\frac{\xi_j}{\dist(a(\om_k), \partial D( a(\om_j), \nu_j))} .
  }
\end{defn} 

We define a localized distance function to the family of all multi-bubble configurations as follows. 

\begin{defn}[Localized multi-bubble proximity function]   Let $y \in \R^2$,  $\rho \in (0, \infty)$,  $u: D( y, \rho) \to \Sp^2$, and let $\gamma_0 \in (0,2 \pi)$ as in Definition~\ref{def:scale}. We define
\EQ{
\bs\de_{\gamma_0}( u; D(y, \rho)) := \inf_{\bs \omega, \vec \nu, \vec \xi} & \bfd_{\gamma_0}( u, \calQ( \bs  \om);  D(y, \rho); \vec \nu, \vec \xi) 
% \Bigg[ E\big( u - \calQ( \om_1, \dots, \om_M);  D(y, \rho)\big) \\
%&\quad + \sum_{j \neq k} \bigg( \frac{\lam(\om_j)}{\lam(\om_k)} +  \frac{\lam(\om_k)}{\lam(\om_j)} + \frac{| a(\om_j) - a(\om_k)|}{\lam(\om_j)} \bigg)^{-1}  \\
%&\quad + \sum_j \frac{\lam(\om_j)}{ \dist( a(\om_j), \p D(y, \rho))}\Bigg]^{\frac{1}{2}}
}
where the infimum above is taken over all $M \in \{0,1, 2, \dots\}$, all possible $M$-bubble configurations $\calQ(\bs \om)$, and over all admissible $\vec \nu = (\nu, \nu_1, \dots, \nu_M)\in (0, \infty)^{M+1},  \vec \xi = (\xi, \xi_1, \dots, \xi_M) \in (0, \infty)^{M+1}$ in the sense of Definition~\ref{def:d}. Since $\gamma_0$ will eventually be fixed we will often suppress the dependence of $\bfd_{\gamma_0}$ and $\bs \de_{\gamma_0}$ on $\gamma_0$ and just write $\bfd, \bs \delta$. %\Red{mention manifold structure of family of multi-bubble configurations?} 
\end{defn}

We prove the following theorem. %\colorbox{BurntOrange}{In finite time $L^\infty$ not correct ?!}

\begin{mainthm}[Convergence to multi-bubbles in continuous time] \label{thm:main} Let $u(t)$ be the unique solution to~\eqref{eq:hmhf} associated to initial data $u_0 \in \E$. 
%Let $u(t)$ 
%be the smooth solution to~\eqref{eq:hmhf} associated to smooth finite energy initial data $u_0: \R^2 \to \Sp^2$ satisfying    $\lim_{x \to \infty}| u_0(x) - u_\infty| = 0$ for some $u_\infty \in \Sp^2$. 
%\red{[which data can we consider, and explain what we mean by solution]} 
Let $T_+ = T_+(u_0) \in (0, \infty]$ denote the maximal time of existence.  There exists $\gamma_0 = \gamma_0(E( u_0))>0$ as in Definition~\ref{def:scale} sufficiently small so that the following conclusions hold. 

\emph{(Finite time blow-up)} 
 Suppose $T_+< \infty$. For every $y \in \R^2$,  %there exists $T_0>0$ and a function $\rho: [T_0, T_+) \to (0, \infty)$ with $\lim_{t \to T_+}\frac{\rho(t)}{ \sqrt{T_+- t}}=\infty$, so that
 \EQ{
 \lim_{t \to T_+} \bs \de_{\gamma_0}\big( u(t); D(y, \sqrt{T_+-t})\big)  = 0. 
 }
 Moreover, let $t_n \to T_+$ be any sequence  and let  $D(y_n, \rho_n)$ be any sequence of discs such that $D(y_n, R_n \rho_n) \subset D(y, \sqrt{T_+-t})$ for some sequence $R_n \to \infty$. Suppose  $\al_n, \be_n$ are sequences with $\al_n \to 0$, $\be_n \to \infty$, $\lim_{n \to \infty} \be_n R_n^{-1} = 0$, and
  \EQ{
 \lim_{n \to \infty}   E\big(u(t_n); D(y_n, \be_n \rho_n) \setminus D( y_n, \al_n \rho_n)\big)  = 0. 
 }
 Then, 
 \EQ{
 \lim_{ n \to \infty} \bs \de_{\gamma_0}\big( u(t_n); D( y_n, \rho_n)\big) = 0. 
 }
% \Red{I suppose this can be made into a statement in continuous time as well}
%If $T_+< \infty$ there exists $T_0>0$ so that for every $y \in \R^2$ and every function $\rho: [T_0, T_+) \to (0, \infty)$ satisfying $\lim_{t \to T_+} \rho(t) = 0$  and 
%\EQ{
%\lim_{t \to T_+} E( u(t); \frac{1}{2} \rho(t); 2 \rho(t)) = 0, 
%}
%we have, 
%\EQ{
%\lim_{t \to T_+} \bs \de( u(t); D(y, \rho(t)))  = 0. 
%}
%In particular, the hypotheses hold for $\rho(t) = \sqrt{T_+- t}$. 

\emph{(Global solution)} Suppose $T_+ = \infty$. For every $y \in \R^2$, 
\EQ{
 \lim_{t \to \infty} \bs \de_{\gamma_0}\big( u(t); D(y, \sqrt{t})\big)  = 0. 
 }
 Moreover, let $t_n \to \infty$ be any sequence  and let  $D(y_n, \nu_n)$ any sequence of discs such that $D(y_n, R_n \nu_n) \subset D(y, \sqrt{t_n})$ for some sequence $R_n \to \infty$. Suppose  $\al_n, \be_n$ are sequences with $\al_n \to 0$, $\be_n \to \infty$, $\lim_{n \to \infty} \be_n R_n^{-1} = 0$, and
  \EQ{
 \lim_{n \to \infty}   E\big(u(t_n); D(y_n, \be_n \rho_n) \setminus D( y_n, \al_n \rho_n)\big)  = 0. 
 }
 Then, 
 \EQ{
 \lim_{ n \to \infty} \bs \de_{\gamma_0}\big( u(t_n); D( y_n, \rho_n)\big) = 0. 
 }

%If $T_+= \infty$, then there exists $T_0>0$ so that for every $y \in \R^2$ and every function $\rho: [T_0, \infty) \to (0, \infty)$ satisfying 
%\EQ{
%\lim_{t \to T_+} E( u(t); \frac{1}{2} \rho(t); 2 \rho(t)) = 0, 
%}
%we have, 
%\EQ{
%\lim_{t \to \infty} \bs \de( u(t); D(y, \rho(t)))  = 0. 
%} 
%In particular, the hypotheses is satisfied for $\rho(t) = \sqrt{t}$. \Red{Perhaps we just fix $\rho(t) = \sqrt{t}$ in the statement and then say later that we prove the stronger result as stated above}
%\red{I think we need some care here in which functions $\rho(t)$ are allowed so we can avoid having a blow up point sitting on the boundary of the disc} 
\end{mainthm}

\begin{rem} \label{rem:continuous}
Theorem~\ref{thm:main} can be viewed as partial progress towards the following questions, which arise naturally from the classical sequential bubbling results~\cite{Struwe85, Qing, DT, Wang, QT, Lin-Wang98, Topping-winding}: 
\begin{itemize} 
\item Can the harmonic maps (bubbles) appearing in Theorem~\ref{thm:main1} be taken independently of the time sequence? 
\item In particular, can the decomposition in Theorem~\ref{thm:main1} be taken in continuous time, i.e.,  does $u(t)$ converge in the energy space to $u^*$ plus a superposition of a fixed collection of harmonic maps that are continuously modulated by a finite number of parameters independently of the degree, (for example via the underlying symmetries such as scaling, spatial translations, and rotations)? 
\end{itemize} 

Topping~\cite{Top-JDG, Topping04} made important progress on these and related questions in the case of a global-in-time solution ($T_+ = \infty$), showing the uniqueness of the locations of the bubbling points, and that $u(t)$ converges weakly to a unique harmonic map as $t \to \infty$, all under restrictions on the configurations of bubbles appearing in the sequential decomposition. His assumption, roughly, is that all of the \emph{concentrating} bubbles have the same orientation. Here we do not make any assumptions on the orientations of the bubbles, but our results in the global-in-time case are of a different nature and we do not recover Topping's conclusions. 

Topping answered the questions above in the negative for the HMHF for maps from $\Sp^2$ into certain target manifolds; see~\cite{Topping-winding}.

The first two authors answered the questions above in the affirmative in the case that the target is $\Sp^2$ and the initial data for~\eqref{eq:hmhf} is $k$-equivariant; see~\cite{JL8}.  
\end{rem}

\begin{rem} 
One can view Theorem~\ref{thm:main} as a statement about the non-existence of bubble collisions (asymptotically in time) that destroy multi-bubble structure. Here a bubble collision on a disc  $D(y, \rho)$ is defined via the growth of the function $\bs\de(u(t); D(y, \rho))$, i.e., $u(t)$ starts close to, but then moves away from the family of multi-bubble configurations on some time interval. Roughly speaking, Theorem~\ref{thm:main} reduces the questions in Remark~\ref{rem:continuous} to an analysis of the dynamics of solutions close to the manifold of multi-bubble configurations.

\end{rem}

%\begin{rem}
% \label{Topping}  Topping~\cite{Top-JDG, Topping04} made important progress on a related question in the global case, showing the uniqueness of the locations of the bubbling points under restrictions on the configurations of bubbles appearing in the sequential decomposition. His assumption, roughly, is that all of the bubbles concentrating at a certain point have to have the same orientation. 

%\Red{Compare/contrast.} 

%We can contrast this assumption with the equivariant setting, where in the decomposition~\eqref{eq:sr-global} subsequent bubbles have opposite orientations. 

%It is known from classical work of Eells and Wood~\cite{EW} that the finite energy harmonic maps $\Sp^2 \to \Sp^2$ are either holomorphic or anti-holomorphic viewed as maps $\C_{\infty} \to \C_{\infty}$ and thus can be identified with the rational functions in $z$ or $\ba z$. 
%\end{rem} 

%\begin{rem} 
%Given Theorems~\ref{thm:main1},~\ref{thm:main}, it is natural to ask which configurations of bubbles are possible in the decomposition.  Van der Hout~\cite{vdHout03} showed that there can only be one bubble in the decomposition in the case of \emph{equivariant}  finite time blow-up; see also~\cite{BvdHH}. In contrast, in the infinite time case, it is expected that there can be equivariant bubble trees of arbitrary size (see recent work of Del Pino, Musso, and Wei~\cite{DPMW} for a construction in the case of the critical semi-linear heat equation). 
%\end{rem} 

\begin{rem} 
There are solutions to the HMHF that develop a bubbling singularity in finite time, the first being the examples of Coron and Ghidaglia~\cite{CG} (in dimensions  $\ge 3$) and Chang, Ding, Ye~\cite{CDY} in two dimensions.  Guan, Gustafson, and Tsai~\cite{GGT} and Gustafson, Nakanishi, and  Tsai~\cite{GNT} showed that $k$-equivariant harmonic maps are asymptotically stable for perturbations within their equivariance classes when  $k \ge 3$, and thus there is no finite time blow up for energies close to the harmonic map in that setting. For $k=2$,~\cite{GNT} gave examples of solutions exhibiting infinite time blow up and eternal oscillations, and recently Wei, Zhang, Zhou~\cite{WZZ} constructed such examples in the case $k=1$. Rapha\"el and Schweyer constructed a stable equivariant blow-up regime for $k=1$ in~\cite{RSc-13} and then equivariant blow up solutions with different rates in~\cite{RSc-14}.  Davila, Del Pino, and Wei~\cite{DDPW} constructed examples of solutions simultaneously concentrating a single copy of the ground state harmonic map at distinct points in space. See also the recent work of Del Pino, Musso, and Wei~\cite{DPMW} for a construction of bubble towers with an arbitrary number of bubbles in the case of the critical semi-linear heat equation. 
\end{rem}

\subsection{Summary of the proof}

We give an informal description of the proof of Theorem~\ref{thm:main} and then we discuss how to deduce Theorem~\ref{thm:main1} from it. %Theorem~\ref{thm:main}. 

To fix ideas, we consider a solution blowing up at a finite time $T_+< \infty$. Theorem~\ref{thm:main} is proved by contradicting the finiteness of the integral 
\EQ{ \label{eq:T} 
\int_0^{T_+} \| \calT( u(t)) \|_{L^2}^2 \, \ud t < \infty, 
}
via a collision analysis in the event that the theorem fails. 
%The above integral is finite due to the energy identity~\eqref{eq:en-id-i}.   
The collision analysis hinges on the notion of a ~\emph{minimal collision energy} and the corresponding \emph{collision (time) intervals} that accompany it. These are defined as follows (see Section~\ref{ssec:collision}).  We let $K$ be the smallest integer so that there exist time sequences $\sigma_n, \tau_n \to T_+$, a sequence of discs $D(y_n, \rho_n) \subset \R^2$, a number $\eta>0$, and a sequence $\eps_n \to 0$ so that $\bs \de(u(\sigma_n); D( y_n, \rho_n)) \le \eps_n$, $\bs \de( u(\tau_n); D(y_n, \rho_n)) \ge \eta$, and $E( u( \sigma_n); D( y_n, \rho_n)) \to 4 K \pi$ as $n\to \infty$. To ensure that $K$ is well-defined and $\ge 1$ in the event that theorem fails; see Lemma~\ref{lem:K}) we also require that $|[\sigma_n, \tau_n]|  \le \eps_n \rho_n^2$. 
%and that the energy in the neck regions $E( u(\sigma_n); D( y_n, 4 \rho_n)\setminus D( y_n, \frac14 \rho_n)$ vanishes as $n  \to \infty$. 
We emphasize that the quantization of the energy of harmonic maps $\R^2 \to \Sp^2$ is used to define $K$ as above.  Roughly speaking, the intervals $I_n :=[\sigma_n, \tau_n]$ have the property that $u$ is close to a multi-bubble configuration on the left endpoint  $t = \sigma_n$ (which we call \emph{bubbling times}) and far from every multi-bubble at the right endpoint $t = \tau_n$ (which we call \emph{ejection times}). %In the event that the theorem fails, these notions are well-defined; see Lemma~\ref{lem:K}.  
%meaning we can find a sequence of such intervals
%where $u$ is a fixed distance $\eta>0$ away from the family of multi-bubbles at the right endpoints, but converges in energy to the multi-bubble family along the sequence of left endpoints.  
%The discs $D( y_n, \rho_n)$ are chosen to capture the minimal possible energy involved in such a collision, and we define this number to be $4 \pi K $. 

The minimality of $K$ is used crucially to relate the lengths of the collision intervals $|I_n|$ to the largest scale of the bubbles involved in the collision (i.e., those bubbles that concentrate within the discs $D(y_n, \rho_n)$).  We call this largest scale $\lam_{\max, n}$ and the key Lemma~\ref{lem:collision-duration} shows (roughly) that every sequence of collision intervals $I_n$ has  subintervals $J_n$ of length at least
\EQ{
|J_n| \gtrsim \lam_{\max, n}^2, 
} 
on which $u(t)$ bounded away from the multi-bubble family, i.e,. $\bs \de( u(t); D(y_n, \rho_n)) \ge \eps>0$ for all $t \in J_n$, for some $\eps>0$. The intuition behind this is the following. Suppose there were a sequence of intervals $J_n = [s_n, t_n] \subset I_n$ for which the $s_n$'s are bubbling times and the $t_n$'s are ejection times, but  $|J_n|  \ll \lam_{\max, n}^2$. This leads to a contradiction of the minimality of $K$, because the time-interval $J_n$ is too short relative to the scales of the largest bubbles ($\lambda_{\max, n}$) for them to become involved in a collision, and thus collisions are captured on smaller discs $D(\ti y_n, \ti \rho_n) \subset D(y_n, \rho_n)$ with $\ti \rho_n \ll \lam_{\max, n}$, and these carry strictly less energy than $4 \pi K$; see the proof of Lemma~\ref{lem:collision-duration}. 

The fact that $u(t)$ is at least distance $\eps>0$ away from the multi-bubble family on $J_n$ can be combined with the classical localized elliptic bubbling lemma described in Remark~\ref{rem:seq} (the Compactness Lemma~\ref{lem:compact}) to show that on the interval $J_n$, the tension satisfies 
\EQ{
\inf_{t \in J_n}  \lam_{\max, n}^2\| \calT(u(t)) \|_{L^2}^2 \gtrsim 1.
}
 The main point here is that the Compactness Lemma~\ref{lem:compact} says that $u(t)$ bubbles at scale $\lam_{\max, n}$ along any sequence of times $\ti s_n$ for which $\lim_{n \to \infty} \lam_{\max, n}^2\| \calT(u(\ti s_n)) \|_{L^2}^2 = 0$, which is impossible. At this point we have contradicted~\eqref{eq:T} since the previous two displayed equations combine to give  
\EQ{
\sum_n \int_{J_n} \| \calT( u(t)) \|_{L^2}^2 \, \ud t  \gtrsim \sum_n \abs{J_n} \lam_{\max, n}^{-2} \gtrsim \sum_n 1 = \infty. 
}

The idea of a (minimal) collision energy and associated collision time intervals are related to analogous concepts in the first two authors' work on the soliton resolution conjecture for nonlinear waves and on continuous bubbling for the $k$-equivariant HMHF; see~\cite{JL6, JL7, JL8}. 

Theorem~\ref{thm:main} and the Compactness Lemma~\ref{lem:compact} are the main ingredients in the proof of Theorem~\ref{thm:main1}. Again, focusing on the finite time blow-up case, it is well known (see Lemma~\ref{lem:ss-bu}) that energy does not concentrate at or outside the self-similar scale $\sqrt{T_+-t}$, so it suffices to examine the behavior of $u(t)$ restricted to discs $D( y, \sqrt{T_+-t})$,  for $y \in \R^2$ a point where energy concentrates. Let $t_n \to T_+$ be any time sequence. By Theorem~\ref{thm:main}, and after passing to a subsequence, there exists an integer $ \ti M \ge 1$ and a sequence of  $M$-bubble configurations  $\bs \Omega_{n} = (\Om_n, \Om_{1, n}, \dots, \Om_{\ti M, n})$  and sequences $\vec \xi_n, \vec \nu_n$ as in Definition~\ref{def:d} so that 
\EQ{
\bfd( u(t_n), \bs \Omega_n; D(y, \sqrt{T_+-t}); \vec \xi_n, \vec \nu_n) \to 0 \mas n \to \infty. 
}
However, the decomposition in Theorem~\ref{thm:main1} involves a \emph{fixed} collection of finitely many harmonic maps, $\om_1, \dots, \om_M$, i.e., a collection independent of $n$. To find such a collection from the $\Omega_{j, n}$ we apply the Compactness Lemma~\ref{lem:compact} to each $\Om_{j, n}$, obtaining a fixed collection of bubbles $\{\om_{j, k}\}_{k =1}^{L_j}$ for each $j \in \{1, \dots, \ti M\}$. A delicate point is that the scales and centers $(b_{j, k, n}, \mu_{j, k, n})$ associated to the harmonic maps $\om_{j, k}$ given by the Compactness Lemma~\ref{lem:compact} may not satisfy~\eqref{eq:ao-param} for distinct $j$. But this potential pitfall  is remedied by the refined information in Theorem~\ref{thm:main} which says $u(t)$ approaches the multi-bubble family at every smaller scale $\rho_n  \le \sqrt{T_+-t_n}$ (excluding of course the precise scales of the bubbles themselves). 

In Section~\ref{sec:hm-hmhf} we give background information on harmonic maps and the harmonic map heat flow.  Much of Section~\ref{sec:hm-hmhf} is classical, except perhaps the notions of scale and center of harmonic maps and Lemma~\ref{lem:sup}, which involves the propagation of localized $L^\infty$ estimates for solutions to~\eqref{eq:hmhf}, which we did not find a reference for in the literature.   Section~\ref{sec:collisions} contains the proofs of the main theorems.

\subsection{Notational conventions}
Constants are denoted $C, C_0, C_1, c, c_0, c_1$. We write $A \lesssim B$ if $A \leq CB$ and $A \gtrsim B$ if $A \geq cB$.
Given sequences $A_n, B_n$ we write $A_n \ll B_n$ if $\lim_{n\to \infty} A_n / B_n = 0$.

For any sets $X, Y, Z$ we identify $Z^{X\times Y}$ with $(Z^Y)^X$, which means that
if $\phi: X\times Y \to Z$ is a function, then for any $x \in X$ we can view $\phi(x)$ as a function $Y \to Z$
given by $(\phi(x))(y) := \phi(x, y)$.

\section{Preliminaries} \label{sec:hm-hmhf} 

\subsection{Properties of harmonic maps}\label{sec:hmap}
%\subsection{Harmonic maps} 

We use a few well-known features of finite energy harmonic maps $\om: \R^2 \to \Sp^2$ namely, their smoothness, the invariance of harmonicity and the energy under conformal transformations of the domain,  and the fact that the energy is quantized.  

%\Red{Perhaps we remove the proof here, and save this full theorem and the sketched proof for the next paper where we will hopefully do modulation, and there we will presumable need and use all the information we know about sphere valued harmonic maps (conformality). This paper really only requires smoothness of the harmonic maps (Helein's theorem) and the energy quantization, which in principle could hold without conformality I suppose? So we could just state a theorem which contains Helein's result and the energy quantization, and give a precise reference. In any case, I now think we should save this full discussion  done below for the next paper. Thoughts?} 

 \begin{thm}\emph{\cite[Theorem 4.1.1]{Hel}\cite[pg. 126, Proposition]{ES}\cite[Theorem 3.6]{SU}~\cite[Section 8, the Remarque on pg. 65]{Lem}}
  \label{thm:hmap}
  Let $\omega:\R^2\to \bbS^2$ be a weak non-constant solution to \eqref{eq:EL} of finite energy. 
  Then $\omega$ is smooth and extends as a smooth harmonic map from the sphere to itself of nonzero degree, which minimizes the energy~\eqref{eq:energy} in its degree class with  $E(\om)=4\pi|\deg(\om)|$. 
%  It is conformal up to a change of orientation, i.e., the Cauchy-Riemann system (with the upper choice of sign being conformal, the lower anti-conformal)
% \EQ{\label{eq:CR}
%  \abl_1 \omega \mp \omega\times \abl_2 \omega=0,{\ \  equivalently\ \ } \abl_2 \omega \pm \omega\times \abl_1 \omega =0
% }
% holds  and $\omega$ is the unique minimizer of the energy in its homotopy class with $\calE(\om)=4\pi|\deg(\om)|$. There exist $P,Q\in\C[z]$ without common linear factor satisfying $$\max(\deg(P),\deg(Q))=|\deg(\omega)|\ge 1$$  and such that $\omega=\frac{P}{Q}$ for $deg(\omega)>0$, or $\bar{\omega}=\frac{P}{Q}$ for $\deg(\omega)<0$.   
 \end{thm}
 
 \begin{rem} 
 The regularity statement above is due to H\'elein and  holds in the more general setting of weak harmonic maps $\omega \in H^1(\calM, \calN)$ where $\calM$ is a closed, orientable Riemannian surface and $\calN$ is a smooth compact Riemannian manifold. The extension of a smooth, finite energy harmonic map $\om: \R^2 \to \Sp^2$ to a smooth, finite energy harmonic map $\ti \om: \Sp^2 \to \Sp^2$ is a consequence of the conformal equivalence between $\R^2$ and $\Sp^2\setminus \{p_0\}$ via the stereographic projection map, and the Removable Singularity Theorem of Sacks-Uhlenbeck~\cite{SU}.  Here we also use the fact, due to Eells and Sampson \cite{ES}, that in the case of orientable Riemannian surfaces $\calM, \calN$, if $\om: \calM \to \calN$ is smooth and $\phi: \calM \to \calM$ is a conformal diffeomorphism, then $\om\circ \phi$ is harmonic if and only if $\om$ is, and moreover $E( \om) = E(\om\circ \phi)$.  The relationship between the topological degree and the energy (energy quantization) generalizes to harmonic maps between closed, orientable, Riemannian surfaces $\om: \calM \to \calN$, where we have $E( \om) = \mathrm{Area}(\calN)|\deg(\omega)|$, see for example Lemaire~\cite[Section 8, the Remarque on pg. 65]{Lem}. 
 %\Red{also, note that Lemma 2.4 uses conformal invariance, right?  this should be mentioned as it is a general property in the 2d case. Also a proof of the equality for the general version of the energy-degree formula in this remark can also be seen as a consequence of conformal invariance of the energy and the Bogomol'nyi factorization. In any case, we should still think about what we need to say here, and what we cite.}
  \end{rem} 

\subsubsection{The scale and center of a harmonic map}

%In this section we fix a small number $\gamma_0 >0$. 

Given a non-constant harmonic map $\om:\R^2 \to \Sp^2 \subset \R^3$ recall the notion of scale $\lam( \om; \gamma_0)$ and center $a(\om; \gamma_0)$ from Definition~\ref{def:scale} and Defintion~\ref{def:center}.  
%We also use the notation, 
%\EQ{
%\om_{b, \mu}( x) := \om\Big( \frac{ x - b}{\mu}\Big) 
%}

\begin{lem}[Center and scale]  
\label{lem:scale} Let $\gamma_0 \in (0, 2\pi)$, 
let $\om: \R^2 \to \Sp^2 \subset \R^3$ be a non-constant harmonic map, let $\lam(\om) = \lam(\om, \gamma_0)$ be its scale from Definition~\ref{def:scale} and let $a(\om) = a(\om, \gamma_0)$ be a choice of center from Definition~\ref{def:center}. Then $\lambda(\om)$ is uniquely-defined and strictly positive and $a(\omega)$ is well-defined. For all $b \in \R^2$ and $\mu \in (0, \infty)$
\EQ{\label{eq:centscale}
\lambda\big( \om\big( \frac{ \cdot - b}{\mu}\big)\big) = \lambda(\omega)\mu ,   \mand   \Big| a\big( \om\big( \frac{ \cdot  - b}{\mu}\big)\big) - b-a(\omega)\mu\Big| \le 2\lambda(\omega)\mu. 
}
\end{lem} 

\begin{proof} 
Since $E(\omega;D(0,R))\to E(\omega)$ as $R\to\infty$, it follows that the scale $\lam(\om)$ is well-defined. 
If $\lam( \om)=0$, then there exist $a_n\in\R^2$ so that 
\EQ{\label{eq:scale contra}
E( \om; D(a_n, 1/n)) \ge E( \om) - \gamma_0 \qquad \forall\; n\ge1. 
}
If $n\ne m$, the $D(a_n, 1/n)\cap D(a_m, 1/m)\ne\emptyset$. Indeed, otherwise
\[
E(\omega)\ge E( \om; D(a_n, 1/n))+E( \om; D(a_m, 1/m)) \ge 2E( \om) - 2\gamma_0
\]
whence $E(\omega)\le 2\gamma_0< 4\pi$ which contradicts that $\omega$ is not constant. Therefore, $\{a_n\}_{n=1}^\infty$ is a Cauchy sequence in $\R^2$, and $a_n\to a_\infty$. Passing to the limit in~\eqref{eq:scale contra} gives a contradiction.

To see that center $a(\om)$ is well-defined, take $\lambda_n\to \lam(\om)$ and $a_n\in\R^2$ such that 
\[
E( \om; D( a_n, \lam_n)) \ge E( \om) -  \gamma_0
\]
As before, we conclude that no two disks $\{D( a_n, \lam_n)\}_{n=1}^\infty$ can be disjoint. Thus,  $a_n\in\R^2$ lie in a compact set and we may assume that $a_n\to a_\infty$ as $n\to\infty$, which is the desired center. 
We note that $\lam( \om)$ is uniquely defined, but $a(\om)$ is defined only up to a distance of $2\lam( \om)$.  The properties~\eqref{eq:centscale} are immediate from the definitions.  
\end{proof} 

%For the proof, see Section~\ref{sec:scalecenter}. 

%The following lemma requires $\gamma_0$ to be an absolute small constant, independently of the degree.

\begin{lem}[Decay of harmonic maps] 
\label{lem:decay}  There exists $\gamma_0 \in (0, 2 \pi)$ with the following property. For any $0<\gamma \le \gamma_0$ and 
 any harmonic map $\om: \R^2 \to \Sp^2 \subset \R^3$ the exterior energy decays at the following rate:
\EQ{\label{eq:ext_ener_dec}
E(\om; \R^2 \setminus D(a(\om; \gamma); R\lam(\om, \gamma)))  \le  \pi R^{-2} 
}
for all $R \ge 2$. 
%For any $N \ge 1$ there exists a $C = C(N)$ with the following property. Let $\om: \R^2 \to \Sp^2 \subset \R^3$ be a harmonic map with $| \deg ( \om)| \le N$, i.e, such that $E( \om) \le 4 \pi N$. Then, 
%\EQ{
%E(\om; \R^2 \setminus D(a(\om); R\lam(\om)))  \le C R^{-2} 
%}
%for each $R \ge 1$. 
\end{lem} 

We use the following $\eps$-compactness result of Ding and Tian~\cite{DT} in the proof of Lemma~\ref{lem:decay}. 

\begin{lem} [$\eps$-compactness] \emph{\cite[Lemma 2.1]{DT}}  \label{lem:eps-compact}  Let $y \in \R^2$ and let $u:D(y, 1) \to \Sp^2 \subset \R^3$ belong to the class $W^{2, 2}(D(y, 1); \Sp^2)$. Then, there exists $\eps_0>0, C>0$ such that if $E(u; D(y, 1)) < \eps_0$, then
\EQ{
\| u - u_{\operatorname{avg}} \|_{W^{2, 2}(D(y, 1/2))} \le C \Big( \sqrt{E(u; D(y, 1))} +  \| \calT(u) \|_{L^2( D(y, 1))}  \Big)
}
where $u_{\operatorname{avg}}$ denotes the mean of $u$ over the disc $D(y, 1)$. In particular, 
\EQ{
\| u-u_{\operatorname{avg}} \|_{L^\infty( D( y, 1/2))} \le  C \Big( \sqrt{E(u; D(y, 1))} +  \| \calT(u) \|_{L^2( D(y, 1))}  \Big)
} 
\end{lem} 
\begin{proof}[Proof of Lemma~\ref{lem:decay}] 
With loss of generality $a(\omega; \gamma)=0$ and $\lambda(\omega; \gamma)=1$. Consider the harmonic map  $\tilde\omega(z)=\omega(1/z)$
 %\Red{[uses the fact that is $\om$ is harmonic and $\phi$ is conformal the $\om \circ \phi$ is harmonic]} 
 for which $0$ is a removable singularity (see \cite{SU}) and  $E(\tilde\omega;D(0;1))\le\gamma$. Applying Lemma~\ref{lem:eps-compact}, we conclude that
\[
\| D^2\, \tilde\omega\|_{L^2(D(0,1/2))} \les\sqrt{\gamma} 
\]
whence $\tilde\omega\in W^{1,p}(D(0,1/2))$ for any $2<p<\infty$. From the equation $\Delta\tilde\omega+\tilde\omega|\nabla\tilde\omega|^2=0$,  
it follows that $D^2 \tilde\omega \in L^p(D(0,1/2))$ for any $2<p<\infty$. In particular, there exists an absolute constant $\gamma_0>0$ such that $0< \gamma\le \gamma_0$ ensures that 
\[
| \nabla \tilde\omega(z)| \le 1\quad \forall\; |z|\le 1/2, 
\]
whence $E(\omega,\R^2\setminus D(0,1/r))=E(\tilde\omega,D(0,r))\le \pi r^2$ for all $r\le\frac12$. 
\end{proof}

\begin{lem}[Energy of multi-bubbles]  \label{lem:mb-energy}  Let $y_n \in \R^2$, $\rho_n>0$ be sequences, and $M \in \N$. Let $\om_\infty \in \Sp^2$ be a constant, $\om_{1}, \dots, \om_{M}$ be nontrivial harmonic maps, and let $b_{n, j} \in D( y_n, \rho_n)$ and $\mu_{n ,j} \in (0, \infty)$ for $j \in \{1, \dots, M\}$ be sequences so that,  
%\EQ{
%\lim_{n \to \infty} \sum_{j =1}^M E( \om_{j, n}) = 4 \pi K. 
%}
%Denote by $\calQ(\om_{1, n}, \dots, \om_{M, n})$ be the corresponding sequence of multi-bubble configurations. Assume that $a(\om_{j, n}) \in D(y, \rho)$ for each $n \in \N$ and $j \in \{1, \dots, M\}$ and that, 
\EQ{\label{eq:sepcond}
\lim_{n \to \infty} \Bigg[\sum_{j \neq k} \bigg( \frac{\mu_{n, j}}{\mu_{n, k} } +  \frac{\mu_{n, k}}{\mu_{n, j}} &+ \frac{| b_{n, j} - b_{n, k}|}{\mu_{n, j}} \bigg)^{-1} + \sum_{j=1}^{M} \frac{\mu_{n, j}}{ \dist(b_{n, j}, \p D(y_n, \rho_n))} \Bigg] = 0. 
}
Then, 
\EQ{
\lim_{n \to \infty} E\Big( \calQ\big(\om_{\infty}, \om_{1} \big( \frac{\cdot - b_{n, 1}}{\mu_{n, 1}}\big), \dots, \om_{M}\big( \frac{\cdot - b_{n, M}}{\mu_{n, M}}\big)\big); D( y_n, \rho_n)\Big) = \sum_{j =1}^M E( \om_j). 
}
\end{lem} 

\begin{proof}[Proof of Lemma~\ref{lem:mb-energy}]
To simplify notation within the proof, we use the shorthand $\om_{n, j} = \om_j\big( \frac{\cdot - b_{n, j}}{\mu_{n, j}} \big)$. Expanding the energy, we obtain
\EQ{\nonumber
E( \calQ(\om_{n, 1}, \dots, \om_{n, M}); D( y_n, \rho_n)) &= \sum_{j=1}^M E(\omega_{n, j}; D( y_n, \rho_n))  + \frac12\sum_{j\ne k } \int_{D(y_n,\rho_n)} \nabla \omega_{n, j}\nabla \omega_{n, k}. 
}
By the separation condition \eqref{eq:sepcond} with respect to $\partial D(y_n,\rho_n)$ and Lemma~\ref{lem:decay},
\[
E(\omega_{n, j}; D( y_n, \rho_n)) = E(\omega_{n, j}) + o_n(1)
\]
as $n\to\infty$. On the other hand, if $j\ne k$, then 
\EQ{
\Big| \int_{D(y_n,\rho_n)} \nabla \omega_{n, j}\nabla \omega_{n, k} \Big| &\le \int  |\nabla \omega_{n, j}| |\nabla \omega_{n, k}| = o_n(1)
}
by the first term of~\eqref{eq:sepcond}. 
\end{proof}

%\begin{lem} Let $\om_1, \om_2: \R^2 \to \Sp^2$ be non-trivial harmonic maps. Then, 
%\EQ{
%\Big| \int_{\R^2} \na \om_1(x) \na \om_2(x) \, \ud x \Big| >0.
%}
%
%\end{lem} 
%
%\begin{proof} 
%Suppose that, 
%\EQ{
%\int_{\R^2} \na \om_1(x) \na \om_2(x) \, \ud x = 0
%}
%then taking the scalar product of  equation for $\om_1$ with $\om_2$ and integrating by parts over $\R^2$ and vice versa we obtain, 
%\EQ{
%0 = \int_{\R^2} \om_1(x) \cdot \om_2(x) | \na \om_1 (x)|^2\, \ud x = \int_{\R^2} \om_1(x) \cdot \om_2(x) | \na \om_2 (x)|^2\, \ud x
%}
%\end{proof} 
%

\subsection{Properties of the harmonic map heat flow} 
\subsubsection{Well-posedness} 

The starting point for our analysis of the HMHF is the classical result of Struwe~\cite{Struwe85}, which says that the initial value problem is well-posed for data in the energy space and solutions are regular up to their maximal time. 
%Since Theorem~\ref{thm:main} concerns only the asymptotic in time behavior of the flow, we  restrict to smooth initial data for simplicity. 

Following Struwe, we introduce the space, 
\EQ{
\calV_{\tau}^T:= \Big\{& u: [\tau, T] \times \R^2 \to \Sp^2 \subset \R^3 \mid \, u\, \,  \textrm{is measurable, and }   \\
&\quad \sup_{t \in [\tau, T]} E( u(t)) + \int_{\tau}^T \|\p_t u(t)\|_{L^2}^2 + \| \na^2 u(t)\|_{L^2}^2 \, \ud t < \infty \Big\}
}
We use the shorthand $\calV^T = \calV_0^T$. 

%As we are concerned with an asymptotic description of the flow, it suffices for us to consider only smooth initial data. \red{[check to see if the decay condition imposed as $x \to \infty$ below is preserved by the flow? If not, needs a reformulation. I wanted a class of data that includes the harmonic maps]} \Blue{Actually, it seems like Struwe's argument for any finite energy data works equally well in $\R^2$. We can just use it and not worry about limit at $\infty$ it seems... }

\begin{thm}[Local well-posedness]\emph{\cite[Theorem 4.1]{Struwe85}}   \label{lem:lwp} Let $u_0 \in \E$. %  \in \dot{H}^1( \R^2 ; \Sp^2) \cap C^0( \R^2; \Sp^2)$. 
%be a smooth function and suppose there exists $u_0(\infty) \in \Sp^2$ and $C>0$ so that $|u_0(x) - u_0(\infty)| \le C(1 + |x|)^{-1}$. 
Then, there exists a maximal time of existence $T_+= T_+(u_0)$ and a unique solution $u \in \bigcap_{T < T_+} \calV^T$ to~\eqref{eq:hmhf} with $u(0)= u_0$. The solution $u(t)$ is regular (e.g., $C^2$) on the open interval $(0, T_+)$. 
%The solution $u$ is unique and satisfies $u \in \calC^\infty([0, T_+) \times \R^2; \Sp^2)$. 
%and such that $\sup_{x \in \R^2}|u(t, x) - u_0(\infty)| \le C(t)(1 + |x|)^{-1}$ for some function $C(t)$. 

%The maximal time is characterized by the following condition: if $T_+<\infty$, there exists $\eps_0>0$, $R_0>0$ such that 
%\EQ{\label{eq:bu-crit} 
%\limsup_{ t \to T_+} \sup_{x \in \R^2} E(u(t); D(x, R)) \ge \eps_0,
%}
%for all $R \in (0, R_0]$. If there is no such $T_+<\infty$, we say $T_+ = \infty$ and the flow is globally defined.

A finite maximal time $T_+< \infty$ is characterized by the existence of an integer $L \ge 1$, a number $\eps_0>0$, and points $\{x_{\ell}\}_{\ell=1}^L \subset \R^2$ such that 
\EQ{
\limsup_{t \to T_+} E( u(t); D(x_\ell, R))  \ge \eps_0, \quad \forall \, \, R>0, \quad \forall 1 \le \ell \le L . 
}
The $\{x_{\ell}\}_{\ell =1}^L$ are called bubbling points and there are at most finitely many. There exists a finite energy mapping $u^*: \R^2 \to \Sp^2$, called the body map, such that $u(t) \rightharpoonup u^*$ as $t \to T_+$ weakly in $H^1( \R^2; \Sp^2)$ and strongly in $H^1_{\loc}( \R^2 \setminus \{x_{\ell}\}_{\ell =1}^L; \Sp^2)$.   

The energy $E(u(t))$ is continuous and non-increasing as a function of $t \in [0, T_+)$,  and for any $t_1 \le t_2 \in [0, T_+)$,  there holds 
\EQ{ \label{eq:energy-identity} 
E(u(t_2)) +  \int_{t_1}^{t_2} \| \calT(u(t)) \|_{L^2}^2\, \ud t = E(u(t_1)). 
}
In particular, there exists $E_+:= \lim_{t \to T_+} E( u(t))$, and 
\EQ{ \label{eq:tension-L2} 
\int_0^{T_+}\| \calT(u(t)) \|_{L^2}^2\, \ud t < \infty. 
}
%If $T_+< \infty$, there exists $\eps_0>0$ such that, 
%\EQ{\label{eq:bu-crit} 
%\limsup_{ t \to T_+} E(u(t); 0, r_0) \ge \eps_0,
%}
%for all $r_0>0$ and the limit  
\end{thm} 

\begin{rem} 
Lemma~\ref{lem:lwp} is proved by Struwe for the HMHF in the case of maps from a closed Riemannian surface $\calM$ to a compact Riemannian manifold  $\calN$; see~\cite[Theorem 4.1]{Struwe85}. The same arguments hold for $\calM= \R^2$.  
%For the case of maps from $\R^2$ we refer the reader to Lin and Wang~\cite[Theorem 5.2.1]{Lin-Wang} for the short time existence of  regular solutions, and to Smith's thesis~\cite{Smith} for the energy concentration properties from Struwe's work~\cite{Struwe85} adapted to the domain $\R^2$. 
\end{rem}

\begin{lem}[Localized energy inequality]  \label{lem:Struwe}  There exists a constant $C>0$ with the following property. Let $u(t)$ be a solution to~\eqref{eq:hmhf} with initial data $u_0$ as in Lemma~\ref{lem:lwp}, on its maximal interval $I_{\max} = [0, T_+)$. Let $0 < t_1 < t_2 < T_+$. Let $R>0$, $\phi \in C^{\infty}_0(\R^2)$ satisfy $0 \le \phi(x) \le 1$ and $\abs{\na \phi} \le R^{-1}$. Then, 
\EQ{ \label{eq:Struwe-forward} 
\int_{\R^2} | \na u(t_2, x)|^2 \phi(x)^2 \, \ud x \le  \int_{\R^2} | \na u(t_1, x)|^2 \phi(x)^2 \, \ud x + C E( u_0) \frac{t_2 - t_1}{R^2} 
}
and, 
\EQ{ \label{eq:Struwe-back} 
\int_{\R^2} | \na u(t_2, x)|^2 \phi(x)^2 \, \ud x  \ge \int_{\R^2} | \na u(t_1, x)|^2 \phi(x)^2 \, \ud x - C \Big(E( u_0) \frac{(t_2 - t_1)}{R^2}  + | E( u(t_1)) - E(u(t_2))|\Big)
}
\end{lem} 

\begin{proof}[Proof of Lemma~\ref{lem:Struwe}] 
Take the dot product of the equation~\eqref{eq:hmhf} with $\p_t u \phi^2$ and integrate by parts to obtain the identity, 
\EQ{ %\label{eq:loc-en-id} 
\| \p_t u(t)  \phi \|_{L^2}^2 + \frac{1}{2}  \frac{\ud}{\ud t}  \| \na u(t) \phi \|_{L^2}^2  = - \sum_{j =1}^2 \int_{\R^2}   \p_j u(t, x)  \cdot \p_t u(t, x)  \p_j \phi(x) \phi(x) \, \ud x. 
}
Integrating the above from $t_1$ to $t_2$ we obtain the identity, 
\EQ{ \label{eq:loc-en-id} 
\int_{t_1}^{t_2} \| \p_t u(t)  \phi \|_{L^2}^2 \, \ud t + \frac{1}{2}  \| \na u(t_2) \phi \|_{L^2}^2 &=  \frac{1}{2}  \| \na u(t_1) \phi \|_{L^2}^2 \\
&\quad - \int_{t_1}^{t_2}\sum_{j =1}^2 \int_{\R^2}   \p_j u(t, x)  \cdot \p_t u(t, x)  \p_j \phi(x) \phi(x) \, \ud x, \ud t 
}
The right-hand side above is bounded by 
\EQ{
\int_{t_1}^{t_2} \Big| \sum_{j =1}^2 \int_{\R^2}   \p_j u(t, x)  \cdot \p_t u(t, x)  \p_j \phi(x) \phi(x) \, \ud x \, \ud t  \Big|  \lesssim  \int_{t_1}^{t_2}\frac{ \sqrt{E( u_0)}}{R}  \| \p_t u(t)  \phi \|_{L^2}\, \ud t 
}
The lemma readily follows after an application of Cauchy Schwarz, where we note that in obtaining~\eqref{eq:Struwe-back} we also make use of the energy identity~\eqref{eq:energy-identity}. 
\end{proof}

\begin{lem} \label{lem:Struwe-lwp}  
Let $u_n(t)$ be a sequence of HMHFs with initial data $u_{n, 0} \in \E$ defined on time intervals $I_n :=[0, \tau_n]$ for a sequence $\tau_n >0$ with $\lim_{ n \to \infty} \tau_n  = 0$, and satisfying $\limsup_{n \to \infty}E( u_{n, 0}) < \infty$. Let $\om$ be a harmonic map and let $R_n>0$ be a sequence such that $\lim_{n \to \infty} \tau_n R_n^{-2} =0$. Suppose that, 
\EQ{
\lim_{n \to \infty} E(u_{n, 0} - \om; D(0, 2R_n))  = 0. 
}
Then, 
\EQ{ \label{eq:en-discR} 
\lim_{n  \to \infty} E( u_n(\tau_n) - \om; D(0,R_n))= 0. 
}
Next, let $\eps_n>0$ be a sequence with $ \eps_n < R_n$ for all $n$ and such that $\lim_{n \to \infty} \tau_n\eps_n^{-2} = 0$. Let $L \ge 1$ be an integer and let $\{x_\ell\}_{\ell =1}^L \subset \R^2$ be such that the discs $D(x_\ell,  \eps_n)$ are disjoint and satisfy $D(x_{\ell},  \eps_n) \subset D(0, R_n)$ for each $n$ and each $\ell \in \{1, \dots, L\}$.  Moreover,  $|x_\ell-x_m|\ge 100\eps_n$ if $\ell\ne m$. Suppose that, 
\EQ{
&\lim_{n \to \infty} E\Big(u_{n, 0} - \om; D(0, 2R_n) \setminus \bigcup_{\ell =1}^L D(x_{\ell} , 2^{-1}\eps_n)\Big)  = 0. %\\
%&\lim_{n \to \infty} E( u_{n, 0} ; D(0, 2\eps_n))\setminus D(0, \frac{1}{2}\eps_n)) = 0.
}
Then, 
\EQ{ \label{eq:en-annulus} 
\lim_{n \to \infty} E\Big(u_{n}(\tau_n) - \om; D(0, R_n) \setminus  \bigcup_{\ell =1}^L D(x_{\ell}, \eps_n)\Big)  = 0. 
}
%provided that $\lim_{n \to \infty} \de_n^{\frac{1}{2}}\eps_n^{-1} = 0$. 
%\red{also need a simpler version on a disc rather than an annulus for~\eqref{eq:t_n-bubbles}}
%\Blue{here we need something that can be used to prove~\eqref{eq:t_n-bubbles} as well as the displayed estimate above~\eqref{eq:tn-inside-Rn}}
\end{lem} 

%\begin{proof}[Proof of Lemma~\ref{lem:Struwe}]
%\end{proof} 
%
\begin{proof}[Proof of Lemma~\ref{lem:Struwe-lwp}] 
The proof is very similar to the Proof of Lemma~\ref{lem:Struwe}. We prove the estimate~\eqref{eq:en-annulus} as the proof of~\eqref{eq:en-discR} is analogous.  Set  $v_n(t):= u_n(t) - \om$. Then, 
\EQ{
\p_t v_n - \De v_n = u_n | \na u_n|^2 - \om| \na \om|^2
} 
Let $\phi_n \in C^\infty_0( \R^2)$ and take the dot product of the above with $\p_t v_n  \phi_n^2$. Recalling that $\p_tv_n = \p_t u_n \perp u_n$, integrating by parts, and integrating in time from $0$ to $\tau_n$,  we obtain the inequality, 
\EQ{ \label{eq:en-diff-ineq} 
\frac{1}{2} \| \na v_n(\tau_n)  \phi_n \|_{L^2}^2  &+ \int_0^{\tau_n}  \| \p_t v_n (t) \phi_n \|_{L^2}^2 \, \ud t  \le \frac{1}{2} \| \na u_{n, 0} \phi_n \|_{L^2}^2\\
&+ \int_0^{\tau_n} \int_{\R^2}  | \na v_n(t)| | \p_t v_n(t) | |\na \phi_n| \phi_n\, \ud x \, \ud t +  \int_{0}^{\tau_n} \int_{\R^2} | \na \om|^2  | \p_t v_n| \phi_n^2 \, \ud x \, \ud t. 
}
Now, let $\phi_n$ be cutoffs supported in the region $D(0, 2R_n) \setminus\bigcup_{\ell =1}^L D(x_{\ell} , 2^{-1}\eps_n))$ and $= 1$ in the region $D(0, R_n) \setminus \bigcup_{\ell =1}^L D(x_{\ell} , 2^{-1}\eps_n))$, satisfying the bound $| \na \phi_n| \lesssim \eps_n^{-1}$. The first term of the last line above satisfies, 
\EQ{
\int_0^{\tau_n} \int_{\R^2}  | \na v_n(t)| | \p_t v_n(t) | |\na \phi_n| \phi_n\, \ud x \, \ud t  \lesssim  ( E( u_{n, 0}) + E( \om)) \frac{ \tau_n}{\eps_n^2}  + \frac{1}{2} \int_0^{\tau_n}  \| \p_t v_n (t) \phi_n \|_{L^2}^2 \, \ud t
} 
and the second term on the right above can be absorbed into the left-hand side of~\eqref{eq:en-diff-ineq}. Similarly, 
\EQ{
 \int_{0}^{\tau_n} \int_{\R^2} | \na \om|^2  | \p_t v_n| \phi_n^2 \, \ud x \, \ud t \lesssim   \tau_n  \| \na \om \|_{L^4}^4 + \frac{1}{2} \int_0^{\tau_n}  \| \p_t v_n (t) \phi_n \|_{L^2}^2 \, \ud t
}
and the second term on the right above can be absorbed into the left-hand side of~\eqref{eq:en-diff-ineq}. The limit~\eqref{eq:en-annulus} readily follows. 
\end{proof} 
%
%\subsection{Variational properties of equivariant bubbles}

\subsubsection{Local $L^\infty$ estimates for the heat flow} 

We use the notation $e^{t \De}$ to denote the heat propagator in $\R^2$, i.e., 
\EQ{
e^{t \De}   v (x):= \frac{1}{4 \pi t} \int_{\R^2} e^{-\frac{|x-y|^2}{4t}} v(y) \, \ud y,  
}
where here $v: \R^2 \to \R^3$.   We also recall Duhamel's formula, 
\EQ{
v(t) = e^{t \De} v(0) + \int_0^t e^{(t-s) \De} ( \p_s v(s) - \De v(s)) \, \ud s.
}

\begin{lem}[Parabolic Strichartz estimates]  \emph{\cite[Lemma 2.5]{Tao4}} There exists a constant $C_0>0$ with the following property.  Let $v_0 \in L^2(\R^2; \R^3)$. Let $T>0$,  $I:= [0, T]$  and let $F \in L^1([0, T]; L^2(\R^2; \R^3))$. Let $v(t)$ denote the unique solution to the linear heat equation
\EQ{
\p_t v - \De v &= F \\ v(0) &= v_0. 
}
Then, 
\EQ{\label{eq:StHeat}
\| v \|_{L^2(I; L^\infty(\R^2; \R^3))} \le C_0 \big( \| v_0 \|_{L^2} + \| F \|_{L^1(I; L^2(\R^2; \R^3))} \big) 
}
\end{lem} 
\begin{proof}
Setting $(Tf)(t):=e^{t\Delta}f$ for $t\ge0$ one has $T^* F=\int_0^\infty e^{s\Delta} F(s)\, ds$. 
    Starting from the two-dimension estimate $\|(Tf)(t)\|_\infty \les t^{-1} \|f\|_1$, we have 
    \[
    (T T^* F)(t) =\int_0^\infty e^{(t+s)\Delta}F(s)\, ds
    \]
    whence 
    \[
    \|(T T^* F)(t)\|_\infty \les \int_0^\infty (t+s)^{-1} \|F(s)\|_1\, ds = \int_0^\infty (1+u)^{-1} \|F(u t)\|_1\, du
    \]
    The right-hand side is  $L^2_t((0,\infty))$ bounded and we conclude that 
    \[
    \| T T^* F\|_{L^2((0,\infty),L^\infty(\R^2))} \les \|F\|_{L^2((0,\infty),L^1(\R^2))}
    \]
    Thus,  $\langle TT^*F,F\rangle=\|T^*F\|_2^2\les \|F\|_{L^2((0,\infty),L^1(\R^2))}^2$, and by duality, we obtain the $F=0$ case of~\eqref{eq:StHeat}, viz. $\|Tv_0\|_{L^2((0,\infty),L^\infty(\R^2))}\les \|v_0\|_2$. On the other hand, if $v_0=0$, then 
    \[
    v(t) = \int_0^t e^{(t-s)\Delta} F(s)\, ds = \int_0^\infty \chi_{[s\le t]}\, e^{(t-s)\Delta}  F(s)\, ds
    \]
    whence 
    \EQ{
   \| v(t) \|_{L^2((0,\infty),L^\infty(\R^2))} &\le \int_0^\infty \| e^{(t-s)\Delta}  F(s)\|_{L^2((s,\infty),L^\infty(\R^2))}\, ds
   \les \int_0^\infty \|  F(s)\|_{L^2(\R^2)}\, ds
    }
    as claimed. 
\end{proof}

\begin{lem} \label{lem:sup} Let $u_n(t)$ be a sequence of solutions to~\eqref{eq:hmhf} with initial data $u_{n, 0} \in \E \cap C^0( \R^2; \R^3)$ and  $\limsup_{n \to \infty} E( u_{n, 0})  < \infty$, defined on time intervals $I_n :=[0, \tau_n]$ for a sequence $\tau_n >0$ with $\lim_{n \to \infty} \tau_n = 0$.  Let $\om$ be a harmonic map (possibly constant) and let $R_n>0$ be a sequence so that $\lim_{n \to \infty} \tau_n R_n^{-2} = 0$. Suppose that, 
\EQ{
\lim_{n \to \infty} \Big(\| u_{n, 0} - \om \|_{L^{\infty}(D(0, 4R_n))} +E(u_{n, 0} - \om; D(0, 4R_n))\Big)  = 0, 
}
then, 
\EQ{
\lim_{n \to \infty} \| u_{n}(\tau_n) - \om \|_{L^\infty(D(0, R_n))} = 0.
}
Next, let $\eps_n>0$ be a sequence with $\eps_n < R_n$ for all $n$ and such that $\lim_{n \to \infty} \tau_n\eps_n^{-2} = 0$. Let $L \ge 1$ be an integer and let $\{x_\ell\}_{\ell =1}^L \subset \R^2$ be such that the discs $D(x_\ell,  \eps_n)$ are disjoint and satisfy $D(x_{\ell},  \eps_n) \subset D(0, R_n)$ for each $n$ and each $\ell \in \{1, \dots, L\}$.  Moreover,  $|x_\ell-x_m|\ge 100\eps_n$ if $\ell\ne m$.  Suppose that, 
\EQ{
&\lim_{n \to \infty} \Big( \| u_{n, 0} - \om \|_{L^\infty(D(0, 4R_n) \setminus \bigcup_{\ell =1}^L D(x_{\ell} , 4^{-1}\eps_n))}+E\Big(u_{n, 0} - \om; D(0, 4R_n) \setminus \bigcup_{\ell =1}^L D(x_{\ell} , 4^{-1}\eps_n)\Big)  \Big) = 0. %\\
%&\lim_{n \to \infty} E( u_{n, 0} ; D(0, 2\eps_n))\setminus D(0, \frac{1}{2}\eps_n)) = 0.
}
Then, 
\EQ{ \label{eq:sup-cheese} 
\lim_{n \to \infty} \| u_{n}(\tau_n) - \om \|_{L^\infty(D(0, R_n) \setminus \bigcup_{\ell =1}^L D(x_{\ell} , \eps_n))} = 0. 
}
\end{lem} 
\begin{proof}
We begin with a solution $u_n(t)$ of the heat flow satisfying
\[
\| u_n(0) -\omega\|_{L^\infty(D(0,4R_n))}+ E(u_n(0)-\omega, D(0,4R_n)) =o_n(1).
\]
We pick $\phi_1$ to be the ground state of the Dirichlet Laplacian on the disk $D(0,3R_n/2)$ and set $v(t)=(u(t)-\omega)\phi_1$, dropping the index~$n$ for simplicity. We normalize $\phi_1(0)=1$ which means that $\|\phi_1\|_\infty=1$. Then $-\Delta \phi_1=\lambda_1^2\phi_1$, $\lambda_1^2\simeq R_n^{-2}$ and 
\[
\partial_t v -\Delta v = (|\nabla u|^2 u - |\nabla\omega|^2\omega)\phi_1 -2\nabla(u-\omega)\nabla\phi_1 + \lambda_1^2 v.
\]
Since $\phi_1$ is not globally smooth, we cannot solve this heat equation on the plane but rather need to use the heat flow on the region $\Omega=\overline{D(0,3R_n/2)}$ with Dirichlet boundary conditions. By the Beurling-Deny theorem, see Davies~\cite[Theorem~1.3.5]{DHeat}, the heat flow is a contraction on  $L^\infty(\Omega)$ and we conclude that 
\EQ{\label{eq:absch1}
\max_{0\le t\le T}\|v(t)\|_\infty &\le \|v(0)\|_\infty + \int_I  \| (|\nabla u(s)|^2 u(s) - |\nabla\omega|^2\omega)\phi_1\|_\infty \,  \ud s \\
&\qquad + 2\int_I  \| \nabla (u(s) -\omega) \nabla \phi_1\|_\infty \, \ud s + \lambda_1^2 |I| \|v\|_{L^\infty(I\times\R^2) }
}
with $[0,\tau_n]=I$. 
If $\lambda_1^2 |I|\le\frac12 $, then the final term gets absorbed to the left-hand side. 
Next,
\EQ{ 
\| (|\nabla u(s)|^2 u(s) - |\nabla\omega|^2\omega)\phi_1\|_\infty & \le 2(\| \phi_2 \nabla (u(s)-\omega)\|_\infty^2 +\|\phi_2 \nabla \omega\|_\infty^2)\|v(s)\|_\infty \\
&+ \|\omega\|_\infty \|\phi_2\nabla(u-\omega)\|_\infty(\|\phi_2\nabla (u-\omega)\|_\infty + 2\|\phi_2\nabla \omega\|_2)
}
where $\phi_2$ is a smooth cutoff to $D(0,2R_n)$ with $\phi_1\phi_2=\phi_1$. We further bound 
\[
\| \nabla (u(s) -\omega) \nabla \phi_1\|_\infty \les R_n^{-1}\|  \phi_2\nabla (u(s) -\omega) \|_\infty
\]
using that $\|\nabla\phi_1\|_\infty \les R_n^{-1}$, which follows by scaling. 
Define $w(s):= \phi_2\nabla (u(s) -\omega)  $ and let $X:=L^\infty(I; L^\infty(\R^2))$, $Y:=L^2(I; L^\infty(\R^2; \R^3))$. Then~\eqref{eq:absch1} implies that 
\[
\|v\|_X \les o(1) +\tau_n + \|v\|_X (\|w\|_Y^2+\tau_n)+ \|w\|_Y^2 + \sqrt{\tau_n R_n^{-2}} \|w\|_Y,
\]
which in turn simplifies to 
\EQ{\label{eq:absch2}
\|v\|_X &\les o(1) + (\|v\|_X +1)\|w\|_Y^2 .
}
To bound $w$ we use the PDE
\EQ{
\partial_t w -\Delta w &= \phi_2 \nabla (u|\nabla u|^2) - 2\sum_{j=1}^2 \nabla \partial_j (u-\omega) \partial_j \phi_2 \\
& \qquad - \nabla(u-\omega)\Delta \phi_2 - \phi_2 \nabla(\omega |\nabla \omega|^2) =: G.
}
By \eqref{eq:StHeat}, and with $Z:=L^1(I; L^2(\R^2; \R^3)$, 
\EQ{\label{eq:absch3}
\| w\|_{Y} \les \|\nabla (u(0) -\omega)  \phi_2\|_2 + \| G\|_{Z}\les o(1) + \| G\|_{Z}.
}
To bound $G$, we estimate with a smooth cutoff $\phi_3$ to $D(0,3R_n)$ so that $\phi_2\phi_3=\phi_2$, 
\EQ{
\| \phi_2 \nabla (u|\nabla u|^2) \|_2 &\les \| \phi_2\nabla u\|_\infty \| \phi_3 \nabla u \|_4^2 + \|\phi_2\nabla u\|_\infty \|\phi_3 D^2u \|_2 \\
\| \nabla \partial_j (u-\omega) \partial_j \phi_2 \|_2 &\les R_n^{-1}\| \phi_3 D^2 (u-\omega) \|_2 \\
\|\Delta \phi_2 \nabla(u-\omega)\|_2 &\les R_n^{-2}\|\phi_3 \nabla(u-\omega)\|_2
}
as well as $\| \phi_2 \nabla(\omega |\nabla \omega|^2)\|_2 \les 1$. Furthermore, 
\EQ{
\| \phi_2\nabla u\|_\infty  &\les \| w\|_\infty +1 \\
\| \phi_3 \nabla u \|_4 &\les \| \phi_3 \nabla (u-\omega) \|_4 + 1\\
\|\phi_3 D^2u \|_2  &\les \|\phi_3 D^2 (u-\omega) \|_2  +1
}
uniformly in $R_n$. 
By \eqref{eq:absch3} therefore
\EQ{\label{eq:absch4}
\| w\|_{Y} &\les  o(1) +\| w\|_{Y}^2 + \int_I \|\phi_3\nabla(u(s)-\omega)\|_4^4\, \ud s  
  + \int_I \|\phi_3 D^2(u(s)-\omega)\|_2^2\, \ud s .
}
To perform energy estimates on $u-\omega$ we apply the methods of Struwe~\cite{Struwe85} to the PDE
\EQ{\label{eq:uomPDE}
\partial_t(u-\omega) -\Delta (u-\omega) = (u-\omega)|\nabla\omega|^2 + u(|\nabla u|^2-|\nabla \omega|^2)
}
Integrating by parts against $\phi_3^2 \partial_t(u-\omega)$ implies that (with $T=\tau_n$)
\EQ{\label{eq:ptuom1}
& \int_0^T \int_{\R^2} |\partial_t(u-\omega)|^2 \phi_3^2 \, \ud x \ud t+ \int_0^T\int_{\R^2} \nabla(u(t)-\omega) \cdot \nabla[\partial_t (u-\omega)\phi_3^2]\, \ud x \ud t \\
&= \int_0^T \int_{\R^2} (u-\omega)|\nabla\omega|^2 \cdot\partial_t (u-\omega)\phi_3^2\, \ud x \ud t + \int_0^T \int_{\R^2} u\cdot \partial_t (u-\omega)\phi_3^2(|\nabla u(t)|^2-|\nabla \omega|^2) \, \ud x \ud t ,
}
which implies 
\EQ{\label{eq:ptuom2}
& \int_0^T \int_{\R^2} |\partial_t(u-\omega)|^2 \phi_3^2 \, \ud x \ud t+ \int_{\R^2} |\nabla(u(T)-\omega)|^2\phi_3^2\, \ud x  \\
&\les o(1) + \int_0^T \int_{\R^2} ||\nabla u(t)|^2-|\nabla \omega|^2|^2\,\phi_3^2\, \ud x \ud t+ \int_0^T \int_{\R^2} |\nabla (u(t)-\omega)|^2 |\nabla\phi_3|^2\, \ud x \ud t \\
&\les o(1) + \int_0^T \int_{\R^2} |\nabla (u(t)- \omega)|^4  \,\phi_3^2\, \ud x \ud t
}
The final term on the second line of \eqref{eq:ptuom2} is dominated by $TR_n^{-2} (E(u(0))+E(\omega))$, and so can be absorbed in the $O(\tau_nR_n^{-2})$.  Multiplying~\eqref{eq:uomPDE} by $-\phi_3^2\Delta(u-\omega)$ and integrating by parts yields
\EQ{\label{eq:ptuom3}
& \sum_{j=1}^2 \int_0^T \int_{\R^2} \partial_t |\partial_j(u-\omega)|^2 \phi_3^2  \, \ud x \ud t+ \int_0^T\int_{\R^2} |\Delta(u(t)-\omega)|^2 \phi_3^2\, \ud x \ud t \\
&= \sum_{j=1}^2  \int_0^T \int_{\R^2} \partial_j[(u-\omega)|\nabla\omega|^2\phi_3^2] \cdot\partial_j (u-\omega)\, dxdt -  2\sum_{j=1}^2\int_0^T \int_{\R^2} \partial_t (u-\omega) \phi_3 \partial_j \phi_3 \cdot\partial_j (u-\omega) \, \ud x \ud t \\
&\qquad -   \int_0^T \int_{\R^2}  \Delta (u-\omega) \cdot [\phi_3^2(|\nabla u(t)|^2-|\nabla \omega|^2)u] \, \ud x \ud t,
}
which implies 
\EQ{\label{eq:ptuom4}
& -\frac12  \int_0^T \int_{\R^2} |\partial_t(u-\omega)|^2 \phi_3^2 \, \ud x \ud t+ \int_{\R^2} |\nabla(u(T)-\omega)|^2\phi_3^2\, dx + \frac12\int_0^T\int_{\R^2} |\Delta(u(t)-\omega)|^2 \phi_3^2\, \ud x \ud t \\
&\les o(1)+ \int_0^T \int_{\R^2} |\nabla (u(t)- \omega)|^4  \,\phi_3^2\, \ud x \ud t.
}
Adding this to \eqref{eq:ptuom2} we obtain 
\EQ{\label{eq:ptuom5}
&   \int_0^T \int_{\R^2} |\partial_t(u-\omega)|^2 \phi_3^2 \, \ud x \ud t+ \int_{\R^2} |\nabla(u(T)-\omega)|^2\phi_3^2\, dx + \int_0^T\int_{\R^2} |D^2(u(t)-\omega)|^2 \phi_3^2\, \ud x \ud t\\
&\les o(1) + \int_0^T \int_{\R^2} |\nabla (u(t)- \omega)|^4  \,\phi_3^2\, \ud x \ud t.
}
For the third term on the left-hand side we used another integration by parts to bound
\[
\int_0^T\int_{\R^2} |D^2(u(t)-\omega)|^2 \phi_3^2\, dx dt \les \int_0^T\int_{\R^2} |\Delta(u(t)-\omega)|^2 \phi_3^2\, \ud x \ud t
+\tau_n R_n^{-2}.
\]
By \cite[Lemma~3.2]{Struwe85}, 
\EQ{\label{eq:4D2u}
\int_0^T \int_{\R^2} |\nabla (u(t)- \omega)|^4  \,\phi_3^2\, dxdt &\les \sup_{0\le t\le T} \int_{D(0,3R_n)} |\nabla (u(t,x)-\omega(x))|^2\, \ud x  \\
&\qquad  \Big(\int_0^T\int_{\R^2} |D^2(u(t)-\omega)|^2 \phi_3^2\, dx dt + \tau_n R_n^{-2}\Big).
}
By Lemma~\ref{lem:Struwe-lwp}, the local energy on $D(0,3)$ is small. Hence we conclude from~\eqref{eq:ptuom5} and the bound~\eqref{eq:4D2u} that 
\EQ{
 & \int_0^T \int_{\R^2} |\partial_t(u-\omega)|^2 \phi_3^2 \, \ud x \ud t+ \int_{\R^2} |\nabla(u(T)-\omega)|^2\phi_3^2\, dx + \int_0^T\int_{\R^2} |D^2(u(t)-\omega)|^2 \phi_3^2\, \ud x \ud t \\
&+\int_0^T \int_{\R^2} |\nabla (u(t)- \omega)|^4  \,\phi_3^2\, \ud x \ud t \les o(1).
}
Inserting this bound into~\eqref{eq:absch4} yields $\|w\|_Y=o(1)$, whence from \eqref{eq:absch2}, finally $\|v\|_X=o(1)$. This finishes the proof for disks. 

For the punctured disks we would like to proceed in the same fashion. As a first step, it appears that we would need to obtain bounds on  the suitably normalized ground state eigenfunction $\phi_1$ of the  Dirichlet Laplacian on the punctured disk $$D^*(0,4R_n)=D(0,4R_n)\setminus \bigcup_{\ell=1}^L D(x_\ell,\eps_n/4)$$ 
where $x_\ell$ are as stated in the lemma. This turns out to be misguided as we will now see, in addition to being delicate in terms of obtaining the needed bounds on $\phi_1$ uniformly in the choice of holes.  In fact, it suffices to select $\phi_1$ smooth on $\Omega=\overline{D^*(0,4R_n)}$, vanishing on $\partial\Omega $ so that $-\Delta \phi_1 = V\phi_1$ with $\|V\|_{L^\infty(\Omega)}\les \eps_n^{-2}$ uniformly in all parameters. By rescaling, it will also suffice to set $R_n=1$. 

We define, with $r=|x|$ and $r_\ell=|x-x_\ell|$,
\EQ{\label{eq:phi1prod}
\phi_1(x):= \chi_0(r) \prod_{\ell=1}^L \chi_\ell(r_\ell)
}
with smooth functions $\chi_j>0$ on the interior of $\Omega$, $0\le j\le L$ that we now specify. 
First, $\chi_0(r)=1$ for $r\le 2$, and on the annulus $3\le r\le 4$, $\chi_0(r)$ agrees with the $L^\infty$-normalized Dirichlet ground state of the disk~$D(0,4)$.  To define $\chi_\ell$, consider {\em all radial} Dirchlet eigenfunctions,  $\{\psi_n\}_{n=1}^\infty$ on the annulus $D(0,1)\setminus D(0,\gamma)$. Then $\psi_n(r)=a_nJ_0(\mu_n r)+b_nY_0(\mu_n r)$ where $\mu_n^2$ is the eigenvalue for $n\ge1$  and $a_n,b_n\in\R$. Since $a_n^2+b_n^2>0$, and $\psi_n(\gamma)=\psi_n(1)=0$ the spectrum is characterized by the conditions
\[
J_0(\mu_n \gamma)Y_0(\mu_n)- Y_0(\mu_n \gamma)J_0(\mu_n)=0
\]
Note that the ratio $R(x):=Y_0(x)/J_0(x)$ is strictly increasing on the interval $(0,\rho_1)$ where $\rho_1>0$ is the smallest positive zero of~$J_0$, as well as on any subsequent interval $(\rho_j,\rho_{j+1})$, $j\ge1$. This follows from the fact that the Wronskian $Y_0'(x)J_0(x)-Y_0(x)J_0'(x)=2/(\pi x)>0$. The first crossing of the graphs, which determines the smallest energy $\mu_1>0$, is determined by $R(x)=R(\gamma x)$. The expansion $R(x)=\frac{2}{\pi}\log x+O(1)$ for $x\to0+$ shows that $R(x)>R(x\gamma)$ for all $0<x<\rho_1$, and the first crossing occurs at $x\in (\rho_1,\rho_2)$ and so the ground state energy $\mu_1\in (\rho_1,\rho_2)$. Similarly, we find the other energies $\mu_j\in  (\rho_j,\rho_{j+1})$, for $j\ge1$. We select an eigenfunction $\psi_k$ with $\mu_k\simeq \gamma^{-1}$. This is possible due to the zeros of $J_0(x)$ forming, to leading order, an arithmetic progression. 

We can now define $\chi_\ell$ in~\eqref{eq:phi1prod} by centering this $\psi_k$ for $\gamma=\eps_\ell$ at $x_\ell$, and gluing it smoothly with the constant~$1$ at a distance $\simeq \eps_\ell$ away from the $x_\ell$ hole. The resulting function~$\chi_\ell>0$ will then satisfy $\chi_\ell(x)\ge \frac12$ provided $|x-x_\ell|\gtrsim \eps_\ell$.   

The associated cutoff function $\phi_1$ satisfies 
\EQ{\label{eq:growneck}
\phi_1(x) \ge c_0 \|\phi_1\|_\infty \text{\ \ for all \ \ } x\in D(0,3R_n)\setminus \bigcup_{\ell=1}^L D(x_\ell,\eps_n/2)
}
with some absolute constant $c_0>0$, independently of $R_n, \eps_n$ and the choice of the centers $x_\ell$ as above. It is clear from the preceding  that the ground state would have energy~$\simeq 1$ and does not satisfy~\eqref{eq:growneck}.  Furthermore, $\|\nabla\phi_1\|_\infty\les\eps_n^{-1}\| \phi_1\|_\infty$ and most importantly, 
\[
\|\phi_1^{-1}\Delta\phi_1\|_\infty \le \sum_{\ell=0}^L \| \chi_\ell^{-1}\Delta \chi_\ell\|_\infty \les \eps_n^{-2}
\]
as desired. 
Because of our assumption $\tau_n \eps_n^{-2}\to 0$ as $n\to\infty$, the proof above applies. In the final step we use~\eqref{eq:en-annulus} to control the $L^4$-norm as before, with one modification: we apply~\cite[Lemma 3.2]{Struwe85} locally on $\eps_n$-disks, and then cover the punctured disk~$D^*(0,4R_n)$ with $\eps_n$-disks followed by a summation over the disks in the cover. Cf.~\cite[Lemma 3.1, 3.3]{Struwe85}. 
\end{proof}

\subsubsection{Concentration properties of the heat flow} 

Here we record the fact that the harmonic map heat flow cannot concentrate energy at the self-similar scale. The case of finite time blow up was treated by Topping in~\cite{Topping-winding} and the global in time case follows quickly from a local energy inequality as in Lemma~\ref{lem:Struwe}.

\begin{lem}[No self-similar concentration in the blow-up case]\emph{\cite[Proof of Theorem 1.6, page 288]{Topping-winding}}  \label{lem:ss-bu} Let $u(t)$ be the solution to~\eqref{eq:hmhf}  with maximal time of existence $T_+< \infty$ and initial data $u_0 \in \E$. Let $x_0 \in \R^2$ denote a bubbling point in the sense of Lemma~\ref{lem:lwp} and suppose that $r>0$ is sufficiently small so that $D(x_0, r)$ does not contain any other bubbling point. Then, 
\EQ{
\lim_{t \to T_+} E( u(t); D( x_0, r) \setminus D( x_0, \alpha \sqrt{T_+ - t})) = E( u^*; D(x_0, r))
}
for any $\al >0$, where $u^*$ is as in Lemma~\ref{lem:lwp}. In particular,  there exist $T_0 < T_+$ and functions $\nu, \xi: [T_0, T_+) \to (0, \infty)$ such that $\lim_{t \to T_+} (\nu(t) + \xi(t)) = 0$ and 
\EQ{
\lim_{ t \to T_+}\Big( \frac{\xi(t)}{\sqrt{T_+-t}} + \frac{\sqrt{T_+- t}}{\nu(t)} \Big) = 0,
}
and so that 
\EQ{
\lim_{t \to T_+} E( u(t); D( x_0, \nu(t)) \setminus D( x_0, \xi(t))) = 0.
}
\end{lem} 

\begin{lem}[No self-similar concentration in the global case]  \label{lem:ss-global} 
Let $u(t)$ be the solution to~\eqref{eq:hmhf} with initial data $u_0 \in \E$. Suppose that $T_+ = \infty$. Let $y \in \R^2$.  Then, 
\EQ{
\lim_{t \to \infty} E( u(t); \R^2 \setminus D( y, \alpha \sqrt{ t})) = 0
}
for any $\al >0$. In particular, there exist $T_0 < \infty$ and a function $ \xi: [T_0, \infty) \to (0, \infty)$ such  that  
\EQ{
\lim_{ t \to T_+} \frac{\xi(t)}{\sqrt{t}}  = 0,
}
and so that 
\EQ{
\lim_{t \to T_+} E( u(t); \R^2 \setminus D( y, \xi(t))) = 0.
}
% In particular, $\lim_{t \to T_+} E( u(t); D( x_0, 2\sqrt{T_+ - t}) \setminus D( x_0, \frac{1}{2} \sqrt{T_+ - t})) = 0$. 
\end{lem} 

\begin{proof} %\label{lem:ss-global} 
Fix $y \in \R^2$ and $\al>0$. Let $\eps>0$ and, using~\eqref{eq:tension-L2} choose $T_0>0$ so that, 
\EQ{
\frac{4 \sqrt{E( u(0))}}{ \al}  \Big( \int_{T_0}^\infty  \| \p_t u(t) \|_{L^2}^2 \, \ud t  \Big)^{\frac{1}{2}} \le \frac{\eps}{2}.   
}
Next, let $T_1 \ge T_0$ be sufficiently large so that, 
\EQ{
E\big( u(T_0);  \R^2 \setminus D\big(y, \frac{\al\sqrt{T}}{4}\big)\big)  \le \frac{\eps}{2} 
}
for all $T \ge T_1$. Fixing any such $T$, we set 
\EQ{
 \phi_T(\abs{x})  = 1 - \chi( 4\abs{x}/ \al \sqrt{T})
 }
%With $T_0$ as above, find $R_0 = R_0( T_0, \eps)>0$ such that 
%\EQ{
%E( u(T_0); R_0, \infty)  \le \eps. 
%}
%Next let $T_1> \max\{T_0, 4R_0^2/ \al^2\}$ and let $R(t)$ be an increasing, differentiable curve such that $R(T_0) = R_0$ and $R(t) = \frac{\al}{2} \sqrt{t}$ for all $t \ge T_1$. Define, 
%\EQ{
%\phi(t, r) = 1 - \chi_{R(t)}(r) 
%}
where $\chi(r)$ is a smooth function on $(0, \infty)$ such that $\chi(r) = 1$ for $r \le 1$, $\chi(r) = 0$ if $r \ge 4$, and  $\abs{\chi'(r)} \le 1$ for all $r \in (0, \infty)$. 
We now use the identity~\eqref{eq:loc-en-id} on the time interval  $[T_0, T$] and with the function $\phi = \phi_T$ to obtain the inequality 
\EQ{
\frac{1}{2}  \| \na u(T) \phi_T \|_{L^2}^2 & \le   \frac{1}{2}  \| \na u(T_0) \phi_T\|_{L^2}^2 +  \int_{T_0}^T   \abs{\na u(t) } \abs{ \p_t u(t) } \abs{ \na \phi_T} \phi_T\, \ud t  \\
&\le E( u(T_0);  \R^2 \setminus D(y, \frac{\al\sqrt{T}}{4})) +  \frac{4 \sqrt{E( u(0))}}{ \al}  \Big( \int_{T_0}^\infty  \| \p_t u(t) \|_{L^2}^2 \, \ud t  \Big)^{\frac{1}{2}} \le \eps
}
which holds for all $T \ge T_1$, completing the proof.  %\EQ{
%\int_0^\infty \bfe(u(T, r)) \, \phi_T(r)^2 \, r \, \ud r  \le \int_0^\infty \bfe(u(T_0, r)) \, \phi_T(r)^2 \, r \, \ud r + \frac{4\sqrt{E(u(0))}}{\al}\Big(\int_{T_0}^T \int_0^\infty (\p_t u (t, r))^2 \, r \, \ud r \, \ud t\Big)^{\frac{1}{2}}
%} 
%Using the above together with~\eqref{eq:T_0-choice} and~\eqref{eq:T_1-choice}  we find that
%\EQ{
%E( u(T); \al \sqrt{T}, \infty) \le \eps.  
%}
%for all $T \ge T_1$, completing the proof of~\eqref{eq:ext-energy-global}.
\end{proof}

\subsection{Localized sequential bubbling} 
%We define a localized distance function 
%\EQ{ \label{eq:delta-def} 
%\bs\delta_{R}(u) := \inf_{m, M, \vec \iota, \vec \lam} \Big( \| u - \calQ(m, \vec \iota, \vec \lam) \|_{\E(r \le R)}^2 + \sum_{j =1}^M \Big( \frac{\lam_j}{\lam_{j+1}} \Big)^k \Big)^{\frac{1}{2}}
%}
%where the infimum is taken over all $m \in \Z$, $M \in \{0, 1, 2, \dots\}$, all vectors $\iota \in\{-1, 1\}^M$, $\vec \lam \in (0, \infty)^M$, and  we use the convention that the last scale $\lam_{M+1} = R$. 
%\begin{lem} Let $u_n(t) \in \E_{\ell, m}$ be solutions to~\eqref{eq:hf} on time intervals $I_n = [0, \rho_n^2]$ where $ \rho_n >0$ for all $n$. Suppose that $R_n \to \infty$ and 
%\EQ{
%\lim_{n \to \infty}  \int_0^{\rho_n^2} \int_0^{\rho_n R_n} t  | \p_t u_n(t, r)|^2 \, r \,\ud r  \,  \frac{\ud t}{t}  = 0
%}
%Then, there exists a sequence of times $t_n \in [0, \rho_n^2]$ and a sequence $r_n  \le R_n$ with $r_n \to \infty$  so that 
%\EQ{
%\lim_{n \to \infty} \bs \delta_{\rho_n r_n}( u(t_n)) = 0. 
%}
%
%\end{lem} 

The following localized sequential bubbling lemma proved in a series of works by Struwe~\cite{Struwe85}, Qing~\cite{Qing}, Ding-Tian~\cite{DT}, Wang~\cite{Wang}, Qing-Tian~\cite{QT}, and Lin-Wang~\cite{Lin-Wang98}. We state as a lemma below a summary of these works, which can be found, for example, in Topping's paper~\cite[Theorem 1.1]{Topping04}. 
% (building off the classic paper~\cite{Struwe85} by Struwe) plays a key role in our analysis; see also the version of this lemma in~\cite[Theorem 1.1]{Topping04}, which summarizes the subsequent generalizations (and alternate proofs) by~\cite{DT, Wang, QT, Lin-Wang98}. 

\begin{thm}[Compactness Lemma] \emph{\cite[Theorem 1.2]{Qing}, \cite[Theorem 1.1]{Topping04}} \label{lem:compact} 
Let $u_n: \R^2 \to \Sp^2 \subset \R^3$ be a sequence of $C^2$  maps 
%such that there exists a constants $u_{\infty, n} \in \Sp^2$  so that  $ \lim_{n \to \infty}| u_n(x)- u_{\infty, n}| =0$,  and with
such that  $\limsup_{n \to \infty} E( u_n) < \infty$. Let $\rho_n \in (0, \infty)$ be a sequence and suppose that 
\EQ{ \label{eq:no-tension} 
\lim_{n \to \infty} \rho_n \| \calT(u_n) \|_{L^2} = 0.
}
Then, for every sequence $y_n \in \R^2$,  there exists a sequence $R_n \to \infty$ 
%such that, up to passing to a subsequence of the $u_n$, we have, 
%\EQ{
%\lim_{n \to \infty} \bs \de_{\gamma_0}( u_n; D(y_n, R_n \rho_n))  = 0, 
%}
%for every $\gamma_0 \in (0, 2\pi)$. 
%In particular, the subsequence of the $u_n$ can be chosen so that there is 
a fixed integer $M \ge 0$,  a constant $C>0$,  a harmonic map $\om_0$ (possibly constant), non-constant harmonic maps $\om_1, \dots, \om_M$, and sequences of vectors $ b_{ 1, n}, \dots, b_{ M, n} \in D( y_n, C \rho_n)$ and scales $\mu_{ 1, n}, \dots, \mu_{ M, n} \in (0, \infty)$ %such that $\mu_{n, j} \ll \rho_n$ for each $j \in \{1, \dots, M\}$ and 
so that  $\max_j\mu_{j,n}/\rho_n\to0$ as $n\to\infty$ and 
\EQ{ \label{eq:Qing} 
\lim_{n \to \infty} &\Bigg[  E\Big( u_n - \om_0\big( \frac{\cdot - y_n}{ \rho_n}\big)- \sum_{j=1}^M \big( \om_{j} \big( \frac{\cdot - b_{ j, n}}{\mu_{ j, n}} \big) - \om_{j}(\infty) \big); D( y_n, R_n \rho_n)\Big)\\
&\quad + \Big\|u_n - \om_0\big( \frac{\cdot - y_n}{ \rho_n}\big) - \sum_{j=1}^M \big( \om_{j} \big( \frac{\cdot - b_{ j, n}}{\mu_{ j, n}} \big) - \om_{j}(\infty) \big) \Big\|_{L^\infty( D( y_n, R_n \rho_n))}  \\
&\quad + \sum_{j\neq k} \Big( \frac{\mu_{j, n}}{\mu_{k, n}} + \frac{\mu_{ k, n}}{\mu_{ j, n}} + \frac{|b_{ k, n} - b_{ j, n}|^2}{\mu_{ j, n} \mu_{ k, n}} \Big)^{-1} + \sum_{j=1}^M \frac{ \mu_{ j, n}}{\dist( b_{ j, n} ,  \p D( y_n, C\rho_n))}  \Bigg] = 0. 
}
In particular $\lim_{n \to \infty} \bs \de( u_n; D( y_n, \ti R_n \rho_n)) = 0$ for any sequence $1 \ll \ti R_n \ll  R_n$. 
There exist $L\le M$ points $x_{1}, \dots, x_{L} \in D(0, C)$ so that, 
\EQ{ \label{eq:w22-body} 
u_n(y_n + \rho_n \cdot) &\rightharpoonup \om_0 \, \, \textrm{weakly in} \, \, H^{1}(D(0, C); \Sp^2) \\
u_n(y_n + \rho_n \cdot) &\to \om_0 \, \, \textrm{strongly in} \, \, W^{2, 2}_{\loc}(D(0, C)\setminus \{x_{1}, \dots, x_{L}\} ; \Sp^2)
}
For each $j \in \{1, \dots, M\}$ there exist a finite set of points $S_j$, possibly empty and with $\# S_j \le M-1$, such that 
\EQ{\label{eq:w22-bubbles} 
u_n( b_{j, n} + \mu_{j, n} \cdot) \to \om_j \, \, \textrm{strongly in} \, \, W^{2, 2}_{\loc}(\R^2 \setminus S_j ; \Sp^2). 
}
Finally, there exists an integer $K \ge 0$ so that
\EQ{ \label{eq:en-quant} 
\lim_{n \to \infty} E( u_n; D( y_n, R_n\rho_n))  = 4 \pi K.
}
%\Red{[state the convergence properties of $u_n( \mu_{k, n}( \cdot + b_{k, n})$]}
%and so that $\mu_{n, j} \le C \rho_n$. 
\end{thm} 

\begin{rem} Theorem~\ref{lem:compact} can be combined with Lemmas~\ref{lem:ss-bu},~\ref{lem:ss-global}, and the bound~\eqref{eq:tension-L2} to prove a sequential decomposition as in Theorem~\ref{thm:main1} along the~\emph{well-chosen} sequence of times described in Remark~\ref{rem:seq}; see, e.g.,~\cite[Section 2]{Topping-winding}. We note that the second statement~\eqref{eq:Qing} gives $L^\infty$ convergence on the whole disc $D( y_n, R_n \rho_n)$ rather than just at the scales of the bubbles, which is all that is required for $\bs \de( u_n; D(y_n,  \ti R_n \rho_n))$ to tend to zero for $1 \ll \ti R_n \ll R_n$; see Definition~\ref{def:d}. 
\end{rem} 
\begin{rem} 
Parker~\cite{Parker} proved an earlier version of Theorem~\ref{lem:compact} in the case when the sequence $u_n$ consists of harmonic maps, i.e., $\calT(u_n) = 0$ for each $n$. We use this restricted version of Theorem~\ref{lem:compact} (for sequences consisting only of harmonic maps) at several instances in the next section. 
\end{rem} 

%----------------------------------------------------------------------------------------------------------------------------------------------------------------------------------------

\section{Proofs of the main results} \label{sec:collisions} 

\subsection{The minimal collision energy}\label{ssec:collision} 
For the remainder of the paper we fix a solution
$u(t)$ of \eqref{eq:hmhf},
defined on the time interval $I_+=[0, T_+)$
where $T_+<\infty$ in the finite time blow-up case and $T_+= \infty$ in the global case. We fix $\gamma_0>0$ such that $\gamma_0 \le \min\{\frac{1}{100}, \frac{1}{100E( u_0)}\}$ and sufficiently small so that Lemma~\ref{lem:decay} holds.  From now on we omit the subscript $\gamma_0$ from $\bfd_{\gamma_0}$ and $\bs \delta_{\gamma_0}$ and for a harmonic map $\omega$ we write $\lambda(\omega) = \lam(\omega; \gamma_0)$ and $a(\om) = a(\om; \gamma_0)$ for the scale and center. 

Our strategy is to study collisions of bubbles, which we define as follows.

\begin{defn}[The minimal collision energy] \label{def:K} Let $K$ be the smallest natural number with the following properties. There exist sequences $y_n \in \R^2$, $\rho_n, \eps_n \in (0, \infty)$, $\sigma_n, \tau_n \in (0, T_+)$ and $\eta>0$, with %$\xi_{0, n} \in (0, \infty)$ with $\lim_{n \to \infty} \xi_{0, n} \rho_n^{-1} = 0$, 
$\eps_n \to 0$,  $0 < \sigma_n < \tau_n<T_+$,  $\sigma_n, \tau_n \to T_+$, %and $\lim_{n \to \infty}( \frac{ \xi_n}{\rho_n} + \frac{\rho_n}{\nu_n}) = 0$, 
so that
\begin{enumerate} 
\item $\bs \de( u( \sigma_n); D(y_n, \rho_n))  \le \eps_n;  $ 
\item $\bs \de( u( \tau_n); D(y_n, \rho_n))  \ge  \eta $; 
\item the interval $I_n:= [\sigma_n, \tau_n]$ satisfies $| I_n | \le \eps_n \rho_n^2$; 
\item $E( u(\sigma_n); D( y_n, \rho_n)) \to 4 K \pi $ as $ n \to \infty$; 
%\item  $E( u(\sigma_n); D( y_n, \nu_n) \setminus D( y_n, \xi_n))  \le \eps_n$; 
%\item  $\| u(\sigma_n) - \om_{ n} \|_{L^\infty( D( y_n, \nu_n) \setminus D( y_n, \xi_n))} \le \eps_n$. %and $| I_n | \le \eps_n \xi_{0, n}^2$.

\end{enumerate} 
We call $[\sigma_n, \tau_n]$ a sequence of collision intervals associated to $K$ and the parameters $y_n, \rho_n, \eps_n$, and $\eta$,  and we write $[\sigma_n, \tau_n] \in  \calC_K(y_n, \rho_n, \eps_n, \eta)$. 
\end{defn} 

\begin{rem} \label{rem:neck} By Definition~\ref{def:d} and Property (1) in Definition~\ref{def:K}, we can associate to each sequence of collision intervals $[\sigma_n, \tau_n] \in \calC_K(y_n, \rho_n, \eps_n, \eta)$  sequences $\xi_n, \nu_n \in (0, \infty)$ with $\lim_{n \to \infty}( \frac{ \xi_n}{\rho_n} + \frac{\rho_n}{\nu_n}) = 0$, and a sequence of constants $\om_n \in \Sp^2$  so that 
\EQ{
\lim_{n \to \infty}\Big( E( u(\sigma_n); D(y_n, 4\nu_n) \setminus D( y_n, 4^{-1}\xi_n)) + \| u(\sigma_n) - \omega_n \|_{L^\infty( D(y_n, 4\nu_n) \setminus D( y_n, 4^{-1}\xi_n))}\Big) = 0. 
}
Using Property (3) in Definition~\ref{def:K} we can always ensure (by enlarging the excised discs above) that 
\EQ{ \label{eq:ci-xi} 
| I_n | = \tau_n - \sigma_n \ll \xi_n^2. 
}
Then, by Lemma~\ref{lem:Struwe} and Lemma~\ref{lem:sup} the limits above can be propagated throughout the whole collision interval $I_n$ yielding, 
\EQ{\label{eq:ci-neck}
\lim_{n \to \infty} \sup_{t \in [\sigma_n, \tau_n]} E( u(t); D(y_n, \nu_n) \setminus D( y_n,  \xi_n)) +  \| u( t) - \omega_{n} \|_{L^\infty( D( y_n, \nu_n) \setminus D( y_n, \xi_n))} = 0.
} 
Moreover the above holds after enlarging $\xi_n$ or shrinking $\nu_n$, i.e, for any $\ti \xi_n, \ti \nu_n$ with  $\xi_n \ll \ti \xi_n \ll \rho_n \ll \ti \nu_n \ll \nu_n$. 
%Similarly, by Properties (3) and (4) together with Lemma~\ref{lem:Struwe} we see that, 
%\EQ{ \label{eq:ci-neck} 
%\lim_{n \to \infty} \sup_{t \in [\sigma_n, \tau_n]} E( u(t); D(y_n, 2 \rho_n) \setminus D( y_n, 2^{-1} \rho_n)) = 0.
%}
\end{rem} 

\begin{lem}[Existence of $K \ge 1$]
\label{lem:K}
If Theorem~\ref{thm:main} is false, then $K$ is well-defined and $K \ge  1$.
\end{lem}

\begin{proof}[Proof of Lemma~\ref{lem:K}]  
%We do the proof in  the case of finite time blow-up, i.e., $T_+< \infty$ as the case of $T_+ = \infty$ is analogous. 
% \red{[proof needs to be modified to accommodate global case as well, or comments needed on what changes]}.  
Assume Theorem~\ref{thm:main} is false (in either the case $T_+< \infty$ or $T_+ = \infty$). Then we can find $\eta>0$, sequences $\tau_n \to T_+$, $y_n \in \R^2$, $0< \rho_n <\infty$ with $\rho_n  \le \sqrt{T_+ - t_n}$  in the case $T_+< \infty$ and $\rho_n \le \sqrt{t_n}$ in the case $T_+ = \infty$  so that 
\EQ{
\bs \de( u( \tau_n); D( y_n, \rho_n)) \ge \eta, \quad \forall\, n, 
}
and sequences $\al_n \to 0$ and $\be_n \to \infty$ so that 
\EQ{ \label{eq:neck-tau}
\lim_{n \to \infty} E( u( \tau_n); D( y_n, \beta_n \rho_n) \setminus D( y_n, \al_n \rho_n))  = 0, 
}
In case $\rho_n \simeq \sqrt{T_+- \tau_n}$ or $\rho_n \simeq \sqrt{t_n}$, the existence of $\alpha_n,  \beta_n$ as above is guaranteed by Lemma~\ref{lem:ss-bu} or Lemma~\ref{lem:ss-global}. 
%By Lemma~\ref{lem:ext-energy-blow-up} \red{[needs formulation and proof. this should be as in equivariant paper]}, we can assume that $\rho_n \le \sqrt{T_*- \tau_n}$. 

We claim that  there exists a sequence of times $\sigma_n  < \tau_n$, $\sigma_n \to T_+$, such that 
\EQ{\label{eq:sigma_n-tension} 
|[\sigma_n, \tau_n]| \ll \rho_n^2, \mand \lim_{n \to \infty} \rho_n^{2} \| \calT( u(\sigma_n)) \|_{L^2}^2 = 0 .
} 
If not, we could find numbers $c, c_0 >0$ and a subsequence of the $\tau_n$ so that 
\EQ{
\rho_n^{2} \| \calT( u(t)) \|_{L^2}^2 \ge c_0, \quad \forall t \in [\tau_{n}- c \rho_n^2, \tau_n].
}
But then, 
\EQ{
\int_0^{T_+}  \| \calT(u(t)) \|_{L^2}^2 \, \ud t  \ge \sum_n \int_{\tau_{n}- c \rho_n^2}^{\tau_{n}}  \| \calT(u(t)) \|_{L^2}^2 \, \ud t \ge c_0 \sum_n \int_{\tau_{n}- c \rho_n^2}^{\tau_{n}} \rho_n^{-2} \, \ud t  = \infty, 
}
and the above contradicts~\eqref{eq:tension-L2}. 

Using~\eqref{eq:Struwe-back} from Lemma~\ref{lem:Struwe}, and the fact that $|E( u(\sigma_n)) - E( u(\tau_n))| \to 0$ since $\sigma_n, \tau_n \to T_+$ (see Lemma~\ref{lem:lwp}), we see that~\eqref{eq:neck-tau} can be used to ensure that
\EQ{\label{eq:neck-sigma} 
\lim_{n \to \infty}E( u(\sigma_n); D( y_n, 2^{-1} \beta_n \rho_n) \setminus D( y_n, 2 \alpha_n \rho_n))= 0. 
}
Given the sequence $\sigma_n$ in~\eqref{eq:sigma_n-tension} we can apply the Compactness Lemma~\ref{lem:compact} to $u(\sigma_n)$, and conclude that after passing to a subsequence (which we still denote by $\sigma_n$), we see that a bubble decomposition as in~\eqref{eq:Qing} holds
%we have
%\EQ{
%\lim_{n \to \infty} \bs \de( u(\sigma_n) ; D(y_n, R_n\rho_n) ) = 0
%}
for some sequence $R_n \to \infty$. Because of~\eqref{eq:neck-sigma} we see that the harmonic map $\om_0$ in~\eqref{eq:Qing} must be constant, i.e., $\om_0(x) = \om \in \Sp^2$, and we can conclude that 
\EQ{\label{eq:sigma_n}
\lim_{n \to \infty} \bs \de( u(\sigma_n) ; D(y_n, \rho_n) ) = 0.
}
%and that 
%\EQ{
%\lim_{n \to \infty} \| u(\sigma_n) - \om_{\infty} \|_{L^\infty( D( y_n, 8 \rho_n) \setminus D( y_n, 8^{-1} \rho_n))}  = 0. 
%}
%for some sequence $\xi_{0, n }\ll \rho_n$, and we note that since $|I_n| \ll \rho_n$ we can fix this sequence $\xi_{0, n}$ so that $|I_n| \ll \xi_{0, n}^2$. 
By Lemma~\ref{lem:mb-energy} we can find an integer $K \ge 0$ so that 
\EQ{
E( u(\sigma_n); D( y_n, \rho_n))  \to 4 \pi K. 
}
We have shown that Properties (1)--(4) hold for the intervals $[\sigma_n,\tau_n]$. 
This proves that $K$ is well defined and $\ge 0$. 

We claim that $K \ge 1$.  Suppose $K=0$ and $y_n, \rho_n, \eps_n, \sigma_n, \tau_n$ are as in Definition~\eqref{def:K}. But then, we can find $\xi_n, \nu_n$ and $\om \in \Sp^2$ as in~\eqref{eq:ci-neck} in Remark~\ref{rem:neck}. By Lemma~\ref{lem:Struwe} we have  
\EQ{
E( u(\tau_n); D( y_n, \rho_n)) = o_{n}(1), 
} 
and by~\eqref{eq:ci-neck} in Remark~\ref{rem:neck} we have, 
\EQ{
\| u(\tau_n) - \om \|_{L^\infty( D( y_n, \nu_n)\setminus D( y_n,  \xi_n ))} + \frac{\xi_n}{\rho_n} + \frac{\rho_n}{\nu_n}  = o_n(1), 
}
 which make it impossible for $(2)$ in Definition~\ref{def:K} to be satisfied. This proves that $K \ge 1$. 
\end{proof}

%\begin{proof}[Proof of Lemma~\ref{lem:energy-trap}]
%\red{This should be a consequence of the Struwe multiplier identity, see Lemma~\ref{lem:Struwe}.}
%\end{proof} 

%--------------------------------------------CONCLUSION--------------------------------------------%
\subsection{Lengths of collision intervals} \label{sec:conclusion}

We assume that Theorem~\ref{thm:main} is false, let $K\ge1$ be as in Lemma~\ref{lem:K} and let $y_n \in \R^2$, $\rho_n \in (0, \infty)$,  $\eps_n \to 0$, $0 < \sigma_n < \tau_n<T_+$ with $\sigma_n, \tau_n \to T_+$, and $\eta>0$ be a choice of parameters given by Definition~\ref{def:K}, i.e., $[\sigma_n, \tau_n] \in  \calC_K(y_n, \rho_n, \eps_n, \eta)$. 
%The key ingredient in the proof of Theorem~\ref{thm:main} is the following lemma, which shows that the collision intervals $[\sigma_n, \tau_n]$ contain subintervals on which the distance function $\de$ is bounded below 
% and normalized as in~\ref{eq:normal}. 

\begin{lem}[Length of a collision interval]\label{lem:collision-duration} There exists $\eta_0>0$ sufficiently small so that  for each $\eta \in (0, \eta_0]$ there exists $\eps>0$ and $c_0>0$ with the following properties. Let 
$[\sigma, \tau] \subset [\sigma_n, \tau_n]$ be any subinterval such that, 
\EQ{
\bs \de( u( \sigma); D( y_n, \rho_n)) \le \eps, \mand \bs \de( u( \tau); D( y_n, \rho_n)) \ge \eta, 
}
and let $\om \in \Sp^2$ and $ \om_1, \dots, \om_M$ be any collection of non-constant harmonic maps, and $\vec \nu = (\nu, \nu_{1}, \dots, \nu_{M}), \vec \xi = (\xi,  \xi_{1}, \dots, \xi_{M}) \in (0, \infty)^{M+1}$ any admissible vectors in the sense of Definition~\ref{def:d} such that, 
\EQ{
\eps \le \bfd( u( \sigma), \calQ(\bs \om); D( y_n, \rho_n); \vec \nu, \vec \xi) \le 2 \eps.  
}
Then 
\EQ{
\tau - \sigma  \ge c_0 \max_{j \in \{1, \dots, M\}}  \lam(\om_j)^2. 
}

\end{lem} 

\begin{cor}\label{cor:sn} Let $\eta_0>0$ be as in Lemma~\ref{lem:collision-duration},  $\eta \in (0, \eta_0]$, and $[\sigma_n, \tau_n] \in \calC_K(y_n, \rho_n, \eps_n, \eta)$. Then, there exist $\eps \in (0, \eta)$,  $c_0>0$ $n_0 \in \N$,  and $s_n \in (\sigma_n, \tau_n)$ such that for all $n \ge n_0$, the following conclusions hold.  First, 
\EQ{
\bs \de( u(s_n); D(y_n, \rho_n)) = \eps.
}
Moreover, for each $n  \ge n_0$ let $M_n \in \N$,  and $\calQ(\bs \om_n)$ be any sequence of $M_n$-bubble configurations, and let $\vec \nu_n = (\nu_{n}, \nu_{1, n}, \dots, \nu_{M, n}), \vec \xi_n = (\xi_{n}, \xi_{1, n}, \dots, \xi_{M, n}) \in (0, \infty)^{M+1}$ be any admissible sequences in the sense of Definition~\ref{def:d}  such that
\EQ{
\eps \le \bfd( u(s_n), \calQ( \bs \om_{n}); D(y_n, \rho_n), \vec \nu_n, \vec \xi_n) \le 2 \eps 
}
for each $n$, and define 
\EQ{
\lam_{\max, n}=  \lambda_{\max}(s_n):= \max_{j =1, \dots, M_n} \lam( \om_{j, n}). 
}
Then, $s_n +  c_0 \lam_{\max}(s_n)^2 \le  \tau_n$ and, 
\EQ{ \label{eq:d>eps} 
\bs \de( u(t); D(y_n, \rho_n)) \ge \eps, \quad \forall \, \, t \in [s_n, s_n + c_0 \lam_{\max}(s_n)^2].
}
%\EQ{\label{eq:dn-cn} 
%d_n - c_n = \frac{1}{n} \lam_K(c_n)^2, 
%}
%and 
%\EQ{ \label{eq:lamKcn} 
%\frac{1}{2} \lam_K(c_n) \le \lam_K(t) \le 2\lam_K(c_n) \quad \forall\, \, t \in [c_n, d_n]. 
%} 
\end{cor}

We make the following definitions. 
\begin{defn} 
We say that two triples $(\om_j, a_{ j, n}, \lam_{j, n})$ and $(\om_{j'}, a_{ j', n}, \lam_{j', n})$ where $\om_j, \om_{j'}$ are nontrivial harmonic maps, $a_{ j, n}, a_{ j', n} \in \R^2$ are sequences of vectors in $\R^2$, and $\lam_{ j, n}, \lam_{ j', n} \in (0, \infty)$ are sequences of scales, are \emph{asymptotically orthogonal}  if 
\EQ{
\lim_{n \to \infty} \Big( \frac{\lam_{ j, n}}{\lam_{ j', n}} + \frac{\lam_{ j', n}}{\lam_{ j, n}} + \frac{\abs{ a_{ j, n} - a_{ j', n}}^2}{\lam_{ j, n} \lam_{ j', n}} \Big) = \infty. 
}
\end{defn} 

\begin{defn} \label{def:tree} 
We say that a sequence of nontrivial harmonic maps $ \mathfrak{h} = \{\om_{n}\}_{n =1}^\infty$ is a \emph{descendant} sequence of an \emph{ancestor} sequence of harmonic maps $\mathfrak{H} = \{\Om_{n}\}_{n =1}^\infty$ if $\frac{\lam(\Om_{n})}{\lam(\om_{n})} \to \infty$ and there exists a constant $C>0$ so that the discs $D( a(\om_{n}), \lam(\om_{n})) \subset D(a(\Om_{n}), C \lam(\Om_{ n}))$ for all sufficiently large $n$. We denote this relation by $\{\omega_n\}\prec \{\Omega_n\}$, and $\{\omega_n\}\preceq \{\Omega_n\}$ allows for equality. %which is a partial order on sequences of harmonic maps. 
Given a natural number $M$, a collection of sequences of harmonic maps $(\frakh_1, \dots, \frakh_M) = ( \{\om_{1, n}\}_{n =1}^\infty, \dots, \{\om_{1, n}\}_{n =1}^\infty)$ with asymptotically orthogonal centers and scales are partially ordered by $\preceq$. The \emph{roots} are defined to be the maximal elements relative to this partial order. In other words, a sequence of harmonic maps $\frakh_j$ is a root if it is not a descendant sequence of any ancestor sequence $\frakh_{j'}$ for any $j' \in \{1, \dots, M\}$. We denote by 
\EQ{
\calR:= \{ j \in \{1, \dots, M\} \mid \frakh_j \, \, \textrm{is a root}\}
}
Finally, to each root $\frakh_{j}$ we can associate a bubble tree $\calT(j):= \{ \frakh_{j'} \mid  \frakh_{j'} \preceq \frakh_j\}$. 
%We note that a sequence $\om_{n}$  can have more than one ancestor, but all but one of its ancestors is the descendent of another of the ancestors of $\om_{n, j}$. 
%We remark that a failure of asymptotic orthogonality between triples triples $(\theta_{j, k}, b_{n, j, k}, \mu_{n, j, k}$ ) and $( \theta_{j', k'}, b_{n, j', k'}, \mu_{n, j', k'})$  with $j \neq j'$  can only occur if $\om_{ n,j}$  is an ancestor of $\om_{n, j'}$ or vice versa.
\end{defn} 

\begin{proof}[Proof of Lemma~\ref{lem:collision-duration}] 
If the Lemma were false,  we could find %an $\ti \eta>0$, $\ti \eps_n \to 0$, $C_n \to \infty$ and 
intervals $[s_n, t_n] \subset [\sigma_n, \tau_n]$ so that, 
\EQ{ \label{eq:sntn} 
\lim_{n \to \infty} \bs \de( u(s_n); D( y_n, \rho_n)) = 0, \quad \lim_{n \to \infty} \bs \de( u(t_n); D( y_n, \rho_n)) >0, 
}
  integers $M_n \ge0$, sequences of $M_n$-bubble configurations $\calQ(\bs \om_n)$, and sequences of vectors $\vec \nu_n = (\nu_n, \nu_{1, n}, \dots, \nu_{M_n, n}) \in (0, \infty)^{M_n+1},  \vec \xi_n = (\xi_n,  \xi_{1, n}, \dots, \xi_{M_n, n}) \in (0, \infty)^{M_n+1}$ such that 
\EQ{ \label{eq:du(s_n)}
\lim_{n \to \infty}\bfd( u( s_n), \calQ( \bs \om_n); D( y_n, \rho_n); \vec\nu_n, \vec \xi_n) = 0, 
}
and so that for $\lam_{\max, n}:= \max_{j =1, \dots, M_n} \lam( \om_{j, n})$, we have, 
\EQ{ \label{eq:st-lam2}
\lim_{n \to \infty} \frac{(t_n - s_n)^{\frac{1}{2}}}{\lam_{\max, n}} = 0 .% \le C_n^{-1} \lam_{\max, n}. 
}
%By Remark~\ref{rem:neck} we may assume that $\xi_{0, n}$ is the same parameter as in the collision interval $\calC_{K}(y_n, \rho_n, \eps_n, \eta, \om_{\infty, n})$. 
Passing to a subsequence, we may assume that $M_n = M$ is a fixed integer and $\om_{ n} = \om \in \Sp^2$ is a fixed constant.  %\red{We assume without loss of generality that $\lambda(\om_{j,n})$ is non-increasing. }

Consider the sequences of harmonic maps, $\frakh_j = \{\om_{j, n}\}_{n =1}^\infty$, for $j = 1, \dots, M$,  together with sequences of centers $a(\om_{j, n})$ and scales $\lam(\om_{j, n})$, and the partial order $\prec$ on $( \frakh_1, \dots, \frakh_M)$ as in Definition~\ref{def:tree}. Using the language of Definition~\ref{def:tree}, we observe that, after passing to a subsequence in $n$, there exists a sequence $\ti R_n \to \infty$ so that for any root sequences  $\frakh_j = \{\om_{j, n}\}_{n =1}^\infty$, $\frakh_{j'} = \{\om_{j, n}\}_{n =1}^\infty$ with $j, j'  \in \mathcal{R}$, the discs $D( a(\om_{j, n}), 4R_n\lam(\om_{j, n}))$ and $D( a(\om_{j', n}), 4R_n\lam(\om_{j', n}))$ are disjoint for each $n$ for any sequence $R_n \le \ti R_n$. 
By Lemma~\ref{lem:decay}, 
\EQ{
\lim_{n \to \infty} E( \om_{j, n}; \R^2 \setminus D( a(\om_{j, n}); 4^{-1} R_n\lam(\om_{j, n})))  = 0
}
for each $j \in \calR$ and for any sequence $R_n \to \infty$, and hence by ~\eqref{eq:du(s_n)}, 
\EQ{ \label{eq:en-outside} 
\lim_{n \to \infty} E( u(s_n); D( y_n, \rho_n) \setminus \bigcup_{ j \in \calR} D( a( \om_{j, n}), 4^{-1}R_n \lam(\om_{j, n}))) = 0,
}
for any sequence $R_n \to \infty$. 
%and moreover for any sufficiently slowly diverging  sequence $R_n$ the discs $D(a(\om_{j, n}), R_n \lam(\om_{j, n}))$ and $D(a(\om_{j', n}), R_n \lam(\om_{j', n}))$ are disjoint as long as $\om_{j, n}$ is not a descendent of $\om_{j', n}$ or vice versa.  

%and we denote the sequences of rescaled, and re-centered harmonic maps by 
%\EQ{
%\Omega_{ j, n}(x) := \om_{j, n}\big( \lam(\om_{j, n})( x + a( \om_{j, n}))\big), \quad j \in \{1, \dots, M\}
%}
Each of the sequences $\{\om_{j, n}\}_{n =1}^\infty$ for $j \in \{1, \dots, M\}$ satisfies the hypothesis of the Compactness Lemma~\ref{lem:compact} (noting that $\calT(\om_{j, n}) = 0$ since $\om_{j, n}$ is harmonic), and passing to a (joint) subsequence we can find  non-negative integers $M_j$, 
a sequence $\breve R_n \le \ti R_n$ with $1 \ll \breve R_n \ll \xi_n \lam_{\max, n}^{-1}$, 
harmonic maps $\om_{j, 0}$ (possibly constant), non-trivial harmonic maps $\theta_{j, k}$, scales $\mu_{j, k, n} \ll  \lam(\om_{j, n})$ and centers $b_{j, k, n}\in D(a(\om_{ j, n}),C \lam(\om_{ j, n}))$ for each $j$ and where $k \in \{1, \dots, M_j\}$, satisfying~\eqref{eq:w22-body},~\eqref{eq:w22-bubbles}, and so that   
\EQ{ \label{eq:Qing-bub} 
&\lim_{n \to \infty} \Bigg[E\Big(  \omega_{ j, n} - \om_{j, 0}\Big( \frac{ \cdot - a( \om_{j, n})}{\lam(\om_{j, n})}\Big)- \sum_{k =1}^{M_j} \big(\theta_{j, k}\Big( \frac{ \cdot - b_{j, k, n}}{ \mu_{j, k, n}} \Big) - \theta_{j, k}( \infty)\big) ; D_{j,n} \Big)   \\
& + \Big\|  \omega_{ j, n} - \om_{j, 0}\Big( \frac{ \cdot - a( \om_{j, n})}{\lam(\om_{j, n})}\Big)- \sum_{k =1}^{M_j} \big(\theta_{j, k}\Big( \frac{ \cdot - b_{j, k, n}}{ \mu_{j, k, n}} \Big) - \theta_{j, k}( \infty)\big)\Big\|_{L^\infty( D_{j,n})} \\
& + 
\sum_{k \neq k'}  \Big( \frac{ \mu_{ j, k, n}}{\mu_{j, k', n}} + \frac{ \mu_{ j, k', n}}{\mu_{ j, k, n}} + \frac{| b_{ j, k, n} - b_{ j, k', n} |^{2}}{\mu_{ j, k, n} \mu_{j, k', n}} \Big)^{-1} +  \sum_{k =1}^{M_j} \frac{ \mu_{j, k, n}}{\dist( b_{j, k, n}, \partial D(a(\om_{ j, n}),C \lam(\om_{ j, n}))} \Bigg]= 0. 
}
where $D_{j,n}:=D(a(\om_{ j, n}), 4 R_n \lam(\om_{ j, n}))$, $C>0$ is some finite constant, and $R_n$ is a sequence, to be fixed below,  such that $1 \ll R_n \le \breve R_n$ . 
%for any $R>0$. 
In this decomposition, we distinguish the (possibly constant) harmonic maps $\om_{j, 0}$, which arise as the weak limits $\om_{j, n}\big( \lam(\om_{j, n})( \cdot + a( \om_{j, n}))\big) \rightharpoonup\om_{j, 0}$, and we call these the body maps associated to the sequence $\frakh_j = \{\om_{j, n}\}_{n =1}^\infty$. 
%Consider the harmonic maps, $\om_{1, n}, \dots, \om_{M, n}$, together with their centers $a(\om_{j, n})$ and their scales $\lam(\om_{j, n})$, and we denote the sequences of rescaled, and re-centered harmonic maps by 
%\EQ{
%\Omega_{ j, n}(x) := \om_{j, n}\big( \lam(\om_{j, n})( x + a( \om_{j, n}))\big), \quad j \in \{1, \dots, M\}
%}
%Each of these sequences $\Omega_{j, n}$ satisfies the hypothesis of Lemma~\ref{lem:compact}, and passing to a (joint) subsequence we can find integers $M_j$, 
%a sequence $R_n$ with $1 \ll R_n \ll \rho_n \lam_{\max, n}^{-1}$, 
%harmonic maps $\theta_{j, k}$ and $\om_{j, \infty}$, scales $\mu_{j, k, n} \to 0$ and centers $b_{j, k, n}\in D(0, 1)$ for each $j$ and where $k \in \{1, \dots, M_j\}$, so that   
%\EQ{ \label{eq:Qing-bub} 
%\lim_{n \to \infty} E\Big(  \Omega_{ j, n} - \sum_{k =1}^{M_j} (\theta_{j, k}\Big( \frac{ \cdot - b_{j, k, n}}{ \mu_{j, k, n}} \Big) - \theta_{j, k}( \infty)) - \om_{j, \infty}; D(0, R_n) \Big)  = 0 
%} 
%%for any $R>0$. 
%In this decomposition we distinguish the harmonic maps $\om_{j, \infty}$, which arise as the weak limits $\Omega_{ j, n} \rightharpoonup\om_{j, \infty}$, and we call these the body maps associated to the $\Om_{j, n}$. 

Define the set of indices
\EQ{
\calJ_{\max}:= \Big\{ j \in \{1, \dots, M\} \mid C_j^{-1} \le \frac{\lam_{\max, n}}{\lam( \om_{j, n})} \le C_j,  \,\,  \textrm{for each}\,\, n\, \,  \textrm{for some} \, \, C_j >1\Big\}
}
and let
 %\red{This set is well-defined and non-empty by our assumption that $\lam( \om_{j, n})$ is non-increasing. In fact, $\calJ_{\max}=\{1,2,\ldots,m_*\}$ with $m_*\ge1$. } Define the non-negative integer $K_0$ by, 
\EQ{ \label{eq:K0-def} 
4 \pi K_0 := \sum_{j \in \calJ_{\max}} E( \om_{j, 0}). 
}
That is, $4 \pi K_0$ is the sum of the energies of the body maps associated to the $\om_{j, n}$ arising from indices $j \in \calJ_{\max}$.  Note that $\calJ_{\max}$ is a (possibly strict) subset of the set of indices $\calR$ associated to the roots.

\textbf{Case 1}: First suppose that $K_0 = K$, which means that $\calJ_{\max} = \calR =  \{1, \dots, M\}$ and all of the energy in $D( y_n, \rho_n)$ is captured by the body maps. In this case, the sequences $\{\om_{j, n}\}$ have no concentrating bubbles, i.e., $M_j = 0$ for each $j$,  and we have, 
\EQ{
\lim_{n \to \infty} E\Big(  \omega_{ j, n} - \om_{j, 0}\Big( \frac{ \cdot - a( \om_{j, n})}{\lam(\om_{j, n})}\Big); D(a(\om_{ j, n}), 4R_n \lam(\om_{j, n}))\Big)  = 0 
} 
and, 
\EQ{
\lim_{n \to \infty} \Big\|   \omega_{j,  n} - \om_{j,  0}\Big( \frac{ \cdot - a( \om_{j,  n})}{\lam_{\max, n}}\Big)\Big\|_{L^{\infty}( D(a(\om_{j,  n}), 4R_n \lam_{\max, n}))}  = 0,  
} 
for each $j \in \{1, \dots, M\}$. Using~\eqref{eq:du(s_n)}, the fact that $\lam( \om_{j, n}) \simeq \lam_{\max, n}$ for each $j \in \{1, \dots, M\}$,  and the above we can now fix (for Case 1) a sequence $R_n \le \breve R_n$ so that
\EQ{
\lim_{n \to \infty} E\Big( u(s_n) - \om_{j, 0}\big( \frac{ \cdot - a( \om_{j, n})}{\lam(\om_{j, n})}\big) ; D( a(\om_{j, n}), 4 R_n \lam_{\max, n})\Big) = 0,
}
and, 
\EQ{
\lim_{n \to \infty} \Big\|   u(s_n) - \om_{j,  0}\Big( \frac{ \cdot - a( \om_{j,  n})}{\lam_{\max, n}}\Big)\Big\|_{L^{\infty}( D(a(\om_{j,  n}), 4 R_n \lam_{\max, n}))}  = 0
}
for each $j \in \{1, \dots, M\}$, i.e., we need to additionally ensure that $4R_n \lam_{\max, n} \le \min\{ \nu_{j, n}\}_{j =1}^M$. Using Lemma~\ref{lem:Struwe-lwp} and Lemma~\ref{lem:sup} %and~\eqref{eq:Rn-neck} 
along with the fact that $(t_n - s_n)^{\frac{1}{2}} \ll \lam_{\max, n}$, we can propagate these estimates to time $t_n$, i.e., 
\EQ{\label{eq:t_n-bubbles-en} 
\lim_{n \to \infty} E\Big( u(t_n) - \om_{j, 0}\big( \frac{ \cdot - a( \om_{j, n})}{\lam(\om_{j, n})}\big) ; D( a(\om_{j, n}),  R_n \lam_{\max, n})\Big) = 0,
}
and, 
\EQ{\label{eq:t_n-bubbles-sup}
\lim_{n \to \infty} \Big\|   u(t_n) - \om_{j,  0}\Big( \frac{ \cdot - a( \om_{j,  n})}{\lam_{\max, n}}\Big)\Big\|_{L^{\infty}( D(a(\om_{j,  n}),  R_n \lam_{\max, n}))}  = 0
}
for each $j \in \{1, \dots, M\}$.
%\EQ{ \label{eq:t_n-bubbles} 
%\lim_{n \to \infty} E\Big( u(t_n) - \om_{j, 0}\big( \frac{ \cdot - a( \om_{j, n})}{\lam(\om_{j, n})}\big) ; D( a(\om_{j, n}), R_n \lam_{\max, n})\Big) = 0. 
%}
%Additionally, by~\eqref{eq:en-outside} we note that at time $s_n$, 
%\EQ{
%\lim_{n \to \infty} E( u(s_n); D(y_n, \rho_n) \setminus \bigcup_{j =1}^M D( a(\om_{j, n});  R_n \lambda_{\max, n})) = 0. 
%}
Using Lemma~\ref{lem:Struwe},  and that $(t_n - s_n)^{\frac{1}{2}} \ll \lam_{\max, n}$ we can also propagate~\eqref{eq:en-outside} to time $t_n$, deducing 
\EQ{\label{eq:t_n-exterior} 
\lim_{n \to \infty} E\Big( u(t_n); D(y_n, \rho_n) \setminus \bigcup_{j =1}^M D( a(\om_{j, n});  R_n \lambda_{\max, n})\Big) = 0.  
}
Combining~\eqref{eq:t_n-bubbles-en}~\eqref{eq:t_n-bubbles-sup},~\eqref{eq:t_n-exterior}, the disjointness of the discs $D(a(\om_{j, n}), R_n \lam(\om_{j, n}))$, the asymptotic orthogonality of the triples $(\om_{j, 0}, a(\om_{j, n}), \lambda(\om_{j, n}))$,  and Remark~\ref{rem:neck}, we find that 
\EQ{
\lim_{n \to \infty} \bs \de( u(t_n) ; D( y_n, \rho_n)) = 0,
}
which contradicts~\eqref{eq:sntn}.

\textbf{Case 2}: Next, consider the case  $K_0 < K$. We show this case leads to a contradiction with the minimality of $K$. Again we will need $R_n \to \infty$ such that $4R_n \lam_{\max, n} \le  \min \{\nu_{j, n}\}_{j \in \calJ_{\max}} $ and $R_n \le \breve R_n$. 

We claim there exists an integer $L \ge 1$, sequences $\{ x_{\ell, n} \}_{\ell =1}^L$ with $x_{\ell, n} \in D(y_n, \xi_n)$ for each $n$ and each $\ell \in \{1, \dots, L\}$,  and a sequence $r_n$ such that 
\EQ{
(t_n - s_n)^{\frac{1}{2}}  \ll r_n \ll  \lam_{\max, n},
}
such that the discs $D(x_{\ell, n}, r_n)$ are disjoint for $\ell \in \{1, \dots, L\}$ and satisfy, 
\EQ{ \label{eq:en-inside} 
\lim_{n \to \infty} E\Big( u(s_n);  \bigcup_{\ell =1}^L D( x_{ \ell, n}, r_n)\Big) = 4\pi K - 4 \pi K_0, 
}
as well as 
\EQ{ \label{eq:rn-sep} 
\lim_{n \to \infty} \frac{| x_{\ell, n} - x_{\ell', n}|}{r_n} = \infty 
}
for $\ell \neq \ell'$, 
and finally such that there exists sequences $\al_n \to 0, \beta_n \to \infty$ so that 
\EQ{ \label{eq:rn-neck} 
\lim_{n \to \infty} \sum_{ \ell = 1}^L E( u(s_n);  D( x_{ \ell, n}, \beta_n r_n) \setminus D( x_{ \ell, n}, \alpha_n r_n)) = 0, 
}
and a sequence $\breve \xi_n$ so that 
\EQ{ \label{eq:inside-xi} 
\xi_n \ll \breve \xi_n \ll \rho_n \mand D(x_{\ell, n}, \be_n r_n) \subset D( y_n, \breve \xi_n). 
}
%for some sequences $\al_n \to 0$ and $\be_n \to \infty$. $D(x_{\ell, n}, r_n) \subset D( y_n, \breve \xi_n)$ (for some sequence $\xi_n \ll \breve \xi_n \ll \rho_n$)
 
 We construct a set of  sequences $\calP:=\{ \{x_{\ell,n}\}\: :\:1\le\ell\le L \}$ and the radii $\{r_n\}$ as follows.  
 %The sequences of harmonic maps $\frakh_j:=\{\omega_{j,n}\}$ with $1\le j\le M$, are partially ordered by~$\prec$.   The roots $\calR$ are defined to be the maximal elements relative to this partial order.  
 Any root $\frakh_j$ with $j\in\calJ_{\max}$ we call a dominant root.
%For any   root $\frakh_j$ which is not dominant (i.e., $j \in \calR\setminus \calJ_{\max}$) we include the sequences $a(\om_{ j, n})$  in the set $\calP$. 
For any dominant root~$\frakh_{j_0}$ we define 
$\calT(j_0)=\{ \frakh_j\preceq \frakh_{j_0}  \}$ as the bubble tree with root~$\frakh_{j_0}$, and 
$\calD(j_0)$ as the maximal elements of the pruned tree $\calT(j_0)\setminus \{   \frakh_{j_0}  \}$. 

We define ~$\calP_0$ as  the points $y_{ \ell, n}$ for $\ell \in \{1, \dots, L'\}$ as an enumeration of all (i) $a(\om_{ j, n})$ with $\frakh_j \in \calR \setminus \calJ_{\max}$ (i.e., the centers of the roots that are not dominant), (ii) 
$a(\om_{ j, n})$ with $\frakh_j \in \calD(j_0)$ for some $j_0 \in \calJ_{\max, n}$, and (iii) sequences $b_{ j_0, k, n}$ associated to
harmonic maps $\theta_{j_0, k} \big( \frac{ \cdot - b_{j_0, k, n}}{\mu_{ j_0, k, n}}\big)$ for some $j_0 \in \calJ_{\max, n}$ that are 
\begin{itemize}
\item  asymptotically orthogonal to every  $\frakh_j \in \calD(j_0)$
\item   not descendants of any $\frakh_j \in \calD(j_0)$.
\end{itemize}
Passing to a joint subsequence we can assume that the limits 
\EQ{
\lim_{n \to \infty} \frac{(t_n - s_n)^{\frac{1}{2}}}{ \dist( y_{\ell', n}, y_{ \ell, n})}
}
exist in $[0, \infty]$ for all $\ell \neq \ell' \in \{1, \dots, L'\}$. 
We define $\calP$ by means of $\calP_0$ by the following algorithm: we include the sequence $y_{ \ell_0, n}\in\calP_0$ in the set $\calP$ if 
\EQ{
\lim_{n \to \infty} \frac{(t_n - s_n)^{\frac{1}{2}}}{ \dist( y_{ \ell_0, n}, y_{ \ell, n})}  = 0, \quad \forall \ell \in \{1, \dots, L')\}\setminus \ell_0.
}
For those $y_{\ell_0, n}\in\calP_0$ for which the above does not hold we define the sets, 
\EQ{
\calB(\ell_0)  := \Big\{ \ell_0 \, \,  \textrm{and any}\, \,   \ell \, \, \textrm{for which } \, \,  \lim_{n \to \infty} \frac{(t_n - s_n)^{\frac{1}{2}}}{ \dist( y_{ \ell_0, n}, y_{ \ell, n})}  \neq0\Big\}.
}
An index $\ell$ can be in at most one set $\calB(\ell_0)$, i.e., the sets $\calB(\ell) = \calB(\ell')$ if $\ell' \in \calB(\ell)$. For each of the sets $\calB(\ell_0)$ we let, for each $n$,  $x_{ \ell_0, n}$ denote the barycenter of the points $y_{ \ell, n}$ associated to indices $\ell \in \calB(\ell_0)$. We include the points $x_{ \ell_0, n}$ in the set $\calP$. 
This completes the construction of the set $\calP$, which consists of finitely many (say $L \in \N$)  sequences $\{x_{ \ell, n}\}  \subset D(y_n, \xi_n)$ for $\ell \in \{1, \dots, L\}$. 

We choose $r_n$ to be any sequence such that, 
\EQ{  \label{eq:rn} 
(t_n - s_n)^{\frac{1}{2}} &\ll r_n \ll \lam_{\max, n}, \\
 R_n\lam( \om_{j, n}) &\ll r_n \quad \forall\,  j \not \in \calJ_{\max}, \\
 \nu_{j, n} &\ll r_n  \quad \forall j \not \in \calJ_{\max}  \\  
 \max(\mu_{j, k, n},\xi_{j,n})  &\ll r_n \quad \forall  \, j \in \calJ_{\max}, \forall  \, k \in \{1, \dots, M_j\},
}
and such that the discs $D(x_{ \ell, n}, r_n)$ satisfy~\eqref{eq:rn-sep}. In view of the definition of $j_0\in \calJ_{\max}$,
\[
\lam_{\max, n}^{-1} |a(\om_{ j_0, n}) - a(\om_{ j, n})| \to\infty
\]
for all $j\in\calR \setminus \calJ_{\max}$. This ensures that for any $x_{ \ell, n}$, which is one of the sequences $a(\om_{ j, n})$ for $j \in\calR \setminus \calJ_{\max}$, the disc $D(x_{\ell, n}, r_n)$ is separated from any of the discs $D(a(\om_{ j_0, n}),R_n \lam_{\max, n})$ for $j_0 \in \calJ_{\max}$ by an amount $\gg r_n$ (we are free to  take $R_n \to \infty$ to be diverging as slowly as needed).  %Such a sequence $r_n$ exists by the construction of the set $\calP$. %Note that $R_n$ above is as in~\eqref{eq:Qing-bub}. 

We claim that the sequences of discs $D(x_{\ell, n}, r_n)$ with $x_{ \ell, n} \in \calP$ satisfy~\eqref{eq:en-inside}. To see this, first note that for any $j_0 \in \calJ_{\max}$ 
\EQ{ \label{eq:body-max} 
\lim_{n \to \infty} E\Big( u(s_n) - \om_{j_0, 0}\big( \frac{ \cdot - a( \om_{j_0, n})}{\lam(\om_{j_0, n})}\big) ; D( a(\om_{j_0, n}), 4R_n \lam_{\max, n}) \setminus \bigcup_{\ell=1}^{L} D( x_{ \ell,n}, r_n)\Big) = 0
}
which follows from the construction of the set $\{x_{\ell, n}\}_{\ell = 1}^L$, the limit in~\eqref{eq:Qing-bub}, and the choice of $r_n$. Note also that $r_n \ll \lam_{\max, n}$ means that, 
\EQ{ \label{eq:body-in-rn} 
\lim_{n \to \infty} E\Big( \om_{j_0, 0}\big( \frac{ \cdot - a( \om_{j_0, n})}{\lam(\om_{j_0, n})}\big) ; \bigcup_{\ell=1}^{L} D( x_{ \ell, n}, r_n)\Big) = 0. 
}
We can conclude from the above,~\eqref{eq:body-max},~\eqref{eq:K0-def}, and~\eqref{eq:en-outside} that, 
\EQ{ \label{eq:en-outside-rn} 
\lim_{n \to \infty} E\Big( u(s_n); D(y_n, \rho_n) \setminus \bigcup_{\ell=1}^{L} D( x_{\ell,n}, r_n)\Big) = 4 \pi K_0
}
The condition~\eqref{eq:en-inside} follows then from above and the disjointness of the discs $D(x_{\ell,n}, r_n)$. The condition~\eqref{eq:rn-neck} and the existence of the sequence $\breve \xi_n$ as in~\eqref{eq:inside-xi} follows from the construction of the set $\calP$ and the choice of $r_n$.  

We claim that there must exist $\ell_1 \in \{1, \dots, L\}$,  $ \eta_1>0$ so that, up to passing to a subsequence in $n$, we have, 
\EQ{ \label{eq:bad-disc} 
\bs \de( u(t_n); D(x_{\ell_1, n}, r_n))  \ge  \eta_1. 
}
To see this, we argue by contradiction. If~\eqref{eq:bad-disc} fails, then we would have, 
\EQ{ \label{eq:all-good} 
\lim_{n \to \infty} \bs \de( u(t_n); D(x_{\ell, n}, r_n))  = 0, \quad \forall \, \ell  \in \{1, \dots, L\}. 
}
We will use the above to show that 
\EQ{ \label{eq:del-all-good} 
\lim_{n \to \infty} \bs \de( u(t_n); D( y_n, \rho_n)) = 0, 
}
which contradicts~\eqref{eq:sntn}. To start,  $(t_n - s_n)^{\frac{1}{2}} \ll r_n$ means we can use Lemma~\ref{lem:Struwe} and~\eqref{eq:rn-neck} to propagate~\eqref{eq:en-inside},~\eqref{eq:en-outside-rn}, and~\eqref{eq:body-max} to time $t_n$, giving, 
\begin{align}  
\lim_{n \to \infty} E\Big( u(t_n);  \bigcup_{\ell =1}^L D( x_{ \ell, n}, r_n)\Big) &= 4\pi K - 4 \pi K_0, \\
\lim_{n \to \infty} E\Big( u(t_n); D(y_n, \rho_n) \setminus \bigcup_{\ell=1}^{L} D( x_{\ell,n}, r_n)\Big) &= 4 \pi K_0 \\
\lim_{n \to \infty} E\Big( u(t_n) - \om_{j_0, 0}\big( \frac{ \cdot - a( \om_{j_0, n})}{\lam(\om_{j_0, n})}\big) ; D( a(\om_{j_0, n}), R_n \lam_{\max, n}) \setminus \bigcup_{\ell=1}^{L} D( x_{ \ell,n}, r_n)\Big) &= 0,  \label{eq:en-tn} 
\end{align} 
for all $ j_0 \in \calJ_{\max}$
where in the last line we remark that for each $\ell \in \{1, \dots, L\}$ either the disc $D(x_{\ell, n}; r_n)$ is completely contained in $D( a(\om_{j_0, n}), R_n \lam_{\max, n})$ or disjoint from it. 

Next, using $\lam_{\max, n} R_n \le \min\{\nu_{j, n}\}_{j \in \calJ_{\max}}$ and $ \max\{ \nu_{j, n}\}_{j \not \in \calJ_{\max}} \ll r_n$, we see that~\eqref{eq:du(s_n)} can be combined  with the middle line of~\eqref{eq:Qing-bub} to yield,  
\EQ{
\lim_{n \to \infty} \Big\| u(s_n) -  \om_{j_0, 0}\big( \frac{ \cdot - a( \om_{j_0, n})}{\lam(\om_{j_0, n})}\big) \Big\|_{L^\infty( D( a(\om_{j_0, n}), 4R_n \lam_{\max, n}) \setminus \bigcup_{\ell=1}^{L} D( x_{ \ell,n}, 4^{-1}r_n))} = 0
}
for all $j_0 \in \calJ_{\max}$. Since $(t_n - s_n)^{\frac{1}{2}} \ll r_n$, Lemma~\ref{lem:sup},~\eqref{eq:body-max}, and~\eqref{eq:rn-neck} allows us to propagate the above to time $t_n$, yielding, 
\EQ{ \label{eq:sup-body} 
\lim_{n \to \infty} \Big\| u(t_n) -  \om_{j_0, 0}\big( \frac{ \cdot - a( \om_{j_0, n})}{\lam(\om_{j_0, n})}\big) \Big\|_{L^\infty( D( a(\om_{j_0, n}), R_n \lam_{\max, n}) \setminus \bigcup_{\ell=1}^{L} D( x_{ \ell,n}, r_n))} = 0
}

Using again Lemma~\ref{lem:Struwe-lwp}  and~\eqref{eq:en-outside}, the construction of the sequences $\{ x_{\ell, n}\}$ and the choice of $\lam_{\max, n} \gg r_n \gg (t_n - s_n)^{\frac{1}{2}}$ as well as $r_n\gg R_n\lambda(\omega_{j,n})$ for all $j\not\in\calJ_{\max}$ we have, 
\EQ{ \label{eq:en-outside-all} 
\lim_{n \to \infty} E\Big( u(t_n); D( y_n, \rho_n) \setminus \Big[ \bigcup_{j \in \calJ_{\max}} D( a(\om_{j, n}), R_n \lam_{\max, n}) \cup \bigcup_{ \ell  =1}^L D( x_{\ell, n}; r_n) \Big] \Big)  = 0.
}
Now, by~\eqref{eq:all-good}, after passing to a joint subsequence in $n$, for each $\ell \in \{1, \dots, L\}$ we can find an integer $\ti M_\ell \ge 0$, a sequence of $\ti M_{\ell}$-bubble configurations $\calQ(\bs \Om_{\ell, n})$, % constants $\Omega_{ \ell, \infty}$, sequences of harmonic maps $\bs \Om_{\ell, n} = ( \Om_{\ell, 1, n}, \dots, \Om_{\ell, \ti M_{\ell}, n})$, 
 and sequences of vectors $\vec \nu_{\ell, n} = (\nu_{\ell, n}, \nu_{ \ell,1,  n}, \dots, \nu_{\ell, \ti M_{\ell}, n})$ and $\vec \xi_{\ell, n} = ( \xi_{\ell, n}, \xi_{ \ell,1,  n}, \dots, \xi_{\ell, \ti M_{\ell}, n})$,  so that 
\EQ{  \label{eq:d-inside-rn} 
\lim_{n \to \infty} \bfd( u(t_n), \calQ( \bs \Om_{\ell, n}); D( x_{\ell, n}, r_n); \vec \nu_{\ell, n}, \vec \xi_{\ell, n}) = 0. 
}
Here $\bs\Om_{\ell, n}=(\Om_{\ell, n},\Om_{\ell,1, n},\ldots,\Om_{\ell, \ti M_{\ell}, n})$. 
Dropping the constants $\Omega_{ \ell,n} \in \Sp^2$ in the $\ti M_{\ell}$-bubble configurations, consider finally the sequence (in $n$) of multi-bubbles formed by the constant $\omega \in \Sp^2$, 
% $\ti \om_{\infty, n} :=u_{\textrm{avg}, n} |u_{\textrm{avg}, n}|^{-1}$ where $u_{\textrm{avg},n} := | D( y_n, \rho_n)|^{-1} \int_{D( y_n, \rho_n)} u(t_n, x) \, \ud x$, \Red{[couldn't think of what else to do here -- this highlights the fact that the constant in the definition of a multi-bubble is still not seen by $\bs\delta$. it is at this stage of the argument that I don't see how to come up with a definition of $\bs \delta$ that allows us match up the different constants here, but still lets us prove Theorem 2]} 
 and the harmonic maps
\EQ{
\{ \Om_{\ell, k, n}\}_{\ell =1, k =1}^{\ell = L, k = \ti M_{\ell}}, \, \{\om_{j, 0, n}\}_{j \in \calJ_{\max}}:=  \Big\{ \om_{j, 0}( \frac{\cdot - a(\om_{j, n})}{ \lam( \om_{j, n})}) \Big\}_{j \in \calJ_{\max}}.
}
For each $j \in \calJ_{\max}$ we define $\nu_{j, n} := R_n$ and $\xi_{j, n} = r_n$ and then defining $$\vec{\ti\nu}_n:= (\nu_n,  (\nu_{\ell, n})_{\ell =1}^L, ( \nu_{j, n} )_{j \in \calJ_{\max}}),\qquad \vec{\ti \xi}_n := ( \breve \xi_n,  (\xi_{\ell, n})_{\ell =1}^L, ( \xi_{j, n} )_{j \in \calJ_{\max}})$$ we claim that 
\EQ{ \label{eq:d-to-zero-all} 
\lim_{n \to \infty} \bfd\big( u(t_n), \calQ(  \om,( \Om_{\ell, k, n})_{\ell =1, k =1}^{\ell = L, k = \ti M_{\ell}},  ( \om_{j, 0, n})_{j \in \calJ_{\max}}
); D( y_n, \rho_n); \vec{\ti \nu}_n, \vec{\ti \xi}_n\big)  = 0,
}
which would yield~\eqref{eq:del-all-good}. Indeed, by~\eqref{eq:d-inside-rn} and since all of the $D( x_{\ell, n}, r_n)$ are disjoint and satisfy~\eqref{eq:rn-sep}, any distinct triples $( \Om_{\ell, k, n}, a( \Om_{\ell, k, n}), \lam( \Om_{\ell, k, n}))$ and $( \Om_{\ell', k', n}, a( \Om_{\ell', k', n}), \lam( \Om_{\ell', k', n}))$ are asymptotically orthogonal for $(\ell, k) \neq ( \ell', k')$. Moreover, for any $\ell$ and $j_0 \in \calJ_{\max}$ for which $D( x_{\ell, n}, r_n) \subset D( a(\om_{j_0, n}), R_n \lam(\om_{j_0, n}))$, the triples $$( \Om_{\ell, k, n}, a( \Om_{\ell, k, n}), \lam( \Om_{\ell, k, n}))\text{\ \  \ and\ \ \ } \Big(\om_{j_0, 0}( \frac{\cdot - a(\om_{j_0, n})}{ \lam( \om_{j_0, n})}),a(\om_{j_0, n}),  \lam( \om_{j_0, n})\Big)$$ are asymptotically orthogonal since $r_n \ll \lam_{\max, n}$.  Indeed, 
  %for any associated harmonic map $\Om_{\ell, k, n}$ we have, 
\EQ{
\lim_{n \to \infty} E\Big( \om_{j_0, 0}\big( \frac{ \cdot - a( \om_{j_0, n})}{\lam(\om_{j_0, n})}\big) ; D( x_{\ell, n}, r_n)\Big) = 0 , \quad \forall j_0 \in \calJ_{\max}, \quad \forall  \ell\in  \{1, \dots, L\}
}
and 
\EQ{
\lim_{n \to \infty} E( \Om_{\ell, k, n}; D(y_n, \rho_n) \setminus D( x_{\ell, n}, r_n))  = 0, \quad \forall \, \ell \in \{1, \dots, L\}, \, k \in \{1, \dots, \ti M_{\ell}\}.
}
These observations, together with~\eqref{eq:en-tn},~\eqref{eq:sup-body}, the estimate~\eqref{eq:en-outside-all}, and Remark~\ref{rem:neck} (using now $\breve \xi_n$ instead of $\xi_n$),  yield~\eqref{eq:d-to-zero-all}. This completes the proof of~\eqref{eq:bad-disc}.

Having established~\eqref{eq:bad-disc} we claim that there exist times $\ti \sigma_n < t_n$ so that 
\EQ{ \label{eq:tisigma} 
t_n - \ti \sigma_n \ll r_n^2  \mand \lim_{n \to \infty} r_n \| \calT( u(\ti \sigma_n)) \|_{L^2}  = 0 .
} 
If not, we could find $c, c_1>0$ and a subsequence of the $t_n$ for which, 
\EQ{
r_n^2 \| \calT( u(t)) \|_{L^2}^2 \ge c_1 \quad \forall t \in [t_n - c r_n^2, t_n]. 
}
But then we would have, 
\EQ{
\sum_n \int_{t_n - c r_n^2}^{t_n} \| \calT( u(t)) \|_{L^2}^2 \, \ud t \ge c_1 \sum_n \int_{t_n - c r_n^2}^{t_n} r_n^{-2}  \, \ud t \ge c c_1 \sum_n 1 = \infty,
}
which contradicts~\eqref{eq:tension-L2}. Given the sequence $\ti \sigma_n$ as in~\eqref{eq:tisigma} we can apply the Compactness Lemma~\ref{lem:compact}, so that after passing to a subsequence in $n$ (still denoted by $\ti \sigma_n, t_n$), we have a bubble decomposition as in~\eqref{eq:Qing} 
%\EQ{ 
%\lim_{n \to \infty} \bs \de ( u( \ti \sigma_n); D( x_{\ell_1, n}, \hat{R}_n r_n) )  = 0,
%}
for some sequence $\hat R_n \to \infty$. The estimate~\eqref{eq:rn-neck} can be propagated to time $\ti \sigma_n$ using Lemma~\ref{lem:Struwe}, which gives, 
\EQ{ \label{eq:no-neck-tis} 
\lim_{n \to \infty} E\big( u(\ti \sigma_n); D( x_{\ell, n}; 2^{-1} \beta_n r_n) \setminus D(x_{\ell, n}; 2 \al_n r_n)\big) = 0.
}
The above ensures that the harmonic map in~\eqref{eq:Qing} at scale $r_n$ must be constant which we denote by $\ti \om \in \Sp^2$, and so we can conclude that in fact, 
\EQ{ \label{eq:prop1} 
\lim_{n \to \infty} \bs \de ( u( \ti \sigma_n); D( x_{\ell_1, n},  r_n) )  = 0,
}
%and using also~\eqref{eq:Qing} we have %we can find a scale $\ti r_n$ such that $t_n - \ti \sigma_n \ll \ti r_n \ll r_n$ so that, 
%\EQ{ \label{eq:prop6}
%\lim_{n \to \infty} \| u(\ti \sigma_n)- \ti \omega_{\infty} \|_{L^\infty( D(x_{\ell_1, n}, 8 r_n) \setminus D(x_{\ell_1, n}, \frac{1}{8}r_n))} = 0. 
%}
By~\eqref{eq:en-quant} we can find an integer $K_1 \ge 0$ so that, 
\EQ{ \label{eq:prop5} 
E( u(\ti \sigma_n); D( x_{\ell_1, n}, r_n)) \to 4 \pi K_1 \mas n \to \infty. 
}
Because of~\eqref{eq:bad-disc} we must have $K_1 \ge 1$ (since $t_n -\ti \sigma_n \ll r_n^2$). 
 
Consider the time intervals $[\ti \sigma_n, t_n]$ and the discs $D(x_{\ell_1, n}; r_n)$. Property (1) from Definition~\ref{def:K} is given by the first line in~\eqref{eq:prop1}.  Property (2) is given by~\eqref{eq:bad-disc}. Property (3) is satisfied  because of the first estimate in~\eqref{eq:tisigma}, and property (4) because of%because of~\eqref{eq:no-neck-tis},  property (5) by
~\eqref{eq:prop5}.
%and property (6) by ~\eqref{eq:prop6}. 

 Lastly,  we claim that $K_1 < K$. This is clear if $K_0>0$ since in that case some energy lies at the scale $\simeq \lam_{\max, n} \gg r_n$. 
% If $K_0=0$ it is also impossible that $K_1 = K$. 
 If $K_0=0$ and $K_1 = K$ then all  of the energy  in the larger discs $D(y_n, \rho_n)$ would be captured within the sequence of discs $D(x_{\ell_1, n}, r_n)$. On the other hand, recall that there is at least one index $j_0$ such that 
 $\lam(\om_{j_0, n}) = \lam_{\max, n}$ and we have chosen $r_n$ so that $r_n \ll \lam_{\max, n}= \lam(\om_{j_0, n})$, which (by Definition~\ref{def:scale}) implies at least $3\pi$ in energy concentrates outside the discs $D(x_{\ell_1, n}, r_n)$, a contradiction.

We conclude that $K_1< K$ and that $[\ti \sigma_n, t_n] \in \calC_{K_1}( x_{\ell_1, n}, r_n,  \eps_{1, n},  \eta_1)$ for some sequence $\eps_{1, n} \to 0$,  contradicting the minimality of $K$. This completes the proof. 
\end{proof} 

\begin{proof}[Proof of Corollary~\ref{cor:sn}]  Let $\eta_0$ be as in Lemma~\ref{lem:collision-duration} and fix an $\eta \in (0, \eta_0]$. Let $\eps>0$ be given by Lemma~\ref{lem:collision-duration} and define $s_n$ by, 
\EQ{
s_n := \inf\{ t \in [\sigma_n, \tau_n] \mid \bs \de( u(\tau); D(y_n, \rho_n)) \ge \eps, \quad \forall \tau \in [t, \tau_n]\}, 
}
which is well-defined for all sufficiently large $n$. 
Then $\bs \de( u(s_n); D(y_n, \rho_n)) = \eps$. Define $\lam_{\max}(s_n)$ as in the statement of the result. By Lemma~\ref{lem:collision-duration} it follows that $s_n +c_0 \lam_{\max}(s_n)^2 \le \tau_n$ for all sufficiently large $n$. The remaining claims hold by the choice of $s_n$.  
\end{proof} 

\subsection{Proof of Theorem~\ref{thm:main}}

\begin{proof}[Proof of Theorem~\ref{thm:main}] 
Assume the theorem is false. Let $K \ge 1$ and  fix collision intervals $[\sigma_n, \tau_n]\in \calC_K( y_n, \rho_n, \eps_n, \eta)$ as in Definition~\ref{def:K} and Lemma~\ref{lem:K}. 
%normalized as in~\eqref{eq:normal}. 
We assume that $\eta>0$ is sufficiently small as in Lemma~\ref{lem:collision-duration} and let $\eps>0$  and $s_n$ be given by Corollary~\ref{cor:sn}, so we have
\EQ{
\bs \de( u(s_n), D( y_n, \rho_n)) = \eps.
}
%Let $\be_n \to 0$ be any sequence and 
Let $M_n$ be a sequence of non-negative integers, $\calQ( \bs \om_n)$ a sequence of $M_n$-bubble configurations, and $\vec \nu_n \in (0, \infty)^{M_n+1}$, $\vec \xi_n \in (0, \infty)^{M_n+1}$  sequences  so that 
\EQ{ \label{eq:d-sn} 
\eps \le \bfd( u(s_n), \calQ(\bs \om_n); D(y_n, \rho_n); \vec \nu_n, \vec \xi_n) \le 2 \eps. 
}
We fix a choice of $\xi_n, \nu_n$ (the first components of the vectors $\vec \xi_n, \vec \nu_n$) as in Remark~\ref{rem:neck} so that~\eqref{eq:ci-xi} and ~\eqref{eq:ci-neck} hold. % holds on the entire interval $[\sigma_n, \tau_n]$. 
Defining $\lam_{\max, n} = \lam_{\max}(s_n)$ as in Corollary~\ref{cor:sn} we have that $[s_n, s_n + c_0 \lam_{\max, n}^2] \subset [\sigma_n, \tau_n]$ and moroever that 
\EQ{ \label{eq:delta-big} 
\bs \de( u(t); D(y_n, \rho_n)) \ge \eps ,\quad \forall t \in [s_n, s_n + c_0\lam_{\max, n}^2], 
}
for all $n$ sufficiently large. Since $\sup_{t < T_+}E(u(t)) < \infty$ we can, after passing to a subsequence, assume $M_n = M$  for some fixed integer $M$ and that the constant $\om_{n}  \in\Sp^2$ in the $M$-bubble configuration $\calQ(\bs \om_n)$ are fixed, i.e., $\om_n = \om \in\Sp^2$. 

We claim there exists $c_1>0$ such that for all $n \ge n_0$ 
\EQ{ \label{eq:lower1} 
\lambda_{\max, n}^2  \| \calT(u(t)) \|_{L^2}^2 \ge c_1, \quad \forall t \in [s_n, s_n + c_0 \lambda_{\max, n}^2]. 
}
 If not, we could find a sequence $t_n \in [s_n, s_n + c_0 \lambda_{\max,n }^2] \subset [ \sigma_n, \tau_n]$ such that
 \EQ{
 \lim_{n \to \infty} \lambda_{\max, n}  \| \calT(u(t_n)) \|_{L^2}  =  0 . 
 }
By the Compactness Lemma~\ref{lem:compact}, for all $x_n \in \R^2$  there exists a subsequence of the $u(t_n)$ and a sequence $R_n(x_n) \to \infty$, such that, 
for any sequence $1 \ll \breve{R}_n \ll  R_n(x_n)$, 
\EQ{ \label{eq:compact-lambda_max} 
\lim_{n \to \infty} \bs \de( u(t_n); D(x_n,  \breve{R}_n \lambda_{\max, n})) = 0. 
}
By Lemma~\ref{lem:Struwe} we also have that 
\EQ{ \label{eq:total_energy} 
\lim_{n \to \infty} E( u(t_n); D(y_n, \rho_n)) = 4 K \pi.  %\mand \lim_{n \to \infty} E( u(t_n); D(y_n, \ti C \rho_n) \setminus D(y_n, \ti C^{-1} \rho_n) )  = 0
} 
where here we have used that $|[\sigma_n, \tau_n]| \ll \rho_n^2$ to propagate Property (4) from Definition~\ref{def:K} from time $\sigma_n$ to time $t_n$. 
Note also that $\rho_n^2 \gg  \xi_n^2 \gg  \tau_n - \sigma_n$ and Corollary~\ref{cor:sn} ensure that $\xi_n \gg \lam_{\max, n}$. 
%for some $\ti C>1$. 

We claim that after passing to a subsequence, there exists an integer $L>0$, sequences $x_{\ell, n}$ for each $\ell \in \{1, \dots, L\}$, a number $R \ge 2$, and a sequence $1 \ll \ti R_n \ll \lam_{\max, n}^{-1} \xi_n$ so that  
\EQ{ \label{eq:R-out} 
E\Big(u(s_n); D( y_n, \rho_n) \setminus \bigcup_{\ell =1}^L D( x_{\ell, n}, R \lambda_{\max, n}\Big) \le \frac{\pi}{2}, 
}
%the discs $D( x_{\ell, n}, \ti R_n \lambda_{\max, n}) \subset D(y_n, \rho_n)$ for each $n$, 
and 
\EQ{ \label{eq:disjoint} 
D( x_{\ell, n}, \ti R_n \lambda_{\max, n}) \cap D( x_{\ell', n}, \ti R_n \lambda_{\max, n}) = \emptyset
}
for any $\ell \neq \ell'$. 
%Next, recall that at time $s_n$ we have, 
%\EQ{
%\eps \le \bfd( u(s_n), \calQ( \om_{1, n}, \dots, \om_{M, n}); D( y_n, \rho_n)) \le 2 \eps. 
%}
We find the points $x_{\ell, n}$ as follows. Passing to a subsequence we can assume the existence of the limits
\EQ{
\lim_{n \to \infty} \frac{|a(\om_{j, n}) - a(\om_{k, n})|}{\lam_{\max, n}} \in [0, \infty]. 
}
for each $j \neq k$. We define the index sets
\EQ{
\calL(j):= \Big\{ j \, \, \textrm{and any index} \, \, k \in \{1, \dots, M\}\, \, \textrm{such that} \lim_{n \to \infty} \frac{|a(\om_{j, n}) - a(\om_{k, n})|}{\lam_{\max, n}} < \infty \Big\}. 
}
and note that for any distinct indices $j, j'$ either $\calL(j) = \calL(j')$ or they are disjoint. For each $n$ and for each of the sets $\calL(j)$ we let $x_{\calL(j), n}$ denote the barycenter of the points $a(\om_{j_1, n}), \dots, a(\om_{j_{\#\calL(j)}, n})$ where each $j_k \in \calL(j)$. There are $L \le M$ many distinct index sets $\calL(j)$ and we let $\{x_{\ell, n}\}_{\ell =1}^L$ be an enumeration of the distinct $x_{\calL(j), n}$.  

Next, from~\eqref{eq:d-sn}, Lemma~\ref{lem:decay} and the definitions of $\bfd$ and $\lambda_{\max, n}$ we can find  $R_1\ge 2$ so that
\EQ{
E\Big( u(s_n); D( y_n, \rho_n) \setminus \bigcup_{j =1}^M D( a( \om_{j, n}),  R_1 \lam_{\max, n})\Big) \le \frac{\pi}{2}.
}
for all sufficiently large $n$. From the above and the definition of the $x_{\ell, n}$ we can find $R\ge R_1$ so that~\eqref{eq:R-out} holds. 
%To prove the remaining properties of the points $x_{\ell, n}$, we first note that $\rho_n^2 \gg  \tau_n - \sigma_n$ and Corollary~\ref{cor:sn} ensure that $\rho_n \gg \lam_{\max, n}$. 
%The existence of a sequence $\ti R_n \to \infty$ so that $D( x_{\ell, n}, \ti R_n \lam_{\max, n}) \subset D( y_n, \rho_n)$ and such 
The existence of a sequence $1 \ll \ti R_n \ll \lam_{\max, n}^{-1} \xi_n$ so that~\eqref{eq:disjoint} holds   follows from definition of the $x_{\ell, n}$. 

Consider each of the sequences $x_{\ell, n}$ as the $x_n$ in~\eqref{eq:compact-lambda_max} and find corresponding sequences $R_{\ell, n}$ so that for any sequence $\breve R_n \le R_{\ell, n}$, 
\EQ{\label{eq:delta-xn} 
\lim_{n \to \infty} \bs \de( u(t_n); D( x_{\ell, n}, \breve{R}_n \lambda_{\max, n})) = 0, \quad \ell = 1, \dots, L.
}
%Fix a sequence $R_n \le \min\{ \ti R_n, R_{\ell, n}\}_{\ell = 1, \dots, L}\}$ and 

Enlarge the sequence $\xi_n$ to a sequence $\ti \xi_n$ as in Remark~\eqref{rem:neck}, i.e. so that $\xi_n \ll \ti \xi_n \ll \rho_n$. Then, since all of the $x_{\ell,n} \in D( y_n, \xi_n)$ and $\lam_{\max, n} \ll \xi_n$, we have
\EQ{ \label{eq:wayin} 
\lim_{n \to \infty} \frac{ \lam_{\max, n}}{\dist( x_{\ell, n}, \partial D(y_n,  \ti \xi_n))}  = 0, 
}
for each $\ell$. We can thus find a sequence $R_n \le \min\{ \ti R_n, R_{\ell, n}\}_{\ell = 1, \dots, L}$ such that $D(x_{\ell, n}, R_n \lambda_{\max, n}) \subset D(y_n, \ti \xi_n)$ for each $\ell$.

Enlarging the excised discs (replacing $R$ by $R_n$) in~\eqref{eq:R-out} we obtain
\EQ{
E\Big(u(s_n); D( y_n, \rho_n) \setminus \bigcup_{\ell =1}^L D( x_{\ell, n}, R_n \lambda_{\max, n}\Big) \le \frac{\pi}{2}. 
}
We use Lemma~\ref{lem:Struwe} to propagate this bound forward to time $t_n$, giving %and using~\eqref{eq:calK} to drop the indices $\ell_0 \in \{ 1, \dots, L\} \setminus \calL$ we have, 
\EQ{ \label{eq:lepi} 
E\Big(u(t_n); D( y_n, \rho_n) \setminus \bigcup_{\ell =1}^L D( x_{\ell, n}, R_n \lambda_{\max, n}\Big) \le \pi. 
}
%passing to a subsequence we may assume that $\lim_{n \to \infty} \ti R_n  \in [1, \infty]$. 
%and hence, 
%\EQ{
%E( u(s_n); D( y_n, \rho_n) \setminus \bigcup_{j =1}^M D( a( \om_{j, n}), \ti R_n \lambda_{\max}(s_n))) \le \frac{\pi}{2} . 
%}%We use Lemma~\ref{lem:Struwe} to propagate this bound forward to time $t_n$, obtaining
%\EQ{ \label{eq:pi/2} 
%E( u(t_n); D( y_n, \rho_n) \setminus \bigcup_{j =1}^M D( a( \om_{j, n}), 2R \lambda_{\max}(s_n))) \le \pi.
%} 
On the other hand, by~\eqref{eq:delta-xn} (replacing $R_{1, n}$ by $R_n$), we can find integers $K_{\ell}$ so that, 
\EQ{
E( u(t_n); D(x_{\ell, n}, R_{ n} \lambda_{\max, n})) \to 4 K_{\ell} \pi \mas n \to \infty, 
}
for each $\ell \in \{1, \dots, L\}$. Combining the above with~\eqref{eq:total_energy} we see that, 
\EQ{
\lim_{n \to \infty} E\Big( u(t_n); D(y_n, \rho_n) \setminus \bigcup_{\ell =1}^L D(x_{\ell, n}, R_{ n} \lambda_{\max, n})\Big) = 4 K \pi - \sum_{\ell=1}^L 4 K_{\ell} \pi. 
}
Comparing the above with~\eqref{eq:lepi} it follows that $\sum_{\ell} 4 K_{\ell} \pi = 4 K \pi$ and thus, 
\EQ{ \label{eq:no-en-out} 
\lim_{n \to \infty} E\Big( u(t_n); D(y_n, \rho_n) \setminus \bigcup_{\ell =1}^L D(x_{\ell, n}, R_{ n} \lambda_{\max, n})\Big) = 0. 
}
From~\eqref{eq:delta-xn} and the definition of $R_n$ we have 
\EQ{ \label{eq:del-xn} 
\lim_{n \to \infty} \sum_{\ell =1}^L  \bs \de( u(t_n); D( x_{\ell, n}, R_n \lambda_{\max, n})) = 0. 
}
and moreover that the discs $D( x_{\ell, n}, R_n \lambda_{\max, n})$ are disjoint by~\eqref{eq:disjoint} and the choice of $R_n \le \ti R_n$. Combining~\eqref{eq:del-xn},~\eqref{eq:wayin}, the disjointness of the discs $D( x_{\ell, n}, R_n \lambda_{\max, n})$, ~\eqref{eq:no-en-out}, and Remark~\ref{rem:neck}, we conclude that 
\EQ{
\lim_{n \to \infty} \bs \de( u(t_n); D( y_n, \rho_n)) = 0.
}
which contradicts~\eqref{eq:delta-big}, proving~\eqref{eq:lower1}.  
%\Red{[consider giving more details here, i.e., maybe ensure we only focus on a disjoint collection of discs]}

By~\eqref{eq:lower1} we have 
\EQ{\nonumber
\sum_n \int_{s_n}^{s_n + c_0 \lambda_{\max}(s_n)^2}  \| \calT(u(t)) \|_{L^2}^2 \, \ud t &\ge c_1  \sum_n \int_{s_n}^{s_n + c_0 \lambda_{\max}(s_n)^2} \lambda_{\max}(s_n)^{-2} \, \ud t  \ge c_0c_1 \sum_{n}1 = \infty.
}
On the other hand, since the intervals $[\sigma_n, \tau_n]$ are disjoint, the above contradicts the bound~\eqref{eq:tension-L2}, i.e., 
\EQ{
\sum_n \int_{s_n}^{s_n + c_0 \lambda_{\max}(s_n)^2}  \| \calT(u(t)) \|_{L^2}^2 \, \ud t \le \int_0^{T_+}  \| \calT(u(t) \|_{L^2}^2 \, \ud t < \infty, 
}
which completes the proof. 
\end{proof} 

\subsection{Proof of Theorem~\ref{thm:main1}}
In this section we prove Theorem~\ref{thm:main1}, using Theorem~\ref{thm:main} as a main ingredient in the proof.

\begin{proof}[Proof of Theorem~\ref{thm:main1}]  We consider the case of finite time blow up, i.e., $T_+ < \infty$, noting that the analysis for the global case is similar.

%Consider first the case of finite time blow up, i.e., $T_+ < \infty$. 
Let $L \ge 1$ and $\{x_\ell\}_{\ell =1}^L$ be the bubbling points given by the local theory of Struwe in Theorem~\ref{lem:lwp}. Let $\rho_0>0$ be sufficiently small so that $D(x_\ell; 2\rho_0) \cap D(x_m; 2\rho_0) = \emptyset$ for each $\ell \neq m$. By Theorem~\ref{lem:lwp} we have that
\EQ{
\lim_{ t\to T_+} E\Big( u(t) - u^*; \R^2 \setminus \bigcup_{\ell =1}^LD(x_\ell; \rho_0) \Big)  = 0. 
}
By Lemma~\ref{lem:ss-bu} we know that for each $\ell$
\EQ{
\lim_{ t\to T_+} E\big( u(t) - u^*; D( x_\ell; \rho_0) \setminus D( x_\ell; \sqrt{T_+- t}) \big) = 0,
}
and since $u^* \in \E$, 
\EQ{
\lim_{ t\to T_+} E\big( u^*;  D( x_\ell; \sqrt{T_+- t}) \big) = 0
}
for each $\ell \in \{1, \dots, L\}$. 
Hence it suffices to examine the solution $u(t)$ in the discs $D( x_\ell; \sqrt{T_+- t})$ for each $\ell \in \{1, \dots, L\}$. Fix an $\ell$ and, to ease notation, we write $y = x_\ell$ below. By Theorem~\ref{thm:main} we know that for $\rho(t):= \sqrt{T_+-t}$ we have, 
\EQ{
\lim_{t \to T_+} \bs \de( u(t); D( y, \rho(t)) ) = 0. 
}
%At this point we have proved the estimate~\eqref{eq:ss-neck}.

Now, let $t_n \to T_+$ be any sequence of times. By the above we can find a sequence $1 \le M_n \le  (4 \pi)^{-1} E( u_0)$, a sequence of $M_n$-bubble configurations $\calQ(\bs \om_n)$, %harmonic maps, $\bs \omega_n = (\om_{1, n}, \dots, \om_{M_n, n})$, constants $\om_{\infty, n} \in \Sp^2$, 
and sequences $\vec \nu_{n} = (\nu_n, \nu_{1, n}, \dots, \nu_{M_n, n})$, $\vec \xi_{n} = ( \xi_n, \xi_{1, n}, \dots, \xi_{M_n, n})$ such that 
\EQ{ \label{eq:u(t_n)} 
\lim_{n \to \infty} \bfd( u(t_n), \calQ( \bs \om_{ n}); D(y, \rho(t_n)); \vec \nu_n, \vec \xi_n)  = 0.
}
Passing to a subsequence of the $t_n$ we may assume that $M_n = M$ is a fixed integer and that the constants $\om_{ n}\in \Sp^2$ in the $M$-bubble configurations $\calQ(\bs \om_n)$ are fixed, i.e., $\om_n = \om \in \Sp^2$. This proves the estimate~\eqref{eq:ss-neck}. 
 % with $\xi_n := \xi_{0, n}$. 

Since each of the $\om_{j, n}$ is a harmonic map (and thus $\calT(\om_{j,n}) = 0)$ these sequences satisfy the hypothesis of the Compactness Lemma~\ref{lem:compact}.  Therefore, after passing to a joint subsequence, for each $j \in \{1, \dots, M\}$ we can find integers $M_j \ge 0$,  harmonic maps $\te_{j, 0}, \te_{ j, 1}, \dots \te_{j, M_j}$ (where only $\te_{j, 0}$ is possibly constant),  along with sequences of vectors $b_{j, k, n} \in D(a(\om_{ j, n}), C_j \lam(\om_{ j, n}))$ and scales $\mu_{ j, k, n} \ll \lam(\om_{ j, n})$ satisfying~\eqref{eq:w22-body},~\eqref{eq:w22-bubbles} and so that, 
\EQ{ \label{eq:bubbles(t_n)} 
\lim_{n \to \infty} &\Bigg[ E \Big( \om_{j, n}- \te_{j, 0}\big( \frac{ \cdot - a( \om_{ j, })}{\lam(\om_{  j, n})}\big) - \sum_{k =1}^{M_j}\Big( \te_{j, k} \big( \frac{\cdot - b_{ j, k, n}}{\mu_{ j, k, n}} \big) - \te_{j, k}( \infty)\Big); D_{j,n} \Big) \\
& + \Big\| \om_{j, n}- \te_{j, 0}\big( \frac{ \cdot - a( \om_{ j, })}{\lam(\om_{  j, n})}\big) - \sum_{k =1}^{M_j}\Big( \te_{j, k} \big( \frac{\cdot - b_{ j, k, n}}{\mu_{ j, k, n}} \big) - \te_{j, k}( \infty)\Big)\Big\|_{L^\infty( D_{j,n})}  \Bigg] = 0, 
}
where $D_{j,n}:=D( a(\om_{j, n}), R_{n} \lam(\om_{ j, n}))$ 
for some sequence $R_n \to \infty$, and where for each  fixed~$j$, 
\EQ{ \label{eq:fixedj} 
\lim_{n \to \infty} \sum_{k \neq k'}  \Big( \frac{ \mu_{ j, k, n}}{\mu_{j, k', n}} + \frac{ \mu_{ j, k', n}}{\mu_{ j, k, n}} + \frac{| b_{ j, k, n} - b_{ j, k', n} |^{2}}{\mu_{ j, k, n} \mu_{j, k', n}} \Big)^{-1} = 0. 
}
To make the notation for the scales and centers of the harmonic maps above more uniform, we also introduce the notation 
\EQ{
\mu_{ j, 0, n}:= \lam(\om_{ j, n}), \quad b_{ j, 0, n} := a(\om_{j, n}).
}
%We will say that two triples $(\om_j, a_{j, n}, \lam_{j, n})$ and $(\om_{j'}, a_{n, j'}, \lam_{n,j'})$ where $\om_j, \om_{j'}$ are nontrivial harmonic maps, $a_{n, j}, a_{n, j'} \in \R^2$ are sequences of vectors in $\R^2$, and $\lam_{n, j}, \lam_{n, j'} \in (0, \infty)$ are sequences of scales, are \emph{asymptotically orthogonal}  if 
%\EQ{
%\lim_{n \to \infty} \Big( \frac{\lam_{n, j}}{\lam_{n, j'}} + \frac{\lam_{n, j'}}{\lam_{n, j}} + \frac{\abs{ a_{j, n}- a_{n, j'}}^2}{\lam_{j, n}\lam_{n, j'}} \Big) = \infty. 
%}
Our goal is to find a collection of asymptotically orthogonal triples, $(\om_j, a_{j, n}, \lam_{j, n})$ as in the statement of Theorem~\ref{thm:main1}. The  sequencies  $\{(\te_{j, k}, b_{ j, k, n}, \mu_{ j, k, n})\}_{j =1, k =0}^{j = M, k = M_j}$ are not guaranteed to be such a 
collection. While~\eqref{eq:fixedj} holds for each fixed $j$, the triples $(\te_{j, k}, b_{j, k, n}, \mu_{ j, k, n})$ and $(\te_{j', k'}, b_{ j', k', n}, \mu_{ j', k', n})$ with $j \neq j'$ might not be asymptotically orthogonal. 
As in Definition~\ref{def:tree} and the proof of Lemma~\ref{lem:collision-duration}, we define the set of indices $\calR$ to be those associated to the roots, i.e., the maximal elements of the sequences $\frakh_j=\{\omega_{j,n}\}$ of harmonic maps under the partial order~$\preceq$.  For each root $\frakh_{j_0}$ we define the bubble tree 
\EQ{\nonumber
\calT(j_0) := \{\frakh_j\preceq \frakh_{j_0}\}.
} 
%Consider an index $j_0 \in \{1, \dots, M\}$ for which the associated sequence of harmonic maps $\om_{n, j_0}$ is not the descendant of any other sequence of harmonic maps. 
 Let  $C_0>0$ be large enough so that $\frakh_j\prec\frakh_{j_0}$ implies 
$
D(a(\om_{j, n}), \lam(\om_{ j, n})) \subset D(a(\om_{ j_0, n}),  C_0\lam(\om_{ j_0, n}))
$
for all~$n$. 
The collection of all harmonic maps, together  with scales and centers,  concentrating inside the discs $D(a(\om_{ j_0,n}),  C_0\lam(\om_{ j_0,n}))$ equals  
\EQ{  \label{eq:calJ0} 
%\Big\{(\om_{j_0, \infty}, b_{n, j_0, 0}, \mu_{n, j_0, 0}), \{ (\te_{j_0, k},  b_{n, j_0, k},  \mu_{n, j_0, k})\}_{k =1}^{M_{j_0}}\Big\} \\\bigcup_{j \in \calJ(j_0)} \Big\{(\om_{ j, \infty}, b_{n, j, 0}, \mu_{n, j, 0}), 
\bigcup_{\frakh_j \in \calT(j_0)}\{ (\te_{j, k},  b_{ j, k, n},  \mu_{ j, k, n})\}_{k =0}^{M_{j}}
}
%We call triples  $(\te_{j, k}, b_{n, j, k}, \mu_{n, j, k})$ and $(\te_{j', k'}, b_{n, j', k'}, \mu_{n, j', k'})$ asymptotically orthogonal if, 
%\EQ{
%\frac{ \mu_{n, j, k}}{\mu_{n, j', k'}} + \frac{ \mu_{n, j', k'}}{\mu_{n, j, k}} + \frac{| b_{n, j, k} - b_{n, j', k'} |^{2}}{\mu_{n, j, k} \mu_{n, j', k'}}  \to \infty \mas n \to \infty.
%}
We let 
\EQ{
\calK(j, k):= \big\{ (j, k)& \, \, \textrm{and any} \, \, (j', k') \, \, \textrm{associated to a triple} \, \, (\om_{j', k'}, b_{j', k', n}, \mu_{ j', k', n}) \\
 &\textrm{not asymptotically orthogonal to} \, \, (\om_{j, k}, b_{ j, k, n}, \mu_{ j, k, n})\big\}
}
If $\#\calK(j, k) = 1$ we keep the triple $ (\te_{j, k},  b_{ j, k, n},  \mu_{ j, k, n})$ in our final collection -- note that~\eqref{eq:sup-blowup} and~\eqref{eq:sup-global} will be consequences of~\eqref{eq:u(t_n)}~\eqref{eq:bubbles(t_n)}, and~\eqref{eq:w22-body}~\eqref{eq:w22-bubbles}.  Now consider a set of indices $(j_1, k_1)$ with $\frakh_{j_1} \in\calT(j_0)$ and such that $\#\calK(j_1, k_1) \ge 2$. 
%This means we can find $b_n \in D( a(\om_{n, j_0}), C_0 \lam(\om_{n, j_0}))$ and $\mu_n \ll \lam(\om_{n, j_0})$ so that 
After performing a fixed (in $n$) rescaling and  translation of each harmonic map $\te_{j, k}$ associated to an index $(j, k) \in \calK(j_1, k_1)$ we may assume that, 
\EQ{
b_{ j, k, n} = b_{ j_1, k_1, n} \mand \mu_{ j, k, n}  = \mu_{j_1, k_1, n},\quad \forall (j, k) \in \calK(j_1, k_1)
}
and to simplify notation below we simply write $b_n = b_{ j_1, k_1, n}$ and $\mu_n = \mu_{ j_1, k_1, n}$. 
By~\eqref{eq:u(t_n)} and~\eqref{eq:bubbles(t_n)} we can also find $r_n \to \infty$ a number $C_1>0$, an integer $L_1 \ge 0$, and a finite number of sequences  of discs $D(c_{ \ell, n}, \rho_{n, \ell})  \subset D( b_n, C_1 \mu_n)$ for $\ell \in \{1, \dots, L_1\}$ and with 
\EQ{ \label{eq:small-holes} 
\frac{\rho_{ \ell, n}}{ \dist( c_{ \ell, n},  \partial D(b_{n},  C_1\mu_n))} \to 0 \mas n \to \infty
}
so that, 
\EQ{ \label{eq:small-necks} 
\lim_{n \to \infty}E( u(t_n); D( c_{\ell, n}, 2 \rho_{\ell, n}) \setminus  D( c_{\ell, n}, \frac{1}{2}  \rho_{\ell, n}) ) = 0,  
}
and, 
\EQ{\label{eq:no-neck1} 
\lim_{n \to \infty}E( u_n; D( b_n, 2r_n\mu_n) \setminus D( b_n, \frac{1}{2} r_n\mu_n))  = 0, 
}
and, 
\EQ{ \label{eq:en-u-b-disc} 
\!\!\!\!\!\!\lim_{n \to \infty} E\Big( u(t_n) - \!\!\!\!\!\!\sum_{(j, k) \in \calK(j_1, k_1)}\!\!\!\Big( \theta_{j, k}( \frac{\cdot - b_n}{ \mu_n}) - \theta_{j, k}( \infty) \Big) ; D( b_n, r_n\mu_n)  \setminus \bigcup_{\ell =1}^{L_1} D(c_{n, \ell}, \rho_{n, \ell})\Big) = 0. 
}
By Theorem~\ref{thm:main} and~\eqref{eq:no-neck1} we know that, 
\EQ{ \label{eq:delta-b-disc} 
\lim_{n \to \infty} \bs \de( u(t_n); D( b_n, r_n \mu_n))  = 0.
}
This means that, after passing to a subsequence, we can find an integer $M_{j_1, k_1}$, a sequence of $M_{j_1, k_1}$-bubble configurations $\calQ(\bs \Om_n)$, with the non-trivial harmonic maps denoted by $\Om_{m, n}$ for $m \in \{1, \dots, M_{j_1, k_1}\}$, 
%$\Om_{\infty, n} \in \Sp^2$, sequences of harmonic maps $\bs \Omega_n = (\Om_{ 1, n}, \dots, \Om_{ M_{j_1, k_1}, n})$,
 sequences $\vec{ \ti \nu}_n = (\ti \nu_n, \ti \nu_{1, n}, \dots, \ti \nu_{M_{j_1, k_1}, n})$ and  $\vec{ \ti \xi}_n = (\ti \xi_n,  \ti \xi_{1, n}, \dots, \ti \xi_{M_{j_1, k_1}, n})$, so that 
\EQ{\label{eq:d-b-disc} 
\bfd( u(t_n), \calQ(\Om_{\infty, n}, \bs \Om_{n}); D( b_n, r_n \mu_n); \vec{\ti \nu}_n, \vec{\ti \xi}_n)  \to 0 \mas n \to \infty. 
}
Consider the centers and scales $a(\Om_{m,n}), \lam(\Om_{m,n})$ associated to the harmonic maps  $\Om_{m, n}$. Using~\eqref{eq:small-holes},~\eqref{eq:small-necks}~\eqref{eq:no-neck1}, and~\eqref{eq:en-u-b-disc} we see that there are only two possible cases. 

\textbf{Case 1:} All of the harmonic maps $\Om_{n, m}$ concentrate within the discs $\bigcup_{ \ell = 1}^{L_1}D( c_{ \ell, n}, \rho_{ \ell, n})$, i.e.,  for each $m \in \{1, \dots, M_{j_1, k_1}\}$, we have,  $D( a( \Om_{ m, n}), \lam(\Om_{ m, n})) \subset D( c_{\ell, n}, \rho_{ \ell, n})$ for some $\ell \in \{1, \dots, L_1\}$. This means that, 
\EQ{
\lim_{n \to \infty} E\Big( u(t_n); D( b_n, r_n\mu_n)  \setminus \bigcup_{\ell =1}^{L_1} D(c_{\ell, n}, \rho_{ \ell, n})\Big) = 0. 
}
In fact, comparing the above with~\eqref{eq:en-u-b-disc} one can deduce that the harmonic maps $$\sum_{(j, k) \in \calK(j_1, k_1)} ( \theta_{j, k}(x) - \theta_{j, k}( \infty) ) = \textrm{constant} \in \R^3.$$ In this case we discard all of the harmonic maps with indices $(j, k) \in \calK(j_1, k_1)$ from the final collection. 

\textbf{Case 2:} Exactly one of the harmonic maps $\Om_{m_1, n}$ has scale $\lam( \Om_{m_1, n}) \simeq  \mu_n$ and center $|a(\Om_{ m_1, n}) - b_n| \lesssim  \mu_n$, and the rest concentrate within the discs $\bigcup_{ \ell = 1}^{L_1}D( c_{ \ell, n}, \rho_{\ell, n})$. We then have, 
\EQ{ \label{eq:Omeganm}
\lim_{n \to \infty} E\Big( u(t_n) - \Om_{ m_1, n}; D( b_n, r_n\mu_n)  \setminus \bigcup_{\ell =1}^{L_1} D(c_{ \ell, n}, \rho_{ \ell, n})\Big) = 0
}
and
\EQ{
\lim_{n \to \infty} \Big\| u(t_n) - \Om_{ m_1, n} \Big\|_{L^\infty( D( b_n, r_n\mu_n)  \setminus \bigcup_{\ell =1}^{L_1} D(c_{ \ell, n}, \rho_{ \ell, n}))} = 0.
}
By another application of the Compactness Lemma~\ref{lem:compact}, we can find a non-trivial harmonic map, which we label  $\Theta_{ j_1, k_1}$, a non-negative integer $P$, scales $\nu_{ p, n} \ll \mu_n$,  centers $d_{ p, n}$,  and non-trivial harmonic maps $\Theta_{p}$, satisfying~\eqref{eq:w22-body} and~\eqref{eq:w22-bubbles}, and  so that
\EQ{
\lim_{n \to \infty} &\Bigg[E\Big( \Om_{ m_1, n} -  \Theta_{ j_1, k_1}\big( \frac{\cdot - b_n}{\mu_n}\big) - \sum_{p=1}^{P} \big(\Theta_p\big(\frac{ \cdot - d_{ p, n}}{\nu_{ p, n}} \big) - \Theta_p(\infty)\big);  D( b_n, r_n\mu_n)\Big)  \\
& + \Big\| \Om_{ m_1, n} -  \Theta_{ j_1, k_1}\big( \frac{\cdot - b_n}{\mu_n}\big) - \sum_{p=1}^{P} \big(\Theta_p\big(\frac{ \cdot - d_{ p, n}}{\nu_{ p, n}} \big) - \Theta_p(\infty)\big)\Big\|_{L^\infty(  D( b_n, r_n\mu_n)) } \Bigg]= 0
}
We know that the harmonic map $\Theta_{j_1, k_1}$  must be nontrivial because of~\eqref{eq:Omeganm} together with
\eqref{eq:en-u-b-disc}, 
where the latter ensures that energy cannot concentrate within the region $$D(b_n, r_n \mu_n) \setminus  \bigcup_{\ell =1}^{L_1} D(c_{\ell, n}, \rho_{ \ell, n})$$
at scales smaller than  $\mu_n$.  Indeed,
by~\eqref{eq:en-u-b-disc} the scales and centers of the non-trivial harmonic maps $\Theta_p$ must all concentrate within the discs $\bigcup_{\ell =1}^{L_1}D(c_{ \ell, n}, \rho_{\ell, n})$ and we can conclude that
\EQ{
\lim_{n \to \infty} E\Big( u(t_n) -  \Theta_{ j_1, k_1}( \frac{\cdot - b_n}{\mu_n}); D( b_n, r_n\mu_n)  \setminus \bigcup_{\ell =1}^{L_1} D(c_{ \ell, n}, \rho_{ \ell, n})\Big) = 0,
}
and
\EQ{
\lim_{n \to \infty} \Big\| u(t_n) -  \Theta_{ j_1, k_1}( \frac{\cdot - b_n}{\mu_n})\Big\|_{L^\infty( D( b_n, r_n\mu_n)  \setminus \bigcup_{\ell =1}^{L_1} D(c_{ \ell, n}, \rho_{ \ell, n}))} = 0
}
In this case, we discard all the triples $(\te_{j, k}, b_{ j, k, n}, \mu_{ j, k, n})$ with indices $(j, k) \in \calK(j_1, k_1)$ from the final collection, and replace them with the triple  $(\Theta_{ j_1, k_1}, b_{j_1, k_1, n}, \mu_{ j_1, k_1, n})$. 

To summarize, we keep for the final decomposition any triples $(\theta_{j, k}, b_{ j, k, n}, \mu_{j, k, n})$ with $\frakh_j \in \calT(j_0)$ if $\#\calK(j, k) = 1$. If $\#\calK(j, k) >1$  we discard all of the triples $( \theta_{j', k'}, b_{ j', k', n}, \mu_{ j', k', n})$ with indices $j' \in \calK(j, k)$  and, in the event of Case $2$ above, replace them with $ \Theta_{j, k}, b_{ j, k, n}, \mu_{ j, k,n}$. We perform this analysis for each index $j_0 \in \calR$, resulting in a final collection of triples that are mutually asymptotically orthogonal  and satisfy
the conclusions of the theorem.
\end{proof}

\bibliographystyle{plain}
\bibliography{HMHF}

\newpage
\bigskip
\centerline{\scshape Jacek Jendrej}
\smallskip
{\footnotesize
 \centerline{CNRS and LAGA, Universit\'e  Sorbonne Paris Nord}
\centerline{99 av Jean-Baptiste Cl\'ement, 93430 Villetaneuse, France}
\centerline{\email{jendrej@math.univ-paris13.fr}}
} 
\medskip 
\centerline{\scshape Andrew Lawrie}
\smallskip
{\footnotesize
 \centerline{Department of Mathematics, Massachusetts Institute of Technology}
\centerline{77 Massachusetts Ave, 2-267, Cambridge, MA 02139, U.S.A.}
\centerline{\email{alawrie@mit.edu}}
} 
\medskip 
\centerline{\scshape Wilhelm Schlag}
\smallskip
{\footnotesize
 \centerline{Department of Mathematics, Yale University}
\centerline{10 Hillhouse Ave, New Haven, CT 06511, U.S.A.}
\centerline{\email{wilhelm.schlag@yale.edu}}
}

%\bigskip
%\centerline{\scshape Jacek Jendrej}
%\smallskip
%{\footnotesize
% \centerline{CNRS and LAGA, Universit\'e  Sorbonne Paris Nord}
%\centerline{99 av Jean-Baptiste Cl\'ement, 93430 Villetaneuse, France}
%\centerline{\email{jendrej@math.univ-paris13.fr}}
%} 
%\medskip 
%\centerline{\scshape Andrew Lawrie}
%\smallskip
%{\footnotesize
% \centerline{Department of Mathematics, Massachusetts Institute of Technology}
%\centerline{77 Massachusetts Ave, 2-267, Cambridge, MA 02139, U.S.A.}
%\centerline{\email{alawrie@mit.edu}}
%} 

\end{document}